%% file: thesis.tex
\def\draft{n}

\documentclass[oneside,12pt]{report}   %
\usepackage{graphicx}  %
\usepackage[letterpaper, left=1.5in, right=1in, top=1in, bottom=1in]{geometry}
\usepackage{setspace}  %
\usepackage{times}  %
\usepackage[explicit]{titlesec}  %
\usepackage[titles]{tocloft}  %
\usepackage[backend=biber,url=true,style=alphabetic]{biblatex}
\usepackage[bookmarks=true, hidelinks, pdfencoding=auto, psdextra]{hyperref}
\usepackage[page]{appendix}  %
\usepackage{rotating}  %
\usepackage[normalem]{ulem}  %
\usepackage{tikz}
\usetikzlibrary{cd}

\usepackage{amsmath} %
\usepackage{amsfonts} %
\usepackage{amssymb} %
\usepackage[shortlabels]{enumitem} %
\usepackage{mathrsfs} %
\usepackage{color} %
\usepackage{amsthm} %

\def\printname#1{
        \if\draft y
                \smash{\makebox[0pt]{\hspace{-0.5in}
                        \raisebox{8pt}{\tt\tiny #1}}}
        \fi
}
\def\lbl#1{\label{#1}\printname{#1}}

\def\nc{\newcommand}

   \nc\FI[2]{\begin{figure}
    \begin{center}\input{#1.pstex_t}\end{center}
    \caption{#2}
    \lbl{#1}
  \end{figure}}
\nc\FIG[3]{\begin{figure}
    \includegraphics[#3]{#1.eps}
    \caption{#2}
    \lbl{fig:#1}
    \end{figure}}
\nc\FF[3]{\begin{figure}
    \includegraphics[#3]{#1.eps}
    \caption{#2}
    \lbl{#1}
    \end{figure}}

    \nc\FIGc[3]{\begin{figure}[htpb]
    \centering
    \includegraphics[height=#3]{#1.eps}
    \caption{#2}
    \lbl{fig:#1}
    \end{figure}}

    \nc\FIGh[3]{\begin{figure}[htpb]
    \includegraphics[height=#3]{#1.eps}
    \caption{#2}
    \lbl{fig:#1}
    \end{figure}}

\theoremstyle{plain}
\newtheorem{theorem}{Theorem}[section]
\newtheorem{lemma}[theorem]{Lemma}
\newtheorem{corollary}[theorem]{Corollary}
\newtheorem{proposition}[theorem]{Proposition}

\theoremstyle{definition}

\newtheorem{remark}[theorem]{Remark}

\newcommand{\red}[1]{{\color{red}#1}}

\newcommand\no[1]{}

\newcommand{\la}{\langle}
\newcommand{\ra}{\rangle}

\newcommand{\RR}{\mathbb{R}}
\newcommand{\CC}{\mathbb{C}}

\newcommand{\ZZ}{\mathbb{Z}}
\newcommand{\BR}{\mathbb{R}}
\newcommand{\BC}{\mathbb{C}}
\newcommand{\BQ}{\mathbb{Q}}
\newcommand{\BN}{\mathbb{N}}
\newcommand{\BZ}{\mathbb{Z}}
\newcommand{\BT}{\mathbb{T}}
\newcommand{\BA}{\mathbb{A}}

\newcommand\bk{\mathbf k}

\newcommand\bl{\mathbf l}
\newcommand\bm{\mathbf m}
\newcommand\bn{\mathbf n}
\newcommand\bv{\mathbf v}
\newcommand\by{\mathbf y}
\newcommand\bu{\mathbf u}
\newcommand\bx{\mathbf x}
\newcommand\bw{\mathbf w}

\newcommand\gr{\text{gr}}

\newcommand{\ve}{\varepsilon}
\newcommand{\vp}{\varphi}

\newcommand{\D}{\Delta}

\newcommand{\La}{\Lambda}
\newcommand{\uLa}{\underline{\Lambda}}
\newcommand{\al}{\alpha}

\newcommand{\heta}{\hat \eta}
\newcommand{\tsur}{\tau^{\text{sur}}}
\newcommand{\intS}{\mathring{\Sigma}}
\newcommand{\circD}{\mathring{\D}}

\newcommand{\cS}{\mathscr{S}}
\newcommand{\tcS}{\widetilde{\cS}}
\newcommand{\cP}{\mathcal{P}}

\newcommand{\tA}{\widetilde{A}}

\newcommand{\tcA}{\widetilde{\mathcal{A}}}

\newcommand{\tZ}{\widetilde{Z}}
\newcommand{\tcZ}{\widetilde{\mathcal{Z}}}
\newcommand{\tBT}{\widetilde{\BT}}

\newcommand{\cA}{\mathcal{A}}
\newcommand{\cB}{\mathcal{B}}
\newcommand{\cC}{\mathcal{C}}
\newcommand{\cD}{\mathcal{D}}
\newcommand{\cN}{\mathcal{N}}
\newcommand{\cR}{\mathcal{R}}

\newcommand{\cF}{\mathcal{F}}
\newcommand{\cH}{\mathcal{H}}
\newcommand{\cT}{\mathcal{T}}

\newcommand{\cQ}{\mathcal{Q}}
\newcommand{\cE}{\mathcal{E}}
\newcommand{\cU}{\mathcal{U}}
\newcommand{\cZ}{\mathcal{Z}}
\newcommand{\cV}{\mathcal{V}}

\newcommand{\cY}{\mathcal{Y}}

\newcommand{\fS}{\mathfrak{S}} %
\newcommand{\ofS}{\overline{\fS}}  %
\newcommand{\fY}{\mathfrak{Y}}
\newcommand{\fX}{\mathfrak{X}}
\newcommand{\sX}{\mathfrak{X}}
\newcommand{\fN}{\mathfrak{N}}
\newcommand{\fM}{\mathfrak{M}}
\newcommand{\fU}{\mathfrak{U}}
\newcommand{\fm}{\mathfrak m}
\newcommand{\fn}{\mathfrak n}
\newcommand{\fT}{\mathfrak T}

\newcommand{\tiX}{\widetilde \fX}

\newcommand{\XD}{\fX(\D)}
\newcommand{\bfN}{\widetilde{\fN}}
\newcommand{\Om}{\Omega}

\newcommand{\vpD}{\vp_\D}
\newcommand{\tvD}{\tilde \varphi_\D}
\newcommand{\tvpD}{\tilde {\vp}_\D}

\newcommand{\SP}{{(\Sigma,\cP)}} %
\newcommand{\SPp}{{(\Sigma,\cP')}}
\newcommand{\SpP}{{(\Sigma',\cP)}}
\newcommand{\Sx}{\cS_\xi}
\newcommand{\Se}{\cS_\ve}
\newcommand{\ofSP}{(\ofS,\cV)} %

\newcommand{\tStP}{(\widetilde{\Sigma},\widetilde{\cP})}
\newcommand{\cSigma}{\mathring{\Sigma}}
\newcommand{\cHd}{\mathcal{H}_\bullet}

\newcommand{\MN}{(M,\cN)}
\newcommand{\pM}{\partial M}
\newcommand{\pS}{\partial \Sigma}

\newcommand\LMN{\cT\MN}
\newcommand\hPhi{\hat \Phi}

\newcommand{\qbinom}[2]{\left[\begin{matrix}
#1 \\ #2
\end{matrix}  \right]}

\newcommand\be{\begin{equation}}
\newcommand\ee{\end{equation}}
\def\fr{\operatorname{fr}}

\newcommand\BDs{B_\D^{\text{sur}}}
\newcommand\BDps{B_{\D,+}^{\text{sur}}}
\newcommand\vD{\varphi_\D}
\newcommand\Dess{\D_{\mathrm{ess}}}

\newcommand{\fXsur}{\mathfrak{X}^{\text{sur}}}
\newcommand{\fXsurD}{\fXsur(\D)}
\newcommand{\fXsurDp}{\fXsur_+(\D)}
\newcommand\vpDp{\vp_{\D'}}
\newcommand\bPsi{\bar \Psi}

\newcommand\bY{\bar Y}

\newcommand{\embed}{\hookrightarrow}

\newcommand{\tr}{\mathrm{tr}}
\newcommand{\Trq}{\mathrm{Tr}_q}
\newcommand{\Mat}{\mathrm{Mat}}
\newcommand{\id}{\mathrm{id}}
\newcommand{\Rel}{\mathrm{Rel}}
\newcommand{\pr}{\mathrm{pr}}
\newcommand\bd{{\mathrm{bd}}}
\newcommand\mon{{\mathrm{mon}}}
\newcommand\inn{{\mathrm{inn}}}

\newcommand{\supp}{\mathrm{supp}}
\DeclareMathOperator{\ord}{ord}
\newcommand{\Col}{\mathrm{Col}}
\newcommand{\Id}{\mathrm{Id}}
\newcommand{\inqa}{\mathrm{in}_{\mathrm{q,a}}}
\newcommand\Newt{\text{Newt}}
\newcommand\Cone{\text{Cone}}

\newcommand\Sp{\mathrm{Spec}}
\newcommand\Hom{{\mathrm{Hom}}}
\newcommand\ot{{\otimes}}

\newcommand\refl{\mathrm{refl}}

\newcommand{\hatP}{\widehat{\cP}}

\newcommand\Xp{\sX_+(\D)}
\newcommand\bXp{\overline\sX_+(\D)}
\newcommand\bXpr{\overline\sX'(\D)}
\newcommand\siD{\overline{\D}_\inn}

\newcommand\reS{\overline{\cS}}
\newcommand\reX{\overline{\sX}}

\newcommand\reNf{\overline{\fN}}

\newcommand\csX{\mathring{\mathfrak{X}}}
\newcommand\bfM{\widetilde {\fM}}
\newcommand\Cx{\mathbb{C}^\times}

\def\XzD{\sX_\xi(\D)}
\def\XeD{\sX_\ve(\D)}

\def\tXeD{\widetilde\sX_\ve(\D)}
\def\tXzD{\widetilde \sX_\xi(\D)}

\def\tPhi{\tilde\Phi}

\def\tSz{\widetilde \cS_\xi}
\def\tSe{\widetilde\cS_\ve}
\def\tF{\tilde F}
\def\tX{\widetilde \sX}
\def\tS{\widetilde \cS}
\def\ttS{\widetilde{\cS}}
\def\SPpp{(\Sigma', \cP')}
\newcommand{\An}{\mathbb{A}}

\def\SQ{(\Sigma, \cQ)}
\def\LSQ{\cT\SQ}

\newcommand\incl[2]{{\includegraphics[height=#1]{#2.eps}}}
\def\Lplu{\left(  \raisebox{-13pt}{\incl{1.2 cm}{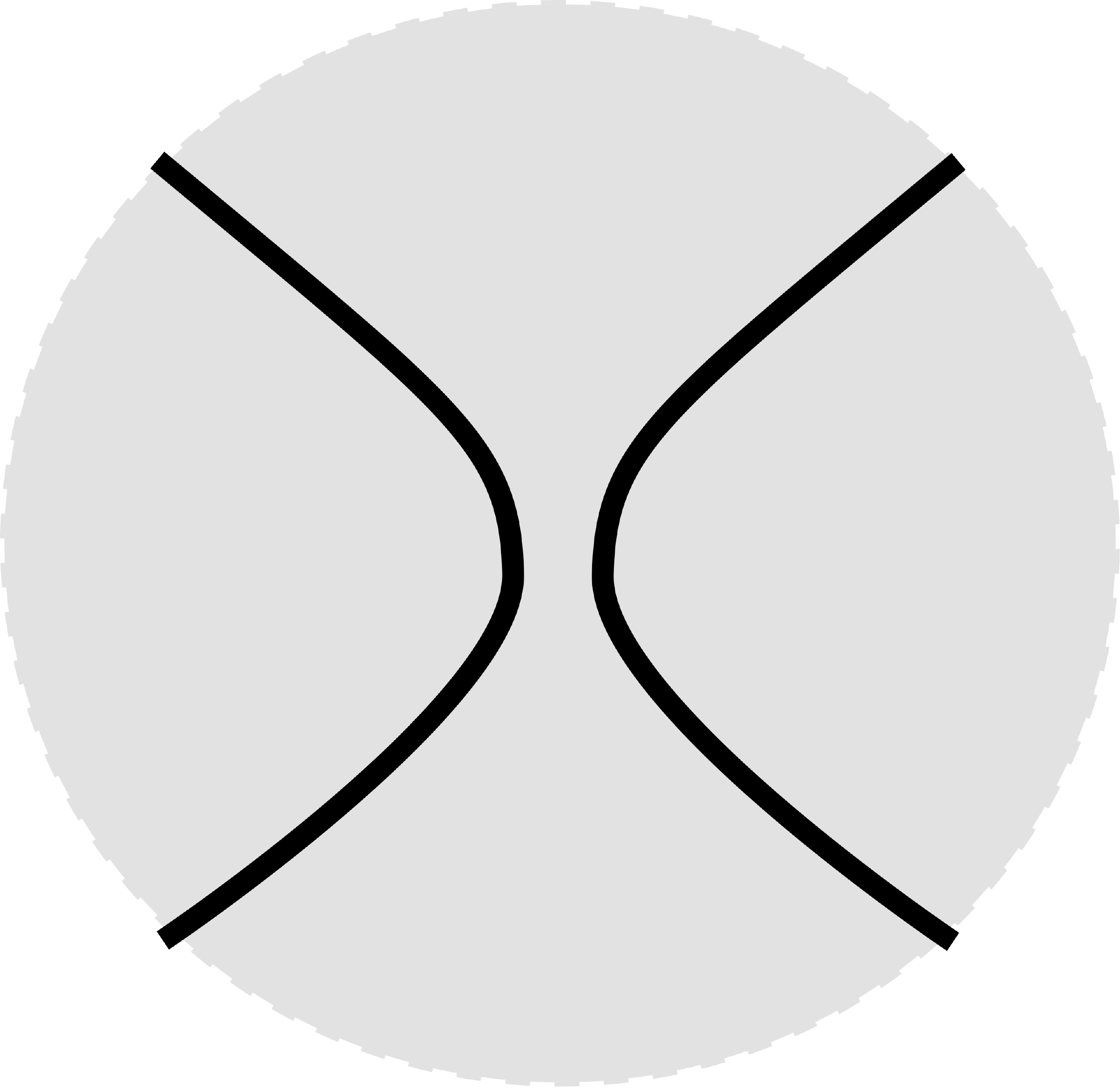}} \right) }
\def\Lpluss{\left(  \raisebox{-13pt}{\incl{1.2 cm}{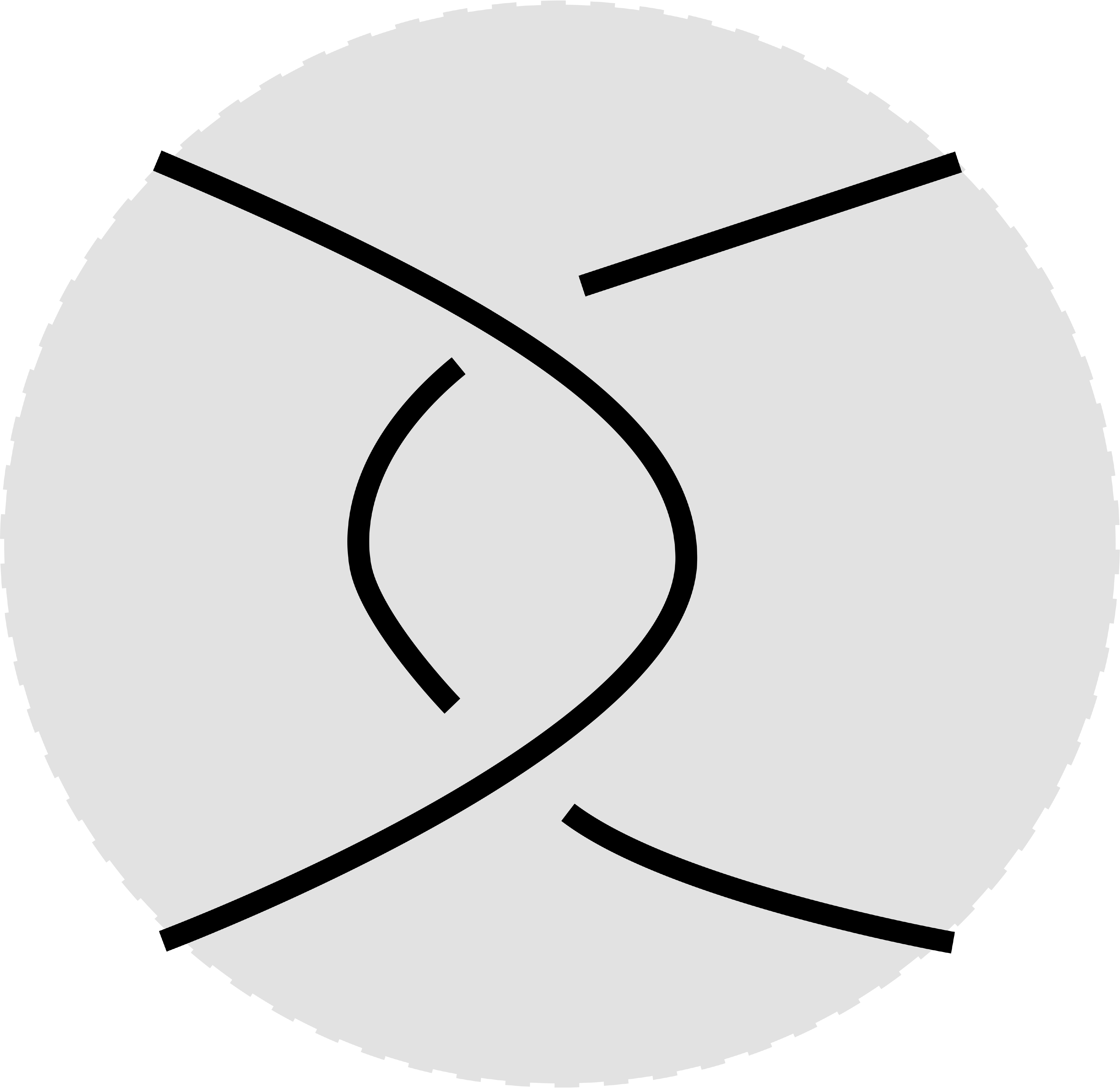}} \right) }
\def\uplus{\left(  \raisebox{-13pt}{\incl{1.2 cm}{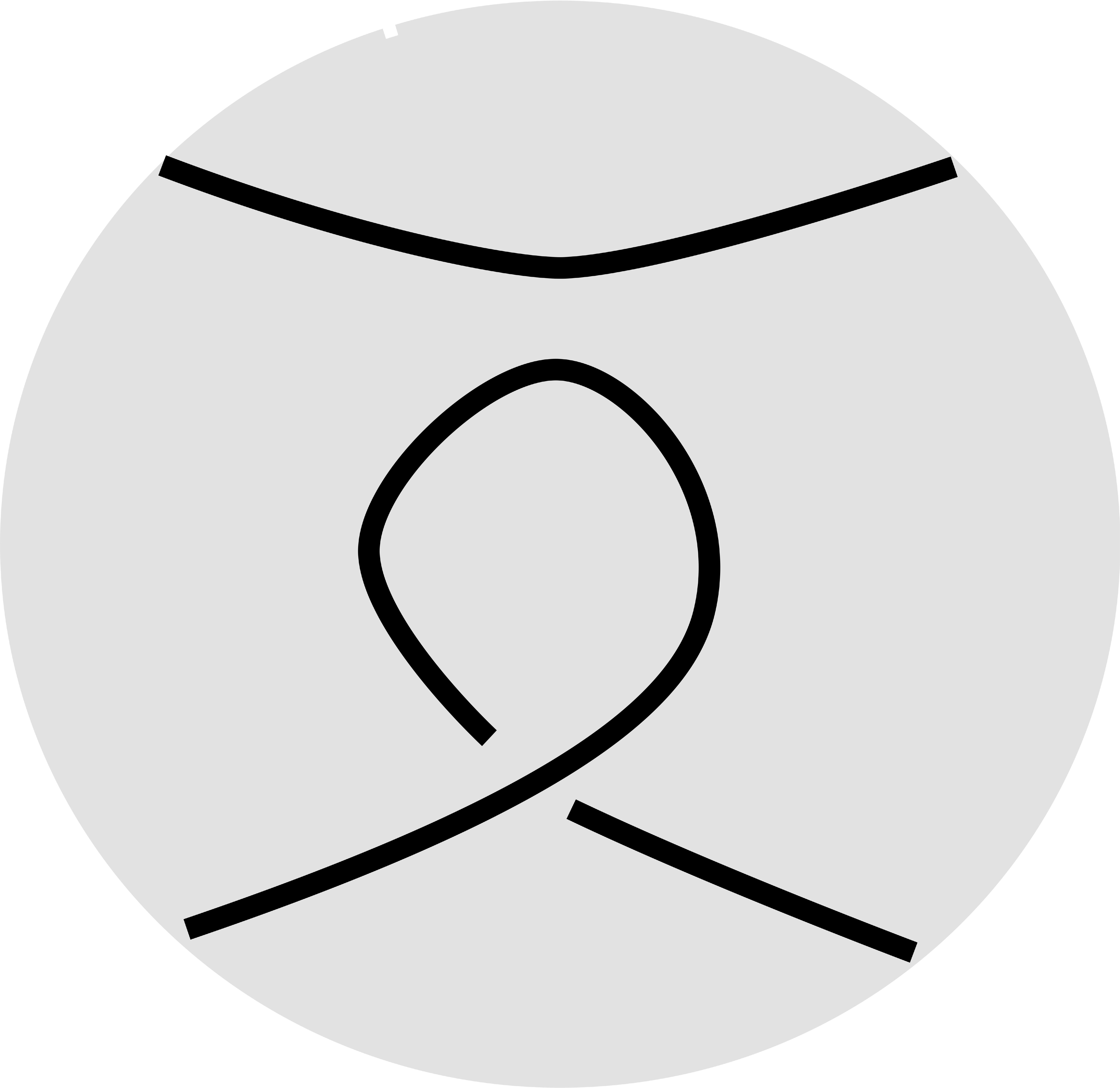}} \right) }
\def\uminus{\left(  \raisebox{-13pt}{\incl{1.2 cm}{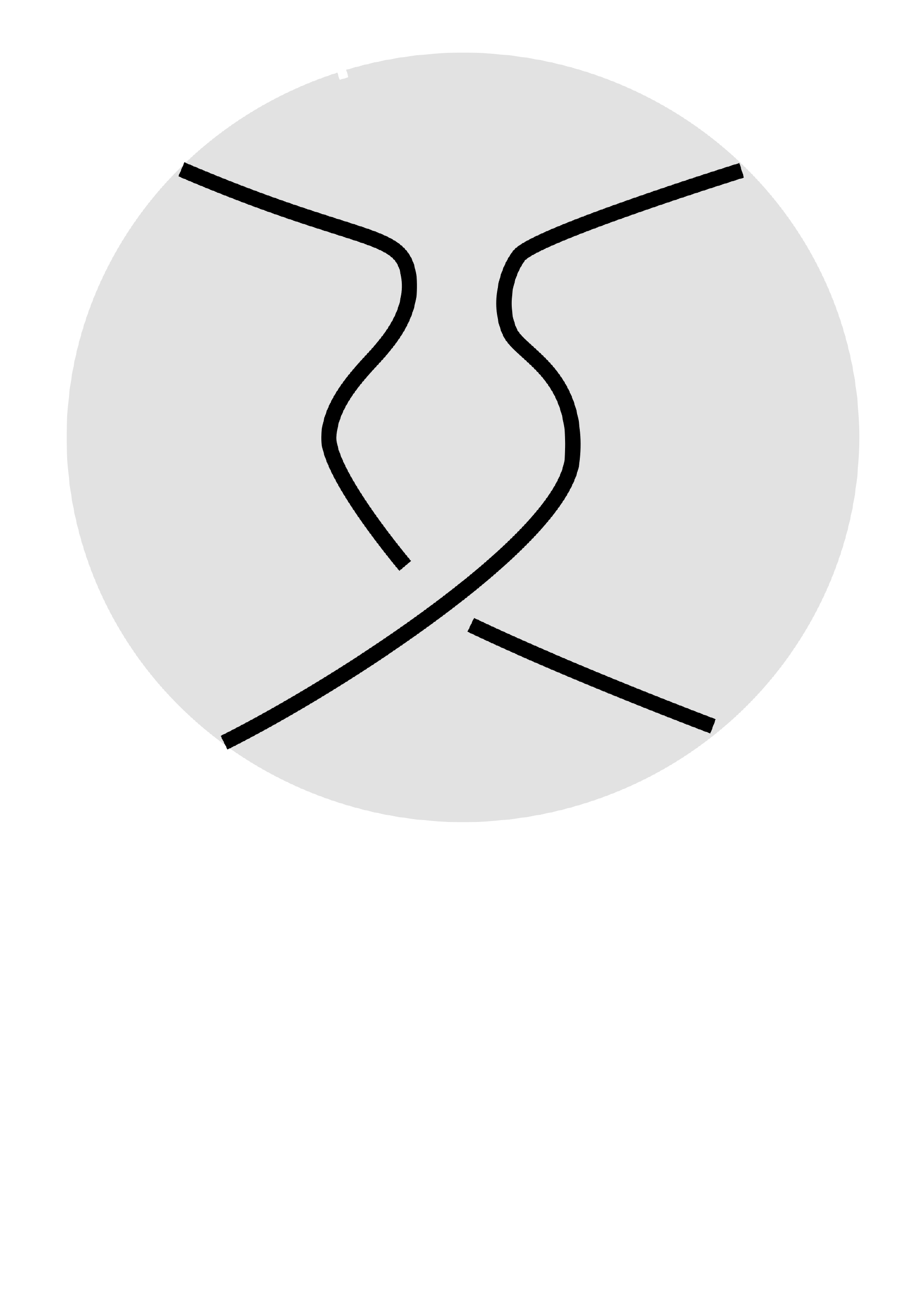}} \right) }
\def\LLL{\left(  \raisebox{-13pt}{\incl{1.2 cm}{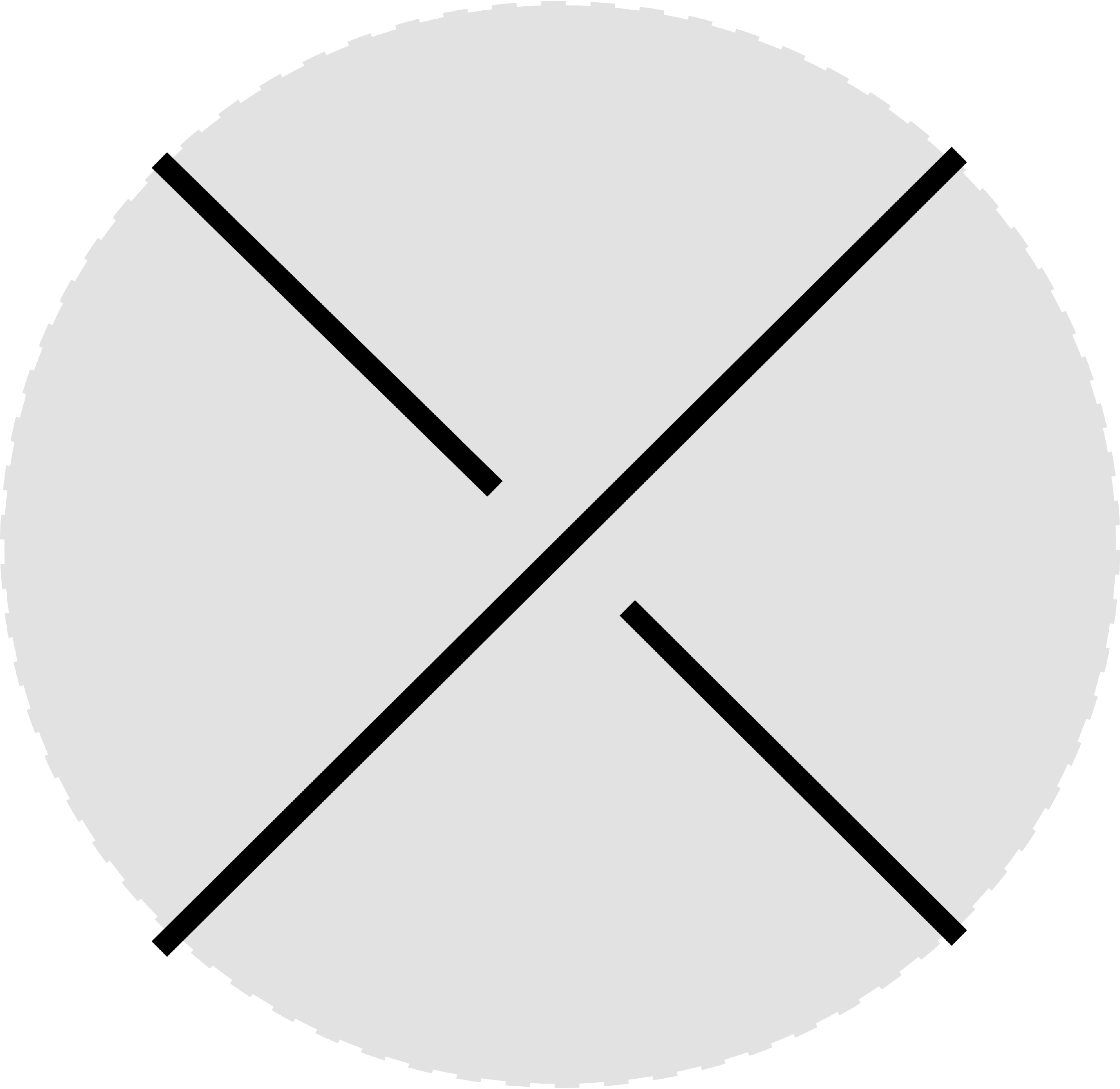}} \right) }
\def\Lminus{\left(  \raisebox{-13pt}{\incl{1.2 cm}{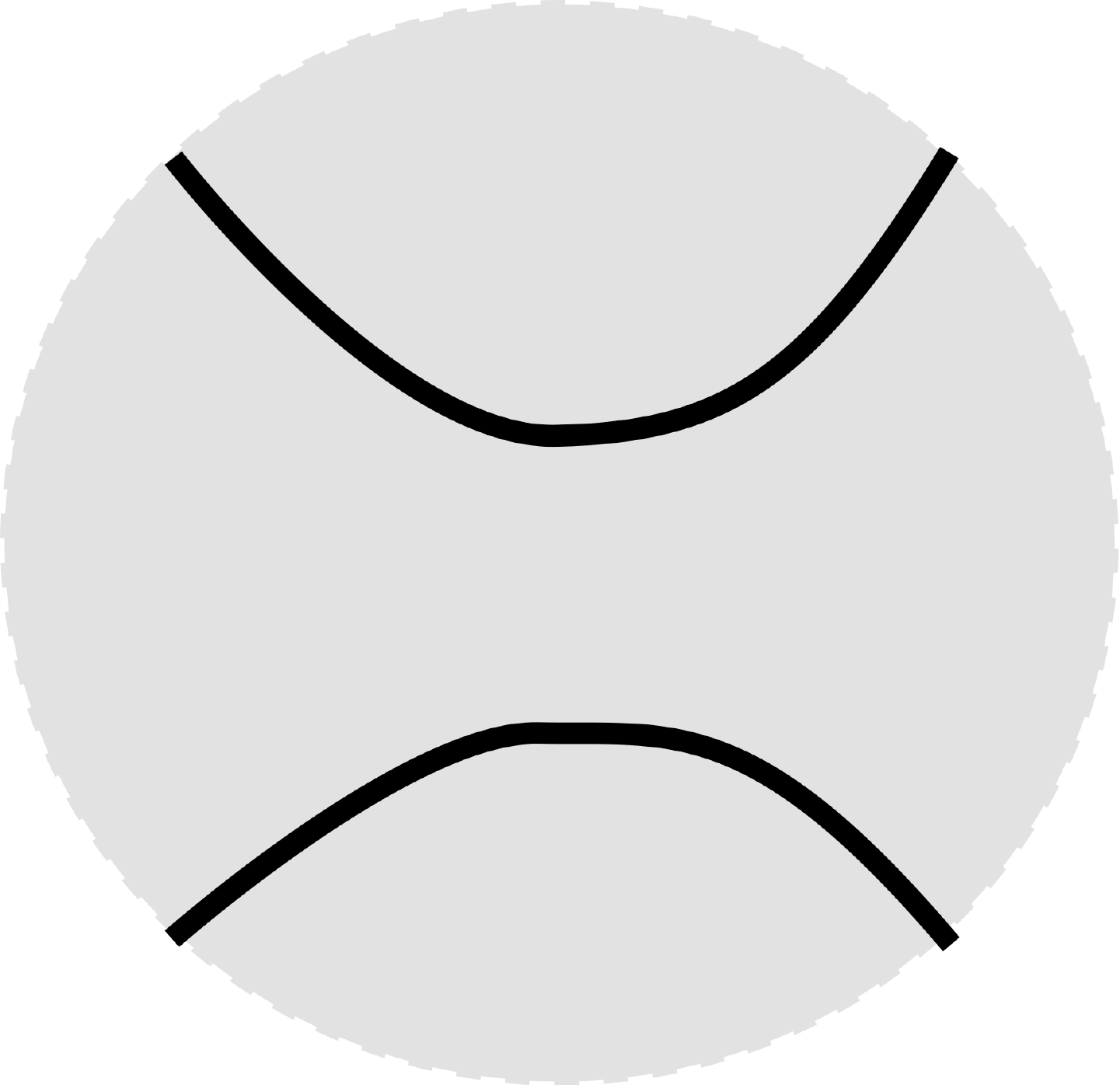}} \right) }
\def\negkinka{\raisebox{-13pt}{\incl{1.2 cm}{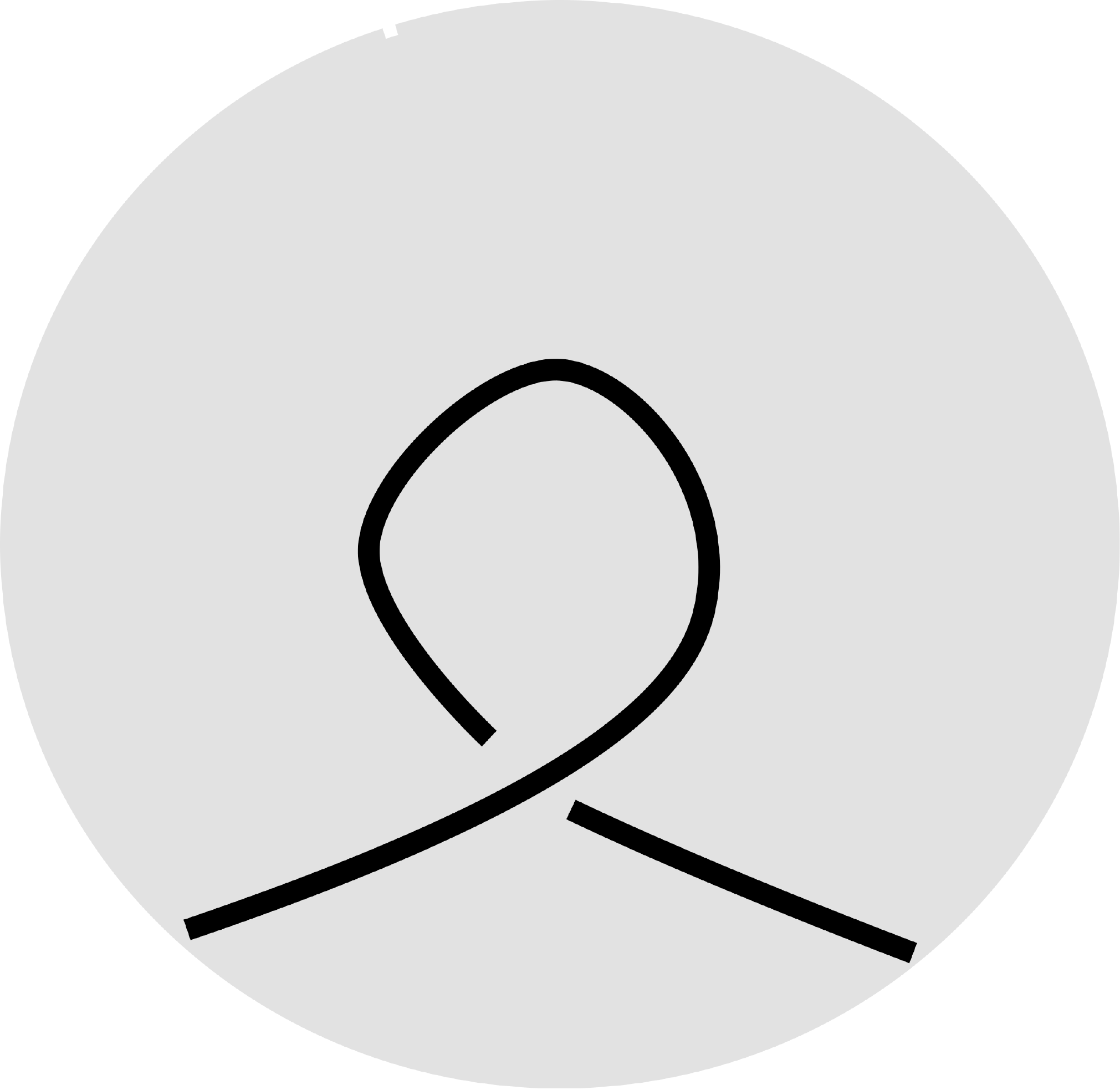}} }
\def\negkinkb{\raisebox{-13pt}{\incl{1.2 cm}{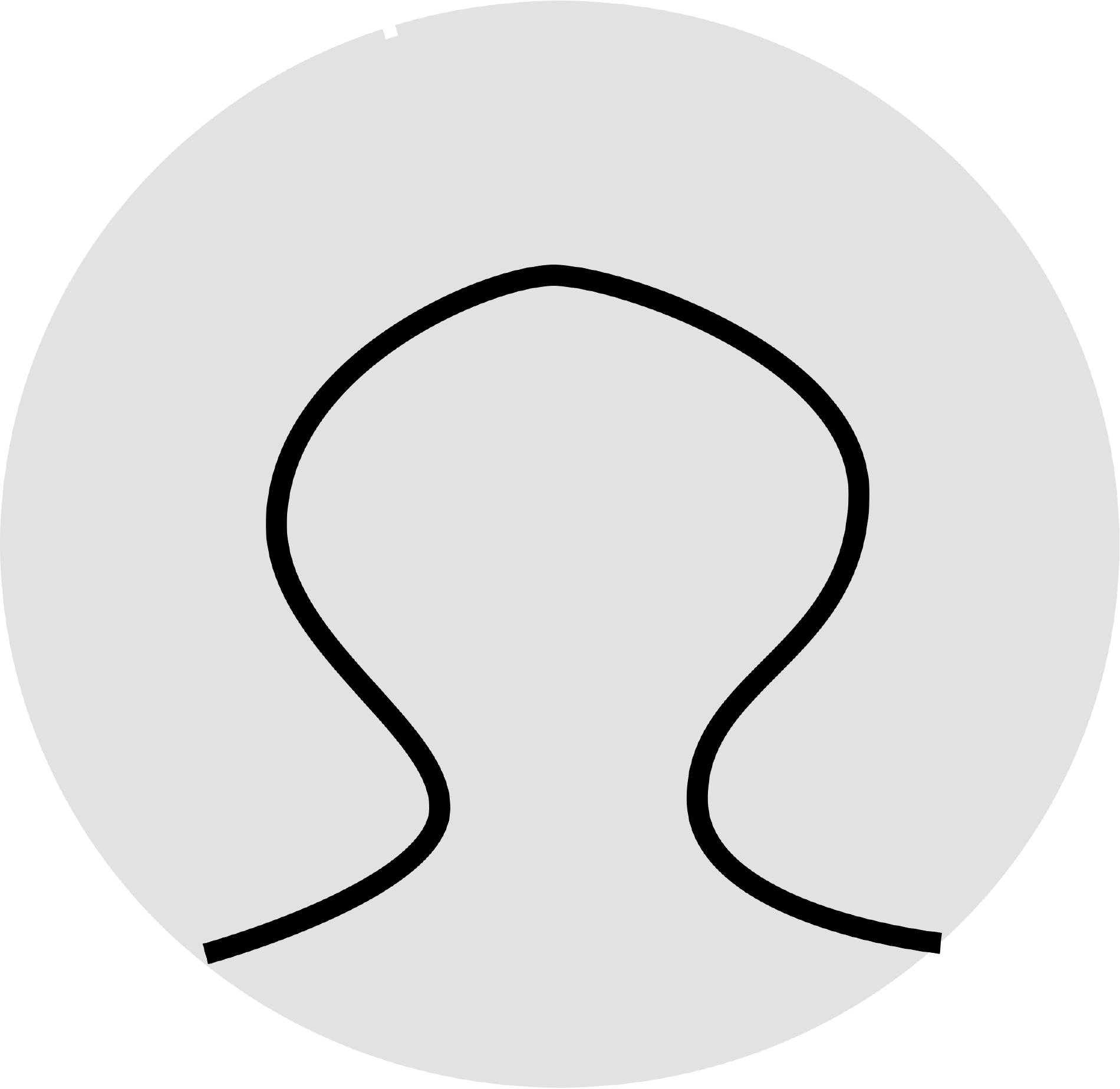}} }
\def\negkinkc{\raisebox{-13pt}{\incl{1.2 cm}{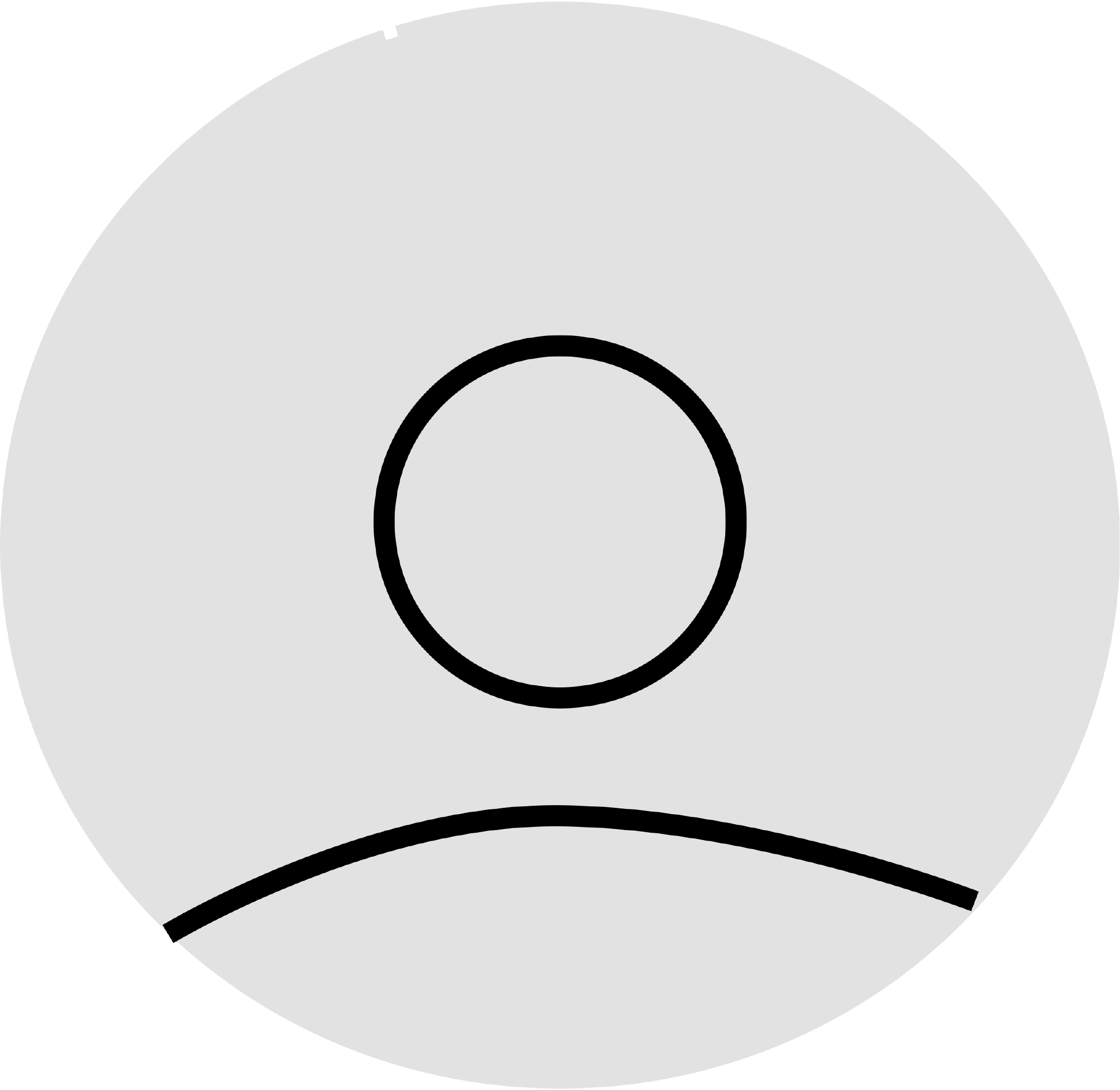}} }
\def\negkinkd{\raisebox{-13pt}{\incl{1.2 cm}{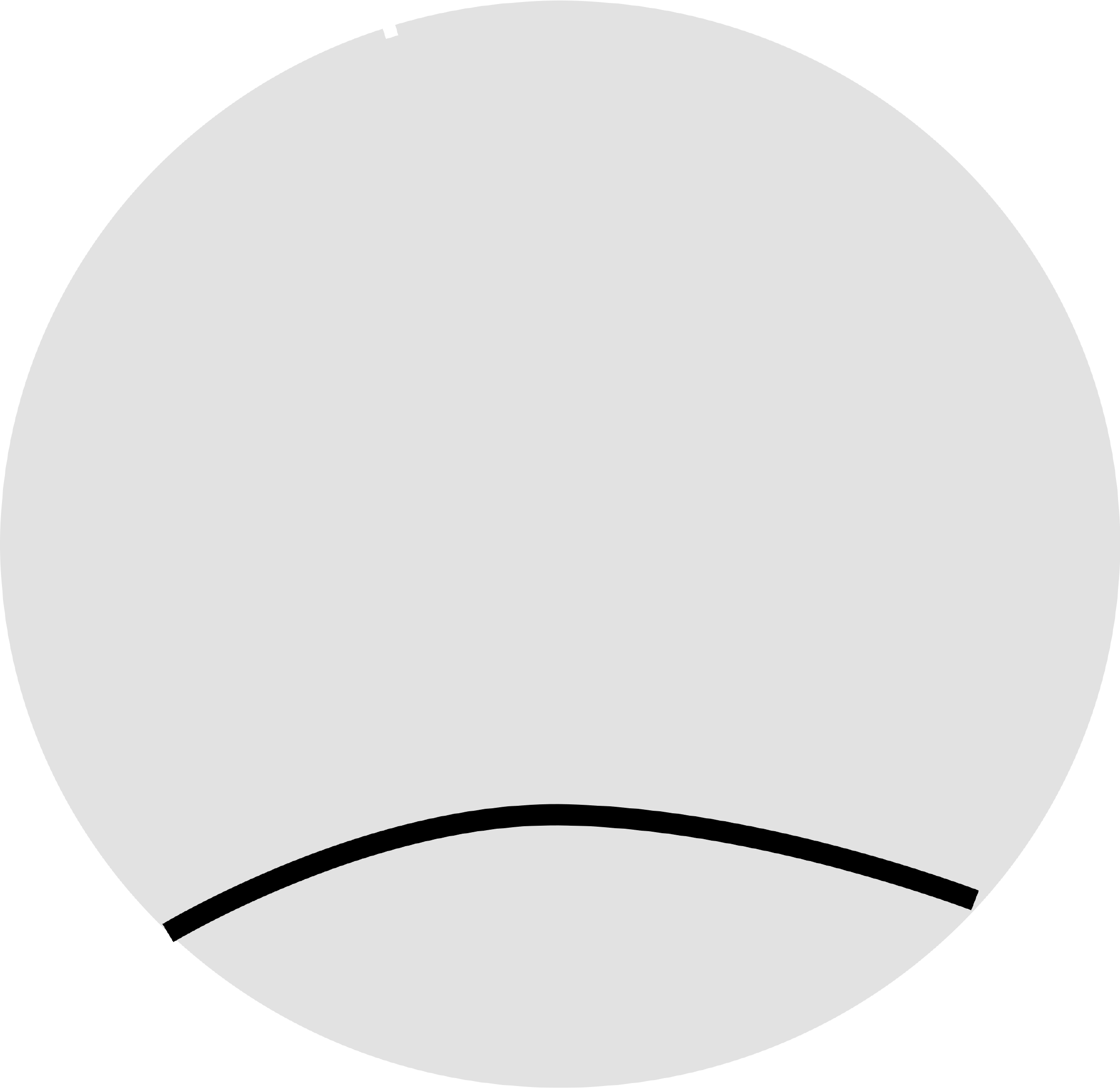}} }

\def\hatP{{\widehat{\partial}}}

\def\SxiS{\cS_{\xi}(\fS)}
\def\ZxiS{Z_{\xi}(\fS)}
\def\YxiS{ {\mathscr Y}_{\xi}(\fS)}
\def\ord{\mathrm{ord}}
\def\onto{\twoheadrightarrow}

\def\oD{{\mathring\Delta}}

\newcommand\TrqN{\text{tr}_{q^{N^2}}}
\newcommand\Trxi{\text{tr}_{\xi}}
\AtEveryBibitem{\clearfield{issn}}
\AtEveryBibitem{\clearlist{issn}}

\AtEveryBibitem{\clearfield{language}}
\AtEveryBibitem{\clearlist{language}}

\AtEveryBibitem{\clearfield{doi}}
\AtEveryBibitem{\clearlist{doi}}

\AtEveryBibitem{\clearfield{url}}
\AtEveryBibitem{\clearlist{url}}

\AtEveryBibitem{%
  \ifentrytype{online}
    {}
    {\clearfield{urlyear}\clearfield{urlmonth}\clearfield{urlday}}}

\begin{document}

\doublespacing  %

\input{titlePage.tex}

\currentpdfbookmark{Title Page}{titlePage}  %

\input{approvalPage.tex}

\input{epigraph.tex}

\input{dedication.tex}

\pagenumbering{roman}
\addcontentsline{toc}{chapter}{Acknowledgments}
\setcounter{page}{5} %
\input{acknowledgements.tex}
\renewcommand{\cftchapdotsep}{\cftdotsep}  %
\renewcommand{\cftchapfont}{\bfseries}  %
\renewcommand{\cftchappagefont}{}  %
\renewcommand{\cftchappresnum}{Chapter }
\renewcommand{\cftchapaftersnum}{:}
\renewcommand{\cftchapnumwidth}{5em}
\renewcommand{\cftchapafterpnum}{\vskip\baselineskip} %
\renewcommand{\cftsecafterpnum}{\vskip\baselineskip}  %
\renewcommand{\cftsubsecafterpnum}{\vskip\baselineskip} %
\renewcommand{\cftsubsubsecafterpnum}{\vskip\baselineskip} %

\titleformat{\chapter}[display]
{\normalfont\bfseries\filcenter}{\chaptertitlename\ \thechapter}{0pt}{\MakeUppercase{#1}}

\renewcommand\contentsname{Table of Contents}

\begin{singlespace}
\tableofcontents
\end{singlespace}

\currentpdfbookmark{Table of Contents}{TOC}

\clearpage

\addcontentsline{toc}{chapter}{List of Figures}
\begin{singlespace}
\setlength\cftbeforefigskip{\baselineskip}  %
\listoffigures
\end{singlespace}
\clearpage
\no{
\addcontentsline{toc}{chapter}{Notation}
\setlength\cftbeforefigskip{\baselineskip} 
\input{notation.tex}
\clearpage 
}

\input{summary.tex}
\clearpage
\pagenumbering{arabic}
\setcounter{page}{1} %

\titleformat{\chapter}[display]
{\normalfont\bfseries\filcenter}{\MakeUppercase\chaptertitlename\ \thechapter}{0pt}{\MakeUppercase{#1}}  %
\titlespacing*{\chapter}
  {0pt}{0pt}{30pt}	%

\titleformat{\section}{\normalfont\bfseries}{\thesection}{1em}{#1}

\titleformat{\subsection}{\normalfont}{\uline{\thesubsection}}{0em}{\uline{\hspace{1em}#1}}

\titleformat{\subsubsection}{\normalfont\itshape}{\thesubsection}{1em}{#1}

\input{introduction.tex}

\input{surfacetopology.tex}

\input{algebra.tex}

\input{kauffman.tex}

\input{quantumtopology.tex}

\input{embedding.tex}

\input{surgerytheory.tex}

\input{chebyshevfrobenius.tex}

\input{center.tex}

\input{maximalorder.tex}

\input{appendix.tex}

\begin{singlespace}  %
	\setlength\bibitemsep{\baselineskip}  %
	\printbibliography[title={References}]
\end{singlespace}

\addcontentsline{toc}{chapter}{References}  %

\input{vita.tex}

\end{document}

%% file: titlePage.tex
\newcommand{\thesisTitle}{Quantum torus methods for Kauffman bracket skein modules}
\newcommand{\yourName}{Jonathan M. Paprocki}
\newcommand{\yourSchool}{Mathematics}
\newcommand{\yourMonth}{December}
\newcommand{\yourYear}{2019}

\begin{titlepage}
\begin{center}

\begin{singlespacing}

\textbf{\MakeUppercase{\thesisTitle}}\\
\vspace{10\baselineskip}
A Dissertation\\
Presented to\\
The Academic Faculty\\
\vspace{3\baselineskip}
By\\
\vspace{3\baselineskip}
\yourName\\
\vspace{3\baselineskip}
In Partial Fulfillment\\
of the Requirements for the Degree\\
Doctor of Philosophy in the\\
School of \yourSchool\\
\vspace{3\baselineskip}
Georgia Institute of Technology\\
\vspace{\baselineskip}
\yourMonth{} \yourYear{}
\vfill
Copyright \copyright{} \yourName{} \yourYear{}

\end{singlespacing}

\end{center}
\end{titlepage}

%% file: approvalPage.tex
\newcommand{\committeeMemberOne}{Dr. Thang T. Q. L\^e, Advisor}
\newcommand{\committeeMemberOneDepartment}{School of Mathematics}
\newcommand{\committeeMemberOneAffiliation}{Georgia Institute of Technology}

\newcommand{\committeeMemberTwo}{Dr. Jean Bellissard}
\newcommand{\committeeMemberTwoDepartment}{School of Mathematics, School of Physics}
\newcommand{\committeeMemberTwoAffiliation}{Georgia Institute of Technology}

\newcommand{\committeeMemberThree}{Dr. Stavros Garoufalidis}
\newcommand{\committeeMemberThreeDepartment}{School of Mathematics}
\newcommand{\committeeMemberThreeAffiliation}{Georgia Institute of Technology}

\newcommand{\committeeMemberFour}{Dr. Wade Bloomquist}
\newcommand{\committeeMemberFourDepartment}{Department of Mathematics}
\newcommand{\committeeMemberFourAffiliation}{University of California, Santa Barbara}

\newcommand{\committeeMemberFive}{Dr. Igor Belegradek}
\newcommand{\committeeMemberFiveDepartment}{School of Mathematics}
\newcommand{\committeeMemberFiveAffiliation}{Georgia Institute of Technology}

\newcommand{\approvalDay}{26}
\newcommand{\approvalMonth}{July}
\newcommand{\approvalYear}{2019}

\begin{titlepage}
\begin{singlespacing}
\begin{center}

\textbf{\MakeUppercase{\thesisTitle}}\\
\vspace{10\baselineskip}

\end{center}
\vfill

\ifdefined\committeeMemberFour

Approved by:
\vspace{2\baselineskip}		%

\begin{minipage}[b]{0.4\textwidth}
	
	\committeeMemberOne\\
	\committeeMemberOneDepartment\\
	\textit{\committeeMemberOneAffiliation}\\
	
	\committeeMemberTwo\\
	\committeeMemberTwoDepartment\\
	\textit{\committeeMemberTwoAffiliation}\\
	
	\committeeMemberThree\\
	\committeeMemberThreeDepartment\\
	\textit{\committeeMemberThreeAffiliation}\\
	
	\vspace{2\baselineskip}		%
	
\end{minipage}
\hspace{0.1\textwidth}
\begin{minipage}[b]{0.4\textwidth}
	
	\committeeMemberFour\\
	\committeeMemberFourDepartment\\
	\textit{\committeeMemberFourAffiliation}\\
	
	\ifdefined\committeeMemberSix
	\committeeMemberFive\\
	\committeeMemberFiveDepartment\\
	\textit{\committeeMemberFiveAffiliation}\\
	
	\committeeMemberSix\\
	\committeeMemberSixDepartment\\
	\textit{\committeeMemberSixAffiliation}\\
	
	Date Approved: \approvalMonth{} \approvalDay, \approvalYear
	\vspace{1\baselineskip}		%
	
	\else
	
	\committeeMemberFive\\
	\committeeMemberFiveDepartment\\
	\textit{\committeeMemberFiveAffiliation}\\
		
	Date Approved: \approvalMonth{} \approvalDay, \approvalYear
	\vspace{5\baselineskip}		%
	
	\fi
	
\end{minipage}

\else

\hspace{0.6\textwidth}
\begin{minipage}[b]{0.4\textwidth}
	
	Approved by:
	\vspace{2\baselineskip}		%
	
	\committeeMemberOne\\
	\committeeMemberOneDepartment\\
	\textit{\committeeMemberOneAffiliation}\\
	
	\committeeMemberTwo\\
	\committeeMemberTwoDepartment\\
	\textit{\committeeMemberTwoAffiliation}\\
	
	\committeeMemberThree\\
	\committeeMemberThreeDepartment\\
	\textit{\committeeMemberThreeAffiliation}\\
	
	\vspace{2\baselineskip}		%
	
	Date Approved: \approvalMonth{} \approvalDay, \approvalYear
	\vspace{\baselineskip}		%
	
\end{minipage}

\fi

\end{singlespacing}
\end{titlepage}

%% file: epigraph.tex
\newcommand{\yourQuote}{The quantum, strangest feature of this strange universe, cracks the armor that conceals the secret of existence.}
\newcommand{\yourAuthor}{John Archibald Wheeler}

\begin{titlepage}
\begin{center}

\vspace*{\fill}
\yourQuote\\
\textit{\yourAuthor}
\vspace*{\fill}

\end{center}
\end{titlepage}

%% file: dedication.tex
\newcommand{\yourDedication}{For Eris}

\begin{titlepage}
\begin{center}

\vspace*{\fill}
\yourDedication\\
\vspace*{\fill}

\end{center}
\end{titlepage}

%% file: acknowledgements.tex
\clearpage
\begin{centering}
\textbf{ACKNOWLEDGEMENTS}
\vspace{\baselineskip}
\end{centering}
\phantom{blah}\\
This dissertation is dedicated to everybody in my life who has supported me throughout my trek in graduate school. I would never have made it this far without a large supportive community cheering me on every step of the way.

I would first and foremost like to thank my adviser, Prof. Thang T. Q. L\^e. Studying under a master of infinite patience with the highest possible standards has been the greatest privilege of my life, and the technical skill and determination he has bestowed upon me through his tireless effort is a priceless gift that I never would have be able to obtain on my own.

I'd also like to thank my oldest mentor, Prof. Jean Bellissard, who first took me under his wing as an undergraduate ten years ago in 2009 directing a reading course on the quantum Hall effect. This was my initial exposure to the intersection of quantum physics and topology, marking my first step towards studying this endlessly fascinating topic.

I also must thank Prof. Stavros Garoufalidis, under whom I first studied Chern-Simons theory, and Prof. Matilde Marcolli, who introduced me to wonders of mathematical physics too numerous to list.

I also must thank everybody who supported my personal development and mental health through all these years. This list is very long, but there is no way I would have survived this trek without their aid. Special thanks must be acknowledged for my parents who have supported me in all of my endeavors for my entire life. And to my friends, too numerous to list, but Michelle Su, Amy Stewart, Megan Iams, Samantha Van Ness, Yeager Gaston, and Jonathan Popham, deserve a particular distinction for having made a special difference in my life at one point or another.

And to the organizations and communities I have had the privilege of being a part of that have kept me happy and sane this whole time. This includes Psi Upsilon Gamma Tau, the Southeast Regional Burning Man community, the Holochain community, Freeside Atlanta, the Qualia Computing community, and the UCLA Mind Wrecking Crew.

And lastly, to Georgia Tech, the focal point of my life since 2007. To Hell With Georgia!

\clearpage

%% file: summary.tex
\clearpage
\begin{centering}
\textbf{SUMMARY}\\
\vspace{\baselineskip}
\end{centering}

We investigate aspects of Kauffman bracket skein algebras of surfaces and modules of 3-manifolds using quantum torus methods. These methods come in two flavors: embedding the skein algebra into a quantum torus related to quantum Teichm\"uller space, or filtering the algebra and obtaining an associated graded algebra that is a monomial subalgebra of a quantum torus. We utilize the former method to generalize the Chebyshev homomorphism of Bonahon and Wong between skein algebras of surfaces to a Chebyshev-Frobenius homomorphism between skein modules of marked 3-manifolds, in the course of which we define a surgery theory, and whose image we show is either transparent or skew-transparent. The latter method is used to show that skein algebras of surfaces are maximal orders, which implies a refined unicity theorem, shows that $SL_2\BC$-character varieties are normal, and suggests a conjecture on how this result may be utilized for topological quantum compiling.

%% file: introduction.tex
\chapter{Introduction}

In this dissertation we study aspects of Kauffman bracket skein modules of 3-manifolds and Kauffman bracket skein algebras of surfaces\footnote{For the entirety of this work, we will collectively refer to both simply as skein modules.} primarily via the utilization of two different ``quantum torus methods,'' one of which is a technique with close ties with quantum hyperbolic geometry and the other is via standard methods of filtered algebras and associated graded algebras. We begin in Section \ref{s.introkauffman} with an overview of what skein modules are, where they came from, and their relationship to low-dimensional topology. We then give a picture of what quantum tori are in Section \ref{s.introtori} and explain how they may be wielded as a powerful tool with which to study skein modules. Then in Section \ref{s.introresults} we give an outline of the content of this dissertation, including a discussion of our primary results and a conjecture on possible applications of some of our results towards topological quantum computing. All work in this dissertation is joint with L\^e \cite{LP2018,LP2019}.

Throughout the introduction we fix a commutative domain $R$ with a distinguished invertible element $q^{1/2}$ whose square $q$ is sometimes called the {\em quantum parameter}. We keep in mind the examples $R=\BZ[q^{\pm 1/2}]$ and $R=\BC$. Several algebras are defined that depend on $q$ and this is denoted with a subscript, but we will frequently suppress this subscript when the value of $q$ is understood.

\section{Kauffman bracket skein modules}\lbl{s.introkauffman}

\subsection{History}

Our story begins with the discovery of the Jones polynomial in 1985 \cite{Jo1985} by Vaughn Jones. This was the first new knot polynomial discovered since Alexander in 1928 \cite{Al1928}, and this surprising discovery landed Jones the Fields medal in 1990. Computation of the Jones polynomial involved methods of operator algebras that were unfamiliar to most topologists, and in the spirit of Conway's discovery of a skein relation to compute the closely related Conway-Alexander polynomial in 1969 \cite{Co1970}, Louis Kauffman uncovered what is now called the Kauffman bracket skein relation in 1987 \cite{Ka1987} given by $T=qT_+ + q^{-1}T_-$, depicted in Figure \ref{fig:figures/skein4}, to compute the closely related Kauffman bracket polynomial. Topologists were left with the unhappy situation that they possessed a new knot invariant but it was described entirely in terms of either operator algebras or diagrams drawn in the plane - no 3-dimensional description of the Jones polynomial yet existed. Atiyah challenged Witten to find such a description, and he succeeded in 1989 \cite{Wi1989} with the invention of Chern-Simons theory - the first topological quantum field theory. This achievement, among others, led to Witten sharing the 1990 Fields medal with Jones, as well as Drinfeld and Mori. This thread in scientific history has since spawned a thousand rivers of inquiry in the field of mathematical physics and physical mathematics, and our contribution here marks another inch towards understanding this story.

\FIGc{figures/skein4}{From left to right:  $T, T_+, T_-$.}{1.2cm}

Skein modules were introduced independently by Przytycki \cite{Pr1999} while attempting to study the invariants of colored knots and by Turaev \cite{Tu1990,Tu1991} as a deformation quantization of the Poisson algebra of loops on surfaces, or relatedly, as a geometric quantization of the Atiyah-Bott-Weil-Petersen-Goldman symplectic structure of the $SL_2\BC$ character variety.

In the spirit of the {\em analysis situs} of Leibniz, Przytycki has described skein modules as {\em algebra situs} - an algebraic topology based on knots. The relationship between manifolds and knots that can be embedded within them has many fruitful analogies with the relationship between rings and their modules. Here, we encodes topological and geometrical information about surfaces and 3-manifolds by forming a module generated by isotopy classes of knots embedded in the manifold quotiented by the Kauffman bracket skein relation in Figure \ref{fig:figures/skein4}. These knots can be thought of as functions on the $SL_2\BC$ character variety of the manifold, thus this makes skein modules a sort of K-theoretic approach to the $SL_2\BC$-character variety. A more physical way of thinking about skein modules would be that special elements known as Jones-Wenzl idempotents represent Temperley-Lieb-Jones anyons conjectured to exist in certain fractional quantum Hall effect systems, giving possible connections towards topological quantum computing that we explore in the last chapter.

Skein modules have been found to have connections with many important objects in classical and quantum topology, including but not limited to character varieties \cite{Bu1997,PS2000,Si2004,Ma2011}, the Jones polynomial \cite{Le2006}, (quantum) Teichm\"uller spaces \cite{BW2011,Le2017}, cluster algebras \cite{Mu2012}, topological quantum computing \cite{Ka2002,Wa2010}, and topological quantum field theories \cite{BHMV1995,BW2016}. They occupy a crossroads where representation-theoretic and diagrammatic aspects of quantum topology as well as classical topology intersect \cite{Tu1991}.

\subsection{Skein modules for marked 3-manifolds}

Let us first introduce Kauffman bracket skein modules for marked 3-manifolds. A {\em marked 3-manifold} is a pair $(M,\cN)$, where $M$ is an oriented  3-manifold with (possibly empty) boundary $\partial M$ and a 1-dimensional oriented submanifold $\cN \subset \partial M$ such that each of the finite number of connected components of $\cN$ is diffeomorphic to the  open interval $(0,1)$. 

An {\em $\cN$-tangle in $\MN$} is a compact 1-dimensional non-oriented submanifold $T$ of $M$ equipped with a normal vector field, called the {\em framing}, such that $\partial T= T \cap \cN$ and the framing at each boundary point of $T$ is a positive tangent vector of $\cN$. Two $\cN$-tangles are {\em $\cN$-isotopic} if they are isotopic in the class of $\cN$-tangles. Given a single component $\cN$-tangle $\al$, if $\al$ is diffeomorphic to $S^1$ then $\al$ is called an {\em $\cN$-knot}, otherwise $\al$ is called an {\em $\cN$-arc}.

The {\em Kaufman bracket skein module of $\MN$ at $q$}, $\cS\MN:=\cS_q\MN$, is the $R$-module freely spanned by $\cN$-isotopy classes of $\cN$-tangles modulo the local relations described in Figure \ref{fig:figures/skein2}.

\FIGc{figures/skein2}{Defining relations of the skein module. From left to right: skein relation, trivial knot relation, and trivial arc relation.}{1.5cm}

When $\cN=\emptyset$ we don't need the third relation (the trivial arc relation), and this is the Kauffman bracket skein module first found by Przytycki \cite{Pr1999} and Turaev \cite{Tu1990,Tu1991}. The skein relation and trivial knot relation were introduced by Kauffman \cite{Ka1987} in his study of the Jones polynomial.

An $\cN$-tangle $T$ is {\em essential} when no component of $T$ satisfies the trivial knot or trivial arc relation. The set of all essential $\cN$-tangles is a basis for $\cS\MN$ over $R$. 

The trivial arc relation was introduced by Muller in \cite{Mu2012} where he defined the skein module of marked surfaces (see Subsection \ref{s.introsurface}), which was then generalized to marked 3-manifolds by L\^e \cite{Le2017}. It turns out that the extension to include the marked set $\cN$ on the boundary $\partial M$ makes the study of the skein modules much easier both in technical and conceptual perspectives.

\subsection{Skein algebras of marked surfaces}\lbl{s.introsurface}

Throughout this work we deal with marked surfaces, which has a special case we call finite type surfaces. This distinction is inessential for most of our main results, which apply to both kinds of surfaces, but the type of surface we use to prove a given result is the context in which that result is most easily proven and the other case follows as a corollary. As the way we treat these objects has subtle but important differences, we use different language and notation when describing them despite the fact that the same language could be used for both types in some cases.

A {\em marked surface} is a pair $\SP$, where $\Sigma$ is an oriented compact surface with (possibly empty) boundary $\pS$ and a finite set $\cP\subset \pS$. Define an $R$-module $\cS\SP:= \cS_q\MN$, where $M = \Sigma \times (-1,1)$ and $\cN=\cP \times (-1,1)$. Given two $\cN$-tangles $T, T'$, define the product $T T'$ by stacking $T$ above $T'$. To be more precise, we isotope $T$ to lie in $\Sigma \times (0,1)$ and $T'$ to lie in $\Sigma \times (-1,0)$, and consider $TT'$ to be the union $T \cup T'$. If $\beta$ is an {\em unmarked boundary component of $\pS$}, i.e. $\beta \cap \cP=\emptyset$, then $\beta$ is a central element of $\cS\SP$. We write $\cH$ for set of all unmarked boundary components of $\pS$. $\SP$ is called {\em totally marked} if $\cH = \emptyset$, and this context is the one originally used by Muller in \cite{Mu2012}.

A {\em $\cP$-arc} is a smooth map $a: [0,1]\to \Sigma$ such that $a(0), a(1) \in \cP$ and $a$ maps $(0,1)$ injectively into $\Sigma \setminus \cP$. A {\em quasitriangulation} of $\SP$, or {\em $\cP$-quasitriangulation}, is a collection $\D$ of $\cP$-arcs which cut $\Sigma$ into triangles and holed monogons. Associated to a quasitriangulation $\D$ is a {\em vertex matrix} $P$, which is an anti-symmetric $\D \times \D$ matrix with integer entries which we will utilize to define the {\em Muller algebra}, a quantum torus into which we will embed the skein algebra of a marked surface (See Section \ref{s.introtori}).

The algebra $\cS\SP$ is closely related to quantum Teichm\"uller space (see the following subsection) and the quantum cluster algebra of the surface \cite{Mu2012}. Two important facts about this algebra which we make extensive use of are that it is an Ore domain (meaning that it has a division algebra) and that the set $B_\SP$ of isotopy classes of essential $\cP \times (-1,1)$-tangles forms an $R$-basis for $\cS\SP$.

\subsection{Skein algebras of finite type surfaces}

A {\em finite-type surface} $\fS$ is an oriented surface $\fS= \ofS \setminus \cV$ where $\ofS$ is an oriented closed connected surface and $\cV$ is a (possibly empty) finite set whose elements are called {\em punctures}. A {\em framed link} in $\fS$ is a disjoint finite collection of diffeomorphic copies of $S^1$ in $\fS \times (-1,1)$ equipped with a smooth normal vector field.

Our base ring for skein algebras of finite type surfaces will always be $R=\BC$. Given a nonzero complex number $\xi:=q$, the Kauffman bracket skein module $\cS:=\cS_\xi(\fS)$ is the $\BC$-vector space freely spanned by isotopy classes of framed links in $\fS \times (-1,1)$ modulo the Kauffman bracket skein relations depicted in the left box of Figure \ref{fig:figures/skein2} with quantum parameter $q$ set to $\xi$. Here the framing is pointing out of the page at the reader.

$\cS$ is made into a $\BC$-algebra by introducing a product as follows. Given two framed links $K,L$, we define their product $KL \in \cS$ by stacking $K$ above $L$ in $\fS \times (1,1)$. This product was first introduced by Turaev \cite{Tu1991} in connection with the Atiyah-Bott-Weil-Petersen-Goldman symplectic structure of the character variety. \no{The algebra $\cS$ is closely related to the quantum Teichm\"uller space of the surface via the quantum trace map of Bonahon and Wong \cite{BW2011}, see Subsection \ref{sec.introembed}.}

Given a finite type surface $\fS:=\ofS \setminus \cV$ with negative Euler characteristic and at least one puncture, an {\em ideal triangulaton} $\fT$ may be defined, which is a maximal collection of isotopy classes of smooth arcs $x:[0,1] \to \ofS$ such that $x(0),x(1) \in \cV$ and $x$ embeds $(0,1)$ into $\fS$. From this datum an antisymmetric matrix known as the {\em face matrix} $Q_\fT$ may be constructed, and the quantum torus $\BT(Q_\fT)$ is a version of the quantum Teichm\"uller space of the surface called a Chekhov-Fock algebra into which the skein algebra may embed via the quantum trace map of Bonahon and Wong \cite{BW2011} (see Subsection \ref{sec.introembed}). This is analogous to the Muller algebra embedding of the skein algebra of a marked surface, and in \cite{Le2017} L\^e works with a more general class of surfaces to show that these embeddings are essentially ``Fourier transforms'' of one another.

As was shown in \cite{PS2000}, a $\BC$-basis $B_\fS$ for $\cS$ is given by the collection of all {\em simple diagrams} on $\fS$, which are finite collections of disjoint diffeomorphic copies of $S^1$ in $\fS$, none of which bound a disk. Each simple diagram gives rise to a framed link with vertical framing. Other important properties of $\cS$ for us include that $\cS$ is a domain \cite{PS2018}, that $\cS$ is finitely generated as a module over its center when the quantum parameter is a root of unity \cite{FKL2017}, and that $\cS$ is a finitely generated $\BC$-algebra \cite{Bu1997}.

\begin{remark}
The skein algebras of both marked surfaces and finite type surfaces may be defined in terms of the skein module of a marked 3-manifold.
\end{remark}

\section{Quantum torus methods}\lbl{s.introtori}

Here we give an overview of our primary lens into the structure of skein modules: quantum tori. These are a sort of algebra of $q$-commuting Laurent polynomials.

The following notion will be useful throughout. Suppose $A$ is an $R$-algebra, not necessarily commutative. Two elements $x,y \in A$ are said to be {\em $q$-commuting} if there exists $\cC(x,y) \in \ZZ$ such that $xy=q^{\cC(x,y)}yx$. Suppose $x_1,x_2,\ldots,x_n \in A$ are pairwise $q$-commuting. Then the {\em Weyl normalization} of $\prod_i x_i$ is defined by
\be\lbl{e.introweyl}
[x_1x_2 \ldots x_n]:=q^{-\frac{1}{2}\sum_{i<j}\cC(x_i,x_j)}x_1x_2 \ldots x_n.
\ee
It is easy to show that the normalized product does not depend on the order, i.e. given an element $\sigma$ of the symmetric group on $n$ letters, $[x_1x_2 \ldots x_n] = [x_{\sigma(1)}x_{\sigma(2)} \ldots x_{\sigma(n)}]$.

\subsection{Quantum tori}

Let $I$ be a finite set. We write $\Mat(I \times I, \BZ)$ for the set of $I \times I$ matrices with entries in $\BZ$. That is, $U \in \Mat(I \times I,\BZ)$ is a function $U: I \times I \to \BZ$. We write $U_{ij}$ for $U(i,j)$. We may occasionally think of ordinary $p \times p$ matrices in this manner without further comment.

Let $U \in \Mat(I \times I,\BZ)$ be antisymmetric, i.e. $U_{ij}=-U_{ji}$. The {\em quantum torus associated to $U$ over $R$} is
\[
\BT_q(U;R):=R \la x_i^{\pm 1}, i \in I \ra / (x_ix_j = q^{U_{ij}}x_jx_i). 
\]
$\BT(U;R)$ is sometimes called the algebra of skew Laurent polynomials or twisted Laurent polynomials.

When $R$ is Noetherian, $\BT(U;R)$ is Noetherian, an Ore domain, and many other nice properties.

Let $\BZ^I$ be the set of all maps $\bk: I \to \BZ$. For $\bk \in \BZ^I$ we define the {\em normalized monomial} to be
\[
X^\bk := [\prod_{i \in I}x_i^{\bk(i)}].
\]
Here the brackets denote the Weyl normalization as in ~\eqref{e.introweyl}, and as such the ordering of the product does not matter.

We also work with an important class of subalgebras of quantum tori known as monomial subalgebras. Let $\BT:=\BT(U;\BC)$ be a quantum torus where $U$ is a $p \times p$ matrix for some positive integer $p$. Let $\La \subset \BZ^p$ be a \no{finitely generated} submonoid of $\BZ^p$. Then the {\em monomial subalgebra associated to $\La$}, $\BA(\La) \subset \BT$, is defined to be the $\BC$-subspace of $\BT$ spanned by $\{ X^\bk:\bk \in \La\}$.

\subsection{Embedding of skein algebra}\lbl{sec.introembed}

\no{In this subsection our base ring is $R=\BC$.}

In \cite{BW2011}, Bonahon and Wong opened a portal to the powerful quantum torus methods for the study of skein algebras of surfaces with their quantum trace map. Let $\fS:=\ofS \setminus \cV$ be a finite type surface with negative Euler characteristic and at least one puncture, and $\fT$ an ideal triangulaton with associated face matrix $Q_\fT$. The {\em quantum trace map} is an embedding $\varkappa_\fT: \cS_q(\fS) \to \BT_q(Q_\fT)$.

$\BT(Q_\fT)$ is known as a Chekhov-Fock algebra \cite{ChFo1999}, whose division algebra (a localized version of the algebra such that every nonzero element is invertible) is a version of quantum Teichm\"uller space that can be found as the quantization of the shear coordinates on the Teichm\"uller space of the surface. This fact, along with the fact that there are {\em transfer isomorphisms} between the Chekhov-Fock algebra associated to each ideal triangulation, allows one to think of $\cS(\fS)$ (or rather its division algebra) as a ``coordinate-free'' version of quantum Teichm\"uller space.

Working in the context of this embedding vastly simplifies many aspects of the study of skein algebras. One immediate computational simplification that is apparent is that resolving the skein relation on $n$ crossings has exponential complexity, but working inside of a quantum torus this computation is reduced to polynomial complexity. This suggests additional reasons to look in this direction to study skein-theoretic algorithms for topological quantum computing beyond what we mention in Subsection \ref{sec.intromaxcons}.

The quantum trace map is mentioned here for historical completeness, but we do not make direct use of it in this work (though many of our results could, in principle, be proven using this method). Instead, we make use of a quantization of Penner coordinates of decorated Teichm\"uller space\cite{Pe2012} known as the Muller algebra, which has the advantage of having simpler change-of-coordinate maps between quasitriangulations than the Chekhov-Fock algeras have between ideal triangulations.

Let $\SP$ be a marked surface with at least one marked point, and let $\D$ be a $\cP$-quasitriangulation. A $\D \times \D$ antisymmetric matrix $P:=P_\D$ known as the {\em vertex matrix} is formed from the combinatorial data of $\D$. Define the {\em Muller algebra} to be 
\be 
 \sX_q(\D) = R \la a^{\pm 1}, a \in \D \mid ab = q^{P_{a,b}}\,  ba\ra,
 \lbl{eq.pre1}
 \ee
and its {\em positive subalgebra} to be
\[
\sX_{q,+}(\D) = R \la a, a \in \D \mid ab = q^{P_{a,b}}\,  ba\ra.
\]
In the case where $\SP$ is totally marked, Muller \cite{Mu2012} defined an embedding $\vpD: \cS_q\SP \to  \sX_q(\D)$ known as the {\em skein coordinate map}, which may be thought of as a cousin of the quantum trace map. Indeed, the quantum trace map and skein coordinate map are related by a ``shear-to-skein'' map explained by L\^e in \cite{Le2017} (see Section \ref{sec.sheartoskein}). Again, $\cS\SP$ embeds as an essential subalgebra of $\sX(\D)$, and the exponential complexity of the resolution of skein relations in $\cS\SP$ is reduced to polynomial complexity in $\sX(\D)$. 

One of the results in this work is a generalization of Muller's skein coordinate map to the case where not every boundary component has a marked point, see Theorem \ref{thm.2} in the results section below.

Again, the apparent disconnect here between finite type and marked surfaces is an illusion - ``generalized marked surfaces'' may be defined for which both the quantum trace map and skein coordinate map exist, and this idea is explored in \cite{Le2017}, but we do not need to make use of such generality in our work here.

\subsection{Associated graded algebra}\lbl{sec.intrograded}

\no{In this subsection the base ring is $R=\BC$ and the quantum parameter $q=\xi$ is a root of unity.}

Besides the skein algebra embeddings described in Subsection \ref{sec.introembed}, we also make use of another quantum torus method in proving our main results. Namely, we filter our skein algebra with some appropriate filtration, note that the associated graded algebra is some subalgebra of a quantum torus, prove the result there, and then lift the result back to the skein algebra. This method is utilized to great effect in works such as \cite{AF2017,Mu2012,Le2015,Le2018}.

Let $\fS$ be a finite type surface with negative Euler characteristic and at least one puncture equipped with an ideal triangulation $\fT$. Then one may form an $\BN$-filtration $\{F_n\}$ on $\cS:=\cS(\fS)$ as follows. Given $x \in \cS$ we may write $x$ as a finite sum $x = \sum c_i K_i \in \cS$ with $c_i \in \BC$, $K_i \in B_\fS$. Then we have $x \in F_n$ if $\sum_{a \in \fT} \mu(a,K_i) \leq n$ for all $i$. Here $\mu(a,K_i)$ is the minimum of $|a \cap K_i|$ over all isotopic copies of $K_i$.

From $\{F_n\}$ we can form the associated graded algebra $\gr \cS$. We show that $\gr \cS$ is a finitely generated monomial subalgebra of a quantum torus, and from this we are able to prove results about $\gr \cS$ that are then lifted to results on $\cS$, as is the usual reason for working with an associated graded algebra.

\section{Original Results}\lbl{s.introresults}

In this section we summarize all original results in this work. Our most important results are the following.
\begin{enumerate}[(i)]
\item The existence of the Chebyshev-Frobenius homomorphism between skein modules of marked 3-manifolds (Subsection \ref{s.introcf}).
\item The image of the Chebyshev-Frobenius homomorphism consists of (skew-)transpent elements (Subsection \ref{s.introcenter}).
\item The construction of a surgery theory for skein algebras of marked surfaces (Subsection \ref{sec.introsurgery}).
\item Skein algebras of surfaces are maximal orders (Subsection \ref{s.introorder}).
\end{enumerate}

\subsection{Skein algebra embedding}

\begin{theorem} [See Theorem \ref{r.torus}]\lbl{thm.2}
Assume $\D$ is a quasitriangulation of a marked surface $\SP$. Then there is a natural algebra embedding 
\[
    \vpD: \cS_q\SP \embed \sX_q(\D)
\]
such that $\vpD(a)= a$ for all $a\in \D$. The image of $\vpD$ is sandwiched between $\sX_{+}(\D)$ and $\sX(\D)$. The algebra $\Sx\SP$ is an Ore domain, and $\vpD$ induces an isomorphism $ \tvD: \widetilde \Sx\SP \overset \cong \longrightarrow \tiX(\D)$ where $\widetilde \Sx\SP $ is the division algebra of $ \Sx\SP$ and $\tiX(\D)$ is the division algebra of $\sX(\D)$.
\end{theorem}
In the case where $\SP$ is totally marked, Theorem \ref{thm.2} was proved by Muller \cite{Mu2012}.

\subsection{Skein algebra of a marked surface is a domain}

In \cite{Mu2012}, Muller showed that the skein algebra of a totally marked surface $\SP$ is a domain via a mostly geometric argument that utilizes one non-geometric choice, namely an arbitrary total ordering on the set of all essential $\cP$-tangles. That the skein algebra of a marked surface is a domain is a simple corollary of this result via functoriality, see Lemma \ref{s.markeddomain}. However, in Appendix \ref{a.domain} we give an alternative purely geometric proof of this fact that utilizes a more algebro-geometric ordering on the set of essential $\cP$-tangles but is admittedly more laborious than the proof presented in \cite{Mu2012}.

\subsection{Chebyshev-Frobenius homomorphism}\lbl{s.introcf}

Let $\SP$ be a marked surface and let the base ring be $R=\BC$ for this subsection. Set the quantum parameter $q=\xi$, where $\xi$ is a primitive root of unity, let $N:=\ord(\xi^4)$ (i.e. the smallest $N$ such that $\xi^{4N}=1$), and let $\ve = \xi^{N^2}$ (this implies that $\ve = \pm 1, \pm i$.). When $\cP=\emptyset$, utilizing the quantum trace map Bonahon and Wong \cite{BW2016-2} constructed a remarkable map $\cS_\ve\SP \to \cS_\xi\SP$ called the Chebyshev homomorphism, which plays an important role in the theory of the skein algebra such as in the conjectural construction of a positive basis \cite{Th2014,Le2016}. This map is given by applying the $N$th Chebyshev polynomial of type one~\eqref{e.introfrob} to knots, which turn out to satisfy the skein relation. Satisfying the skein relation requires a number of difficult to understand ``miraculous cancellations''. One main result of this dissertation is to extend Bonahon and Wong's Chebyshev homomorphism to skein modules of marked 3-manifolds and to give a conceptual explanation of its existence from basic facts about $q$-commutative algebra stemming from the much simpler to understand Frobenius homomorphism~\eqref{e.introfrob} between quantum tori (see Subsection \ref{sec.introchebycom}).

Let $\MN$ be a marked 3-manifold. A 1-component $\cN$-tangle $T$ is  diffeomorphic to either the circle $S^1$ or the closed interval $[0,1]$. For a 1-component $\cN$-tangle $T$ and an integer $k\ge 0$, write $T^{(k)} \in \cS_\xi\MN$ for the {\em $k$th framed power of $T$} obtained by stacking $k$ copies of $T$ in a small neighborhood of $T$ along the direction of the framing of $T$. Given a polynomial $P(z) = \sum c_i z^i \in \BZ[z]$, the {\em threading} of $T$ by $P$ is given by $P^{\fr}(T) = \sum c_i T^{(i)} \in \cS_\xi\MN$. 

The Chebyshev polynomials of type one $T_n(z) \in \BZ[z]$ are defined recursively as
\be\lbl{e.introcheb}
T_0(z)=2, \ \ \ T_1(z)=z, \ \ \ T_n(z) = zT_{n-1}(z)-T_{n-2}(z), \ \ \forall n \geq 2.
\ee

The extension of Bonahon and Wong's result to marked 3-manifolds is the following.

\begin{theorem}\lbl{thm.1} (See Theorem \ref{t.ChebyshevFrobenius})

Suppose $(M,\cN)$ is a marked 3-manifold and $\xi$ is a complex root of unity. Let $N=\ord(\xi^4)$. Let $\ve =\xi^{N^2}$. Then there exists a unique $\BC$-linear map $\Phi_\xi: \cS_\ve(M,\cN) \to \cS_\xi(M,\cN)$ such that for any $\cN$-tangle $T = a_1 \cup \cdots \cup a_k \cup \al_1 \cup \cdots \cup \al_l$ where the $a_i$ are $\cN$-arcs and the $\al_i$ are $\cN$-knots,
\begin{align*}
\Phi_\xi(T) & = a_1^{(N)} \cup \cdots \cup a_k^{(N)} \cup (T_N)^{\fr}(\al_1) \cup \cdots \cup (T_N)^{\fr}(\al_l) \quad \text{in }\cS_\xi\MN \\
& : = \sum_{0\le j_1, \dots, j_l\le N} c_{j_1} \dots c_{j_l}  a_1^{(N)}
 \cup \dots \cup  a_k^{(N)} \cup \, \al_1^{(j_1)} \cup \cdots \cup \al_l^{(j_l)} \quad \text{in }\cS_\xi\MN,
\end{align*}
where $T_N(z) = \sum_ {j=0}^N c_j z^j$ are the Chebyshev polynomials of type one.
\end{theorem}

We call $\Phi_\xi$ the {\em Chebyshev-Frobenius homomorphism}. Our construction and proof are independent of the previous results of \cite{BW2016-2} which requires the quantum trace map of \cite{BW2011}, and \cite{Le2015} which gives a skein-theoretic approach to the existence of the Chebyshev homomorphism. This is true even for the case $\cN=\emptyset$. 

Note that when $T$ has only $\cN$-arc components, then $\Phi_\xi(T)$ is much simpler as it can be defined using monomials and no Chebyshev polynomials are involved. The main strategy we employ is to understand the Chebyshev-Frobenius homomorphism for this simpler case, then show that the $\cN$-knot components case can be reduced to this case. We summarize this strategy as follows.

Suppose $\SP$ is a {\em triangulable} marked surface, i.e. every boundary component of $\Sigma$ has at least one marked point and $\SP$ has a quasitriangulation. In this case a quasitriangulation is called  simply a {\em triangulation}. Let $\D$ be a triangulation of $\SP$.

Let $\xi$ be a non-zero complex number (not necessarily a root of unity), $N$ be a positive integer, and $\ve=\xi^{N^2}$. From the presentation \eqref{eq.pre1} of the Muller algebras $\sX_\ve (\D)$ and $\sX_\xi (\D)$, one sees that there is an algebra homomorphism, called the Frobenius homomorphism, uniquely defined by
\be\lbl{e.introfrob}
F_N : \sX_\ve (\D) \to \sX_\xi(\D), \quad\text{given by}\ F_N(a)= a^N \quad \forall a \in \D,
\ee
which is injective, see Proposition \ref{r.Frobenius}.

Identify $\cS_\nu \SP$ with a subset of $\sX_\nu(\D)$ for $\nu=\xi$ and $\nu=\ve$ via the embedding $\vpD$. Consider the following diagram. 
\be \lbl{dia.sx0}
\begin{tikzcd}
\cS_\ve \SP  \arrow[hookrightarrow]{r} \arrow[dashed,"?"]{d}& \sX_\ve(\D)  \arrow[d, "F_N"] \\
\cS_\xi \SP \arrow[hookrightarrow]{r} & \sX_\xi(\D)
\end{tikzcd}
\ee
We consider the following questions about $F_N$:

A. For what  $\xi \in \BC\setminus \{0\}$ and $N\in \BN$ does $F_N$  restrict to a map from $\cS_{\ve}\SP$ to $\cS_{\xi}\SP$ such that the restriction does not depend on the triangulation $\D$?

B.  If $F_N$ can restrict to such a map, can one define the restriction of $F_N$ onto $\cS_{\ve}\SP$ in an intrinsic way, not referring to any triangulation $\D$?

It turns out that the answer to Question A is that $\xi$ must be a root of unity, and $N$ must be the order of $\xi^4$, see Theorem \ref{thm.A}. Then Theorem \ref{thm.surface}, answering Question B, states that under these assumptions on $\xi$ and $N$, the restriction of $F_N$ onto $\cS_{\ve}\SP$ can be defined in an intrinsic way without referring to any triangulation. Namely, if $a$ is a $\cP$-arc, then $F_N(a)= a^N$, and if $\al$ is a $\cP$-knot, then $F_N(\al)= T_N(\al)$. From these results and  the functoriality of homomorphisms between skein modules we can prove Theorem \ref{thm.1}.

In Appendix \ref{app.simple}, we give a direct proof of the existence of the Chebyshev-Frobenius homomorphism for the disk with four marked points without utilizing the Muller algebra method. This gives additional geometric intuition about this homomorphism. One suggested exercise for the reader is to perform this task for the punctured torus making use of the quantum trace map, to convince oneself that this result also works for finite type surfaces.

\subsection{Surgery}\lbl{sec.introsurgery} One important consequence of allowing of unmarked boundary components for the skein coordinate map is that we can develop a surgery theory. In this subsection we require the base ring $R$ to be Noetherian.

Let $\SP$ be a marked surface. One can consider the embedding $\vpD: \cS\SP \embed \sX(\D)$ as a coordinate system of the skein algebra which depends on the quasitriangulation $\D$. A function on $\cS\SP$ defined through the coordinates makes sense only if it is independent of the coordinate systems. To help with this independence problem (specifically, Question B from the previous subsection) we develop a surgery theory of coordinates in Chapter \ref{c.surgery} which describes how the coordinates change under certain modifications of the marked surfaces. We consider two such modifications: one is to add a marked point and the other one is to plug a hole, (i.e. glue a disk to a boundary component with no marked points). Note that the second one changes the topology of the surface, and is one of the reasons why we want to extend Muller's result to allow unmarked boundary components. Besides being a crucial tool for proving the existence of the Chebyshev-Frobenius homomorphism of Theorem \ref{thm.1}, we think our surgery theory will find more applications elsewhere.

\subsection{Centrality and (skew-)transparency}\lbl{s.introcenter}

\no{With the help of Theorem \ref{thm.2} we are able to determine the center of $\cS_q\SP$ for generic $\xi$.}
\no{
\begin{theorem}[See Theorem \ref{thm.center1}] Suppose $\SP$ is a marked surface with at least one quasitriangulation and $q$ is not a root of unity. Then the center of $\cS\SP$ is the $\BC$-subalgebra generated by $z_\beta$ for each connected component $\beta$ of $\pS$. Here if $\beta \cap \cP=\emptyset$ then $z_\beta= \beta$, and if $\beta \cap \cP\neq \emptyset$ then $z_\beta$ is the product of all $\cP$-arcs lying in $\beta$.
\end{theorem}
}

Let $\SP$ be a marked surface and let the base ring be $R=\BC$.

The center of an algebra is important, for example, in questions about the representations of the algebra. When $\xi$ is a root of unity and $\cP=\emptyset$, the center of $\Sx(\Sigma, \emptyset)$ is determined in \cite{FKL2017} and is instrumental in proving the main result there, the unicity conjecture of Bonahon and Wong which we refine in this work (see Subsection \ref{sec.intromaxcons}).\no{ In a subsequent paper we will determine the center of $\Sx\SP$ for the case when $\xi$ is a root of unity.} In \cite{PS2018}, the center of a {\em reduced} skein algebra, which is a quotient of $\Sx\SP$, was determined for the case when $\xi$ is not a root of unity.

A notion more general than being central in the skein algebra of a surface that also makes sense for skein modules are {\em transparent} and {\em skew-transparent} elements of skein modules of marked 3-manifolds.  Informally, an element $x \in \cS_\xi\MN$ is transparent if passing a strand of an $\cN$-tangle $T$ through $x$ does not change the the value of the union $x \cup T$.  

We generalize the result in \cite{Le2015} for unmarked 3-manifolds to obtain the following theorem. 

\begin{theorem} (See Theorem \ref{r.mutransparent})
Suppose $\MN$ is a marked 3-manifold, $\xi$ is a root of unity, $N=\text{ord}(\xi^4)$, and $\ve=\xi^{N^2}$. Then $\xi^{2N} = \pm 1$.  Let $\Phi_\xi: \cS_\ve\MN \to \cS_\xi\MN$ be  the Chebyshev-Frobenius homomorphism. Then the image of $\Phi_\xi$ is transparent in the sense that if $T_1,T_2$ are $\cN$-isotopic $\cN$-tangles disjoint from another $\cN$-tangle $T$, then in $\cS_\xi\MN$ we have
$$ \Phi_\xi(T) \cup T_1 = \Phi_\xi(T) \cup T_2.$$
\end{theorem}
Note that since $\text{ord}(\xi^4)=N$, we have either $\xi^{2N}=1$ or $\xi^{2N}=-1$. When $\xi^{2N}=-1$, the corresponding result is that the image of $\Phi_\xi$ is {\em skew-transparent}, see 
Theorem \ref{r.mutransparent}. Theorem \ref{r.mutransparent} can be depicted pictorially as the identity in Figure \ref{fig:figures/transparent}.

\FIGc{figures/transparent}{Applying the Chevyshev-Frobenius homomorphism to an $\cN$-tangle makes it transparent or skew-transparent.}{1.5cm}

We find another application of the skein coordinate map in the following theorem about the center of the skein algebra of a marked surface when the quantum parameter $q$ is not a root of unity.

Let the base ring $R$ be a commutative domain (not necessarily Noetherian) with distinguished invertible element $q^{1/2}$ and fix a marked surface $\SP$. Let $\cH$ denote the set of all unmarked components in $\pS$ and $\cHd$ the set of all marked components.  If $\beta\in \cH$ let $z_\beta=\beta$ as an element of $\cS$. If $\beta\in \cHd$ let $z_\beta = \left [ \prod a \right] \in \cS$, where the product is over all boundary $\cP$-arcs in $\beta$. Here the brackets denote the Weyl normalization, see~\eqref{e.introweyl}.

\begin{theorem}(See Theorem \ref{thm.center1})
Suppose $\SP$ is a marked surface. Assume that $q$ is not a root of unity. Then the center $Z(\cS\SP)$ of $\cS\SP$ is the $R$-subalgebra generated by $\{z_\beta \mid  \beta \in\cH\cup \cHd\}$.
\end{theorem}

We prove this result by embedding $\cS\SP$ into a Muller algebra and studying the center there.

\subsection{Maximal order}\lbl{s.introorder}

Our last main result in the following.

\begin{theorem}\lbl{r.intromax}(See Theorem \ref{r.mainresult})
Let $\fS$ be a finite-type surface with at least one puncture, and let $q \in \BC^\times$ be a root of unity. Then $\cS_\xi(\fS)$ is a maximal order.
\end{theorem}

A maximal order is a noncommutative generalization of an integrally closed domain. In particular, the center of any maximal order is integrally closed. There are several closely related definitions of maximal order (see e.g. \cite{MRS2001}), but here we use the following one which coincides with several other definitions in our specific case.

Let $A$ be an associative $\BC$-algebra that (i) is a domain, and (ii) is finitely generated as a module over its center. Let $Z$ be the center of $A$, and $\tZ$ the field of fractions of $Z$. Then $A$ is a {\em maximal order} if given any ring $B \subset A \otimes_Z \tZ$ such that $A \subset B \subset (z^{-1})A$ for some $0 \neq z \in Z$ then $A=B$. Here we emphasize that $(z^{-1})A$ is a set, not a localization of $A$. We note that under these conditions, $A \otimes_Z \tZ$ is always a division algebra. It is easy to show that for commutative rings this definition is equivalent to being an integrally closed domain.

When the Euler characteristic is negative, our strategy to prove Theorem \ref{r.intromax} is to use the filtration on $\cS:=\cS_\xi(\fS)$ defined in Subsection \ref{sec.intrograded} in terms of an ideal triangulation and work with the associated graded algebra $\gr \cS$. We first prove that $\gr \cS$, which is a finitely generated monomial subalgebra of a quantum torus, is a maximal order. We then apply Proposition \ref{r.gradedorderlift} which shows that, under our circumstances, when the associated graded algebra is a maximal order then so is the filtered algebra. Proposition \ref{r.gradedorderlift} is probably known to experts, but does not appear in the literature to our knowledge.

The special cases where the Euler characteristic of the surface is non-negative have simple direct proofs which we do not present here. Furthermore, we remark that it is straightforward to extend Theorem \ref{r.mainresult} to the reduced skein algebra as described by \cite{PS2018} and the skein algebra of marked surfaces as described in Subsection \ref{s.introsurface}. \no{and L\^e's stated skein algebra as described in \cite{Le2018}.}

\begin{remark}
We also believe it is true that the skein algebra of a closed surface is a maximal order. We do not have a proof of this at the time of publication, but the only tool needed to do so is a good $\BN^p$-filtration on the skein algebra of the closed surface. Our Proposition \ref{r.gradedorderlift}, which works for $\BN^p$-filtrations, is more than sufficient for the case of surfaces with at least one puncture because $\{F_n\}$ is merely an $\BN$-filtration, but we present it here nonetheless with the closed surface case in mind.
\end{remark}

In the course of proving Theorem \ref{r.mainresult}, we obtain another simple result.

\begin{proposition}[See Proposition \ref{r.ALamaximal}]
Let $p$ be a non-negative integer, and let $U$ be a $p \times p$ antisymmetric matrix with integer entries. Let $\La \subset \BZ^p$ be a finitely generated submonoid of $\BZ^p$ of rank $p$ that is primitive in its group completion. Then the monomial subalgebra $\BA(\La) \subset \BT(U;\BC)$ is a maximal order.
\end{proposition}
$\La$ is {\em primitive} in its group completion $\uLa$ if given $c\bk \in \La$ for some positive integer $c$ and $\bk \in \uLa$, then $\bk \in \La$ as well. The case where $\La = \BN^p$ was already known in \cite{DCP1993}, and $\La = \BZ^p$ was present there implicitly.

\subsection{Consequences of maximal order}\lbl{sec.intromaxcons}

That $\cS(\fS)$ is a maximal order has several important consequences:
\begin{enumerate}[(i)]
\item The variety of classical shadows $\YxiS$ is normal.
\item The $SL_2(\BC)$ character variety $\sX(\fS)$ is normal.
\item A refined unicity theorem.
\end{enumerate}

We write $\Sp(\cS_q(\fS))$ for the set of equivalence classes of finite-dimensional irreducible representation of $\SxiS$ over $\BC$. Let $\ZxiS$ be the center of $\SxiS$. $\Sp(\ZxiS)$ turns out to be an affine variety which we denote by $\YxiS$, called the {\em variety of classical shadows} in \cite{FKL2017}. Since the center of a maximal order is integrally closed, we have the following.

\begin{theorem}(See Theorem \ref{s.shadow})
For every finite type surface $\fS$ with at least one puncture and every root of unity $\xi$ the classical shadows variety $\YxiS$ is a normal affine variety.
\end{theorem}
A short argument gives the following corollary.
\begin{corollary} For every finite type surface $\fS$, the $SL_2(\BC)$-character variety $\sX(\fS)$ is a normal affine variety.
\end{corollary}
This corollary was originally shown by Simpson in \cite{Si94-1,Si94-2} by a lengthy and difficult argument. That we are able to reobtain this result with our short argument is illustrative of the power of the methods of quantum topology for proving results about classical topology.

Another consequence of Theorem \ref{r.intromax} is a refinement of the unicity theorem first proved in \cite{FKL2017}.

Suppose $\xi$ is a root of unity.\no{ Let $N =\ord(\xi^4)$.} The {\em degree} $\deg(\SxiS)$ is defined to be the square of the dimension of the division algebra $\widetilde {\SxiS}$ over its center $\widetilde{\ZxiS}$, which is a field. The degree $\deg(\SxiS)$ was first calculated in \cite{FKL2019}.
Since $\SxiS$ is a maximal order, by applying a theorem of de Concini, Kac, and Procesi \cite{DKP93}, we get the following refinement of the unicity theorem of \cite{FKL2017}.
\begin{theorem}(See Theorem \ref{r.refinedunicity})
Suppose $\fS$ is a finite type surface of negative Euler characteristic with at least one puncture and $\xi$ is a root of unity.

(a) The variety of shadows $\YxiS$ parameterizes the $D(\fS, \xi)$-dimensional semisimple representations of $\SxiS$.

(b) The central character map $\chi: \Sp(\SxiS) \to \YxiS$ is surjective and each fiber consists of all those irreducible representations of $\SxiS$ which are irreducible components of the corresponding semisimple representation.  In particular, each irreducible representation of $\SxiS$ has dimension $\leq D(\fS, \xi)$. A point $a\in \YxiS$ is called generic if the dimension of the irreducible representation $\chi^{-1}(a)$ is $D(\fS, \xi)$.

(c) The set of generic points in  $\YxiS$ is a non-empty Zariski open dense subset of $\YxiS$.

(d) Let $\cU \subset \YxiS$ be the open and dense subset consisting of all points $u$ such that $|\chi^{-1}(u)| = 1$. Let $x \in \YxiS \setminus \cU$ and write $\chi^{-1}(x) = \{r_1,\ldots,r_s\}$. Then since $\sum_{i=1}^s \dim(r_i) = D(\fS,\xi)^2$, in particular we have that $s \leq D(\fS,\xi)$.
\end{theorem}

Finally, the author presents a possible application of Theorem \ref{r.intromax} to topological quantum compiling in Subsection \ref{s.tqc}. Briefly stated, skein algebras of surfaces encode some aspects of the dynamics of non-abelian anyons in certain topological quantum computer architectures \cite{Fr1998,Ki2003,Ka2002,Wa2010}. An algorithm for the asymptotically optimal compilation for a single qubit for Fibonacci anyons, $SU(2)_k$ anyons, and others, is presented in \cite{KY2015} and this algorithm cruicially utilizes maximal orders in quaternion algebras. We conjecture in Subsection \ref{s.tqc} that their algorithm may be applied to representations of skein algebras of marked surfaces under certain circumstances. Other results in this dissertation, such as the (skew-)transparent image of the Chebyshev-Frobenius algorithm, and the use of quantum torus methods in general, may also find use in this context.

\subsection{Chebyshev polynomial for \texorpdfstring{$q$}{q}-commuting variables}\lbl{sec.introchebycom}
In the course of proving the existence of the Chebyshev-Frobenius homomorphism, we obtain the following simple but useful result given in Proposition \ref{r.chebygen} relating Chebyshev polynomials and $q$-commuting variables.

\begin{proposition}[See Proposition \ref{r.chebygen}]
Suppose $K,E$ are variables and $q$ an indeterminate such that $KE=q^2EK$ and $K$ is invertible. Then for any $n \ge 1$,
$$
T_n(K+K^{-1}+E) = K^n + K^{-n} + E^n + \sum_{r=1}^{n-1}\sum_{j=0}^{n-r} c(n,r,j)[E^rK^{n-2j-r}],
$$
where $c(n,r,j) \in \BZ[q^{\pm 1}]$ is given explicitly by  \eqref{eq.cnrj} as a ratio of $q$-quantum integers, and in fact $c(n,r,j) \in \BN[q^{\pm 1}]$. Besides, if $q^2$ is a root of unity of order $n$, then $c(n,r,j)=0$.  
\end{proposition}

In particular, if $q^2$ is a root of unity of order $N$, then
$$ 
T_N(K+K^{-1}+E) = K^N + K^{-N} + E^N.
$$
The above identity was first proven in \cite{BW2011}, as an important case of the calculation of the Chebyshev homomorphism. See also \cite{Le2015}. An algebraic proof of identities of this type is given in \cite{Bo2017}. Here the identity follows directly from Proposition \ref{r.chebygen}. The new feature of Proposition \ref{r.chebygen} is that it deals with generic  $q$ and may be useful in the study of the positivity of the skein algebra, see \cite{Th2014,Le2018-2}.

\section{Outline}

We give here a brief summary of each chapter in this dissertation.

\begin{enumerate}
    \setcounter{enumi}{+1}
    \item \textbf{Topology} We review the topological prerequisites to our results. We discuss finite-type surfaces, marked surfaces, marked 3-manifolds, $\cN$-tangles, simple diagrams, quasitriangulations, ideal triangulations, hyperbolic structures, and character varieties.
    \item \textbf{Algebra} We review the algebraic prerequisites to our results. This includes Chebyshev polynomials, filtered and graded algebras, division in noncommutative rings, maximal orders, and quantum tori.
    \item \textbf{Kauffman bracket skein modules} We explain basic aspects of Kauffman bracket skein modules of 3-manifolds and algebras of surfaces, and discuss how the skein algebra is a quantization of the character variety and the special role of the skein algebra at 4th roots of unity.
    \item \textbf{Quantum Teichm\"uller theory} We summarize a few aspects of quantum Teichm\"uller theory, most importantly we discuss the Muller algebra into which the skein algebra embeds.
    \item \textbf{Skein algebra embeddings} We introduce the quantum trace map of Bonahon and Wong and the skein coordinate map of Muller, as well as its generalization to the context of marked surfaces which are not totally marked.
    \item \textbf{Surgery theory} We introduce a surgery theory for Kauffman bracket skein algebras of marked surfaces, an important theoretical tool for our proof of the Chebyshev-Frobenius homomorphism.
    \item \textbf{Chebyshev-Frobenius homomorphism} We prove the existence of the Chebyshev-Frobenius homomorphism of marked 3-manifolds and discuss its consequences for marked and finite type surfaces.
    \item \textbf{Center of $\cS\SP$ and (skew-)transparency} We utilize the Chebvyshev-Frobenius homomorphism to identify (skew-)transparent elements in skein modules of marked 3-manifolds. We also utilize the Muller algebra to identify the center of $\cS\SP$ when $q$ is not a root of unity.
    \item \textbf{$\cS(\fS)$ is a maximal order} We prove that $\cS$ is a maximal order when and discuss the consequences of this result. We also present a conjecture on the relationship of this result to topological quantum compiling.
    \item \textbf{Appendix A: Geometric proof that $\cS$ is a domain} We present a purely geometric argument that the skein algebra of a marked surface is a domain.
    \item \textbf{Appendix B: Chebyshev-Frobenius homomorphism for the marked disk} We give a direct skein-theoretic proof of the existence of the Chebyshev-Frobenius homomorphism for the disk with four marked points.
\end{enumerate}

%% file: surfacetopology.tex
\chapter{Topology}

\section{Introduction}

In this dissertation we are concerned primarily with 2- and 3-dimensional manifold theory as it is revealed via the techniques of quantum topology. Our lens into 3-manifold theory will primarily be via the embedding of tangles and surfaces into 3-manifolds, thus we begin our discussion with a short review of relevant aspects of surface and 3-manifold topology.

Closed compact surfaces have been completely classified up to homeomorphism since the early 1900's with the Euler characteristic acting as a complete invariant. Similarly, surfaces with boundary and/or punctures are classified up to homeomorphism by their genus, the number of boundary components, and the number of punctures. Surfaces with infinite genus, number of boundary components, or number of punctures are beyond the scope of this dissertation and will not be discussed.

We begin in Section \ref{sec.surfacetypes} by fixing definitions for the various types of surfaces encountered in this work, namely finite type surfaces and marked surfaces. In Section \ref{sec.top3} we introduce marked 3-manifolds, the spaces upon which our primary object of study, Kauffman bracket skein modules, depends upon functorially. In Section \ref{sec.surfacetangles} we define the notion of tangles in surfaces and 3-manifolds, which are collections of knots and arcs considered up to isotopy that act as generators of skein modules. In Section \ref{sec.tangle3} we define combinatorial models for surfaces in the forms of ideal triangulations for finite type surfaces and quasitriangulations for marked surfaces. This will give us scaffolding to describe simple coordinate systems to describe tangles on surfaces. Lastly in Section \ref{s.charvariety}, knots will define a basis for the ring of regular functions on the $SL_2\CC$-character variety of the surface, of which we will understand the Kauffman bracket skein algebra to be a quanitzation. As we show in Chapter \ref{c.maxorder}, this noncommutative view on character varieties ultimately allows us to give a much shorter proof of the difficult result of Simpson \cite{Si94-1,Si94-2} that the character variety is normal - which is to say, its ring of regular functions is integrally closed. \no{Finally we end with a discussion on hyperbolic geometry of surfaces and Teichm\"uller spaces, which we will eventually quantize and then embed the skein algebra into. Quantum Teichm\"uller spaces are easy to study algebras known as quantum tori, and our main proof strategy is to work with these algebras instead of the skein algebras themselves. }

\section{Surfaces}\lbl{sec.surfacetypes}

A {\em marked surface} is a pair $\SP$, where $\Sigma$ is an oriented connected surface with (possibly empty) boundary $\pS$ and a possibly empty finite set $\cP\subset \pS$. A point in $\cP$ is called a {\em marked point}. A connected component of $\pS$ is {\em marked} if it has at least one marked point, otherwise it is {\em unmarked}. The set of all unmarked components is denoted by $\cH$ and the set of all marked components is denoted by $\cHd$. We call $\SP$ {\em totally marked} if $\cH=\emptyset$, i.e. every boundary component has at least one marked point. The marked points serve as endpoints for $\cP$-arcs, which will be found as elements of the skein algebra of a marked surface and used to construct quasitriangulations, see Subsection \ref{sec.quasitri}. Our marked surface is the same as the marked surface in \cite{Mu2012}, or a ciliated surface in \cite{FG2000}, or a bordered surface with punctures in \cite{FST2008} if one considers a boundary component without any marked points as a puncture.

A {\em finite type surface} $\fS$ is an oriented surface $\fS = \ofS \setminus \cV$ where $\ofS$ is an oriented closed connected surface and $\cV$ is a (possibly empty) finite set. A point $v \in \cV$ is called a {\em puncture}. The genus $g = g(\fS)$ and the number of punctures $p=|\cV|$ completely determines the diffeomorphism class of $\fS$, and so we write $\fS:=\fS_{g,p}$. We will utilize these punctures as endpoints for ideal triangulations, see Subsection \ref{sec.idealtri}. A finite type surface may also be defined as a marked surface with empty boundary (and thus also an empty marked set).

Our main results, namely the existence of the Chebyshev-Frobenius homomorphism and that skein algebras are maximal orders, apply to surfaces of either kind with no restriction. Thus the difference is ultimately immaterial and only a proof technique, where we utilize the type of surface for which the result is easiest to prove and derive the other as a corollary.

\section{3-manifolds}\lbl{sec.top3}

Our primary object of study, the Kauffman bracket skein module, is generated by framed tangles embedded in marked 3-manifolds. In the case of Kauffman bracket skein algebras of surfaces, the associated 3-manifold is a cylinder over the surface. Here we introduce these objects.

\subsection{Marked 3-manifolds}

A {\em marked 3-manifold $\MN$} is a pair consisting of an oriented connected 3-manifold $M$  with (possibly empty) boundary $\pM$ and a 1-dimensional oriented submanifold $\cN \subset \pM$ such that $\cN$ consists of a finite number of connected components, each of which is diffeomorphic to the interval $(0,1)$.

\subsection{Cylinders over surfaces}\lbl{sec.cylinders}

To every marked surface $\SP$ we associate the marked 3-manifold $(\Sigma \times (-1,1), \cP \times (-1,1))$. To every finite type surface $\fS$ we associate the marked 3-manifold $(\fS \times (-1,1), \emptyset)$. Skein algebras of surfaces will be defined as the skein module of the associated marked 3-manifold equipped with a product structure we are able to define due the presence of a vertical direction.

It is important for technical reasons that we use $(-1,1)$ instead of $[-1,1]$, as the inclusion of a boundary on the interval complicates some matters. Nonetheless, we expect our results hold when we use the closed interval instead.

\section{1-dimensional submanifolds of surfaces}\lbl{sec.surfacetangles}

In this section we introduce a class of 1-dimensional submanifolds in marked and finite type surfaces considered up to isotopy, which will form the generating set for Kauffman bracket skein algebras of surfaces as well as building materials for quasitriangulations of marked surfaces and ideal triangulations of finite type surfaces. We also define the geometric intersection index, an important technical tool.

Throughout this section we fix a marked surface $\SP$ and a finite type surface $\fS = \ofS \setminus \cV$.

\subsection{Homotopy vs. isotopy}

One bugbear eternally looming over topologists is the distinction between homotopy and isotopy. Luckily for us, these two things coincide in the case of surfaces.

\begin{theorem}\lbl{r.homoisiso}
Let $\Sigma$ be a second-countable connected orientable surface. If two homeomorphisms $f_1,f_2: \Sigma \to \Sigma$ are homotopic, then they are isotopic.
\end{theorem}

This result was first proven for closed surfaces by Baer \cite{Ba1927,Ba1928}, and then for non-compact surfaces with boundary by Epstein in \cite{Ep1966}. Thus, from here on out we will only refer to isotopy classes of 1-dimensional submanifolds of surfaces.

\subsection{\texorpdfstring{$\cP$}{P}-tangles}

A {\em $\cP$-tangle} is an immersion $T: C \to \Sigma$, where $C$ is a compact non-oriented 1-dimensional manifold, such that
\begin{enumerate}[(i)]
\item the restriction of $T$ onto the interior of $C$ is an embedding into $\Sigma \setminus \cP$, and
\item $T(\partial C) \subset \cP$.
\end{enumerate}
The image of a connected component of $C$ is called a {\em component} of $T$. When $T$ is homeomorphic to $S^1$, we call $T$ a {\em $\cP$-knot}. When $T$ is homeomorphic to a disjoint union of $S^1$'s, we call $T$ a {\em $\cP$-link}. When $C$ is homeomorphic to $[0,1]$, we call $T$ a {\em $\cP$-arc}.

Two $\cP$-tangles are {\em $\cP$-isotopic} if they are isotopic through the class of $\cP$-tangles.

For a ring $R$, we will write $\cT\SP$ for the free $R$-module of $\cP$-tangles considered up to isotopy. The ring $R$ under consideration will always be made clear.

\begin{remark}
We emphasize here that our tangles are not actually ``tangled'' due to the embedding requirement. We use this terminology because later on we will define $\cN$-tangles of marked 3-manifolds $\MN$ which may actually be ``tangled''. The Kauffman bracket skein algebra of a surface is defined in terms of a marked 3-manifold obtained by taking the cylinder over the surface and tangles inside of this cylinder, and these tangles will be generated by the ``unentangled tangles'' drawn on the surface.
\end{remark}

A $\cP$-arc $x$ is called a {\em boundary arc}, or $x$ is {\em boundary}, if it is $\cP$-isotopic to a $\cP$-arc contained in $\pS$.
A $\cP$-arc $x$ is called an {\em inner arc}, or $x$ is {\em inner}, if it is not a boundary arc.
 
A $\cP$-tangle $T \subset \Sigma$ is {\em essential} if it does not have a component bounding a disk in $\Sigma$; when such a component is homeomorphic to $S^1$ it is called a {\em smooth trivial knot} in $\Sigma \setminus \cP$, else it is a $\cP$-arc bounding a disk in $\Sigma$. By convention, the empty set is considered an essential $\cP$-tangle that is $\cP$-isotopic only to itself. We write $B_\SP$ for the set of all essential $\cP$-tangles considered up to isotopy.

\subsection{Loops and ideal arcs on finite-type surfaces}\lbl{sec.loopsfinitetype}

A {\em loop} in $\fS$ is an unoriented submanifold diffeomorphic to $S^1$. Given a smooth map $x: [0,1] \to \ofS$ such that $x(0),x(1) \in \cV$ and $x$ embeds $(0,1)$ into $\fS$ then the image of $x$ in $\fS$ is an {\em ideal arc}. Isotopies of ideal arcs are always considered to be within the class of ideal arcs, and similarly for loops.

A loop is {\em trivial} if it bounds a disk in $\fS$. A {\em simple diagram} is a disjoint union of non-trivial loops. We always identify a simple diagram with its isotopy class. We write $B_\fS$ for the set of all isotopy classes of simple diagrams in $\fS$. We note that $B_\fS = B_{(\fS,\emptyset)}$, where on the right hand side we use the notation of marked surfaces.

\subsection{Geometric intersection index}\lbl{sec.intersectionindex}

Fix a marked surface $\SP$. Given two $\cP$-tangles $S,T$, the {\em geometric intersection index} $\mu(S,T)$ is the minimum of $|(S' \setminus \cP) \cap (T' \setminus \cP)|$ where $S',T'$ ranges over all $S',T'$ such that $S'$ is $\cP$-isotopic to $S$ and $T'$ is $\cP$-isotopic to $T$. We say that $S$ and $T$ are {\em taut} if the number of intersection points of $S$ and $T$ in $\Sigma \setminus \cP$ is equal to $\mu(S,T)$. 

One application of Theorem \ref{r.homoisiso} that we will make use of is the following lemma.

\begin{lemma}[\cite{FHS1982}]\lbl{l.FHS} Let $x_1,x_2,\ldots,x_n$ be a finite collection of essential $\cP$-tangles, then there are essential $\cP$-tangles $x_1',x_2',\ldots,x_n'$ such that,

\begin{enumerate}[(i)]
\item for all $i$, $x_i'$ is $\cP$-isotopic to $x_i$, and
\item for all $i$ and $j$, $x_i'$ and $x_j'$ are taut.
\end{enumerate}
\end{lemma}

Suppose $x$ is an ideal arc or simple diagram in $\fS$, and $\alpha$ is a simple diagram in $\fS$. The {\em geometric intersection index} $\mu(x,\alpha)$ is the minimum of $|x' \cap \alpha'|$ where $x'$ ranges over all ideal arcs/simple diagrams isotopic to $x$ and $\alpha$' ranges over all simple diagrams isotopic to $\alpha$.

\section{Tangles in 3-manifolds}\lbl{sec.tangle3}

Fix a marked 3-manifold $\MN$.

{\em An $\cN$-tangle $T$} in $M$ consists of a compact non-oriented 1-dimensional submanifold of $M$ equipped with a normal vector field
such that $T \cap \cN = \partial T$ and at a boundary point in $\partial T \cap \cN$, the normal vector is tangent to $\cN$ and agrees with the orientation of $\cN$. Here a normal vector field is a vector field which is not collinear with the tangent space of $T$ at any point. This vector field is called the {\em framing} of $T$, and the vectors are called {\em framing vectors} of $T$. Two $\cN$-tangles are {\em $\cN$-isotopic} if they are isotopic through the class of $\cN$-tangles. The empty set is also considered a $\cN$-tangle which is $\cN$-isotopic only to itself.

For a ring $R$, we will write $\cT\MN$ for the free $R$-module of $\cN$-tangles considered up to $\cN$-isotopy. The ring $R$ under consideration will always be made clear.

Given a finite type surface $\fS$, we consider the cylinder $\fS \times (-1,1)$ to be a marked 3-manifold $(\fS \times (-1,1),\emptyset)$, and call the $\emptyset$-tangles {\em framed links}. A simple diagram $\alpha$ in $\fS$ defines a unique isotopy class of framed link in $\fS \times (-1,1)$ with vertical framing, which we will also denote as $\alpha$.

\section{Triangulations}\lbl{sec.triquasitri}

The introduction of combinatorial structures to topological spaces is an extremely powerful technique to make computation of topological invariants tractable. We make extensive use of this strategy in our main results. This strategy is occasionally at odds with clarity with regards to the topological reasons that a given result is true, however in our case the additional combinatorial structures are highly illuminating rather than obscuring.

Here we will introduce ideal triangulations of finite type surfaces and quasitriangulations of marked surfaces. These structures, while not intrinsic to the surface, are used to construct easy to understand algebras called quantum tori into which we will either embed the Kauffman bracket skein algebra in the case of the Chebyshev-Frobenius homomorphism, or find a quantum torus as an associated graded algebra in the case of the maximal order result. This vastly simplifies our algebraic workload.

\subsection{Ideal triangulations}\lbl{sec.idealtri}

Fix a finite type surface $\fS = \fS_{g,p}$ with negative Euler characteristic. One may parameterize the set $B_\fS$ of simple diagrams of $\fS$ by embedding it in the free abelian group $\BZ^r$, where $r=6g-6+3p$. This embedding depends on a {\em coordinate datum}. In this subsection we describe the coordinate embedding for finite type surfaces with at least one puncture.

We assume for this subsection that $p \geq 1$ and that $\fS$ has negative Euler characteristic $\chi(\fS) = 2-2g-p$ and we set $r:=6g-6+3p = -3\chi(\fS)$. A {\em coordinate datum} of $\fS$ is an {\em ordered ideal triangulation}, which is any sequence $(e_1,\ldots,e_r)$ of disjoint ideal arcs in $\fS$ such that no two are isotopic. It is well known that an ordered ideal triangulation always exists when $\chi(\fS)<0$ and $p \geq 1$ \cite{Pe2012}. The components of the disconnected surface $\fS \setminus \{e_1,\ldots,e_r\}$ are called the {\em ideal triangles} of $(\fS,\D)$. Given an ideal triangle $\fT$ of $(\fS,\D)$, the distinct ideal arcs $a,b,c \in \{e_1,\ldots,e_r\}$ such that $\fT \cup \{a,b,c\}$ is a connected topological space are called the {\em edges of $\fT$}.

\begin{remark}\lbl{r.idealselffold}
    For our purposes, we may assume that there are no ``self-folded triangles'' (i.e. one for which two edges coincide), since it is well-known that if an (ordered) ideal triangulation exists then one without self-folded triangles exists as well \cite{Pe2012}.
\end{remark}

Fix an ordered ideal triangulation $\fT = (e_1,\ldots,e_r)$ of $\fS$. We define the {\em edge coordinate map} $E_\fT$ with respect to $\fT$ by

\be\lbl{e.edgecoords}
E_\fT: B_\fS \to \BN^r, \ \ \ \ E_\fT(L) := (\mu(e_i,L))_{i=1}^r.
\ee
It is well-known that $E_\fT$ is injective (see e.g. \cite{Pe2012}).\no{ We write $\La := E_\fT(B_\fS)$ and $\uLa$ for the group generated by $\La$.}

\subsection{Quasitriangulations}\lbl{sec.quasitri}

Fix a marked surface $\SP$ for the remainder of this section. A triangulation of $\SP$ is a special case of a quasitriangulation, and so we begin by describing what it means for a marked surface to be quasitriangulable. In all but a small number of simple cases, a marked surface can be obtained by gluing together a collection of triangles and holed monogons along edges. Such a decomposition is called a quasitriangulation. We now give a formal definition. For additional details see \cite{Pe2012}.

$\SP$ is called {\em quasitriangulable} if 
\begin{enumerate}[(i)]
\item there is at least one marked point, and
\item $\SP$ is not a disk with less than two marked points, or an annulus with one marked point.
\end{enumerate}

A {\em quasitriangulation} $\D$ of a quasitriangulable marked bordered surface $\SP$, also called a {\em $\cP$-quasitriangulation} of $\Sigma$, is a collection of $\cP$-arcs such that
\begin{enumerate}[(i)]
\item[(i)] no two $\cP$-arcs in $\D$ intersect in $\Sigma \setminus \cP$ and no two are $\cP$-isotopic, and
\item[(ii)] $\D$ is maximal amongst all collections of $\cP$-arcs with the above property, meaning that it is not a proper subset of any other such collection.
\end{enumerate}

An element of $\D$ is also called an {\em edge} of the $\cP$-quasitriangulation $\D$. Let $\D_\bd$ be the set of all {\em boundary edges}, i.e. edges which are boundary $\cP$-arcs. The complement $\D_\inn:= \D\setminus \D_\bd$ is the set of all {\em inner edges}. 
Then $\D_\inn$ cuts $\Sigma$ into {\em $\cP$-triangles} (or just {\em triangles} when $\cP$ is clear) and {\em holed monogons}. A $\cP$-triangle is a smooth map $\tau: \sigma \to \Sigma$ from a regular triangle (in the standard plane) to $\Sigma$ such that (i) the restriction of $\tau$ to the interior of $\sigma$ is a diffeomorphism onto its image, and (ii) the restriction of $\tau$ onto each edge of $\sigma$ is a $\cP$-arc, called an edge of $\tau$. Here a holed monogon is a region in $\Sigma$ bounded by an unmarked component of $\pS$ and a $\cP$-arc, see Figure 2.1 and \cite{Pe2012} for a more precise definition.

\begin{figure}\lbl{fig:figures/hole}
\centering
\includegraphics[width=4cm]{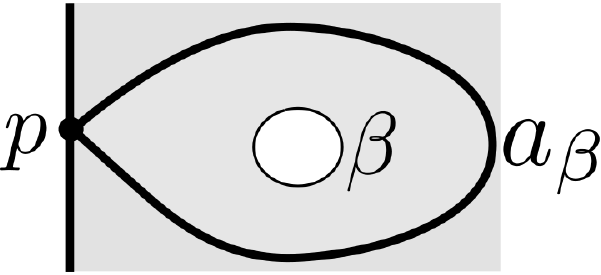}
\caption{The region bounded by and including $\beta$ and $a_\beta$ is an example of a holed monogon.}
\end{figure}

For an unmarked component $\beta\in \cH$ let $a_\beta\in \D$ be the only edge on the boundary of the monogon containing $\beta$. We call $a_\beta$ the {\em monogon edge} corresponding to $\beta$, see Figure 2.1. We denote by $\D_\mon \subset \D$ the set of all monogon edges.

\subsection{Triangulations}\lbl{sec.triangulations}

The situation simplifies if $\SP$ is totally marked and quasitriangulable, i.e. $\cH = \emptyset$, and $\SP$ is not a disk with $\le 2$ marked points. Then $\D_\mon = \emptyset$, and instead of ``quasitriangulable" and ``$\cP$-quasitriangulation'' we use the terminology {\em triangulable} and {\em $\cP$-triangulation}. Thus every triangulable surface is totally marked. A $\cP$-triangulation $\D$ cuts $\Sigma$ into $\cP$-triangles, meaning that the closure of each connected component of $\Sigma \setminus \cup_{a \in \D}$ has the structure of a $\cP$-triangle. We denote by $\cF(\D)$ the set of all triangles of the triangulation $\D$.

\begin{remark}\lbl{r.noselffold}
Since $\cP \subset \pS$, we cannot have self-folded triangles (one for which two edges coincide).
\end{remark}

\section{Character varieties}\lbl{s.charvariety}

Throughout this section we fix an oriented 3-manifold $M$. Let $\pi_1(M)$ be the fundamental group and let $G$ be an algebraic reductive Lie group. Then the {\em $G$-character variety} of $M$ is a space of equivalence classes of group homomorphisms
\[
X(M,G):=\Hom(\pi_1(M),G)//G.
\]
In general $\Hom(\pi_1(M),G)/G$ is neither separated nor an algebraic variety, so we denote with the double slash the quotient in the sense of geometric invariant theory. The $G$-character variety may be defined for manifolds in general, but in this work we are only concerned with 3-manifolds (which often are cylinders over surfaces).

We are focused on the case $G=SL_2\CC$, however $G=SU(2)$ and $G=SL_2\RR$ have ancillary roles to play.

\subsection{Trace functions}\lbl{sec.gtracefun}

Given a representation $\rho: \pi_1(M) \to G$ where $G$ is a group, the {\em character} of the representation is the composition $\chi_\rho: \tr \circ \rho$. The ring of regular functions on $X(M,G)$ is called the {\em ring of $G$-characters of $M$}, denoted by $\cR(M,G)$.

For each $\gamma \in \pi_1(M)$ there is a function $\tr_\gamma: X(M,G) \to \CC$ given by $\chi_\rho \mapsto \chi_\rho(\gamma)$. This is known as a {\em trace function}, and as we shall learn, such trace functions form a basis for $\cR(M,G)$ for certain choices of $G$.

Let $\vec{K} \subset \Sigma$ be an embedded oriented 1-manifold homeomorphic to $S^1$. Each $\vec{K}$ determines a unique trace function $\tr_\gamma \in \cR(M,G)$ as follows. Choose $\gamma \in \pi_1(M)$ such that $\gamma$ is freely homotopic to an embedding of $\vec{K}$. Then because traces are invariant under conjugation, we may define $\chi_\rho(\vec{K})=\chi_\rho(\gamma)$. 

\subsection{\texorpdfstring{$SL_2\CC$}{SL2C}-character variety}\lbl{sec.topcharvar}

Recall that the following trace identity, known as the fully polarized Cayley-Hamilton identity, holds for all $A,B \in SL_2\CC$:
\begin{equation}\lbl{e.sl2traceid}
\tr(AB) + \tr(AB^{-1}) = \tr(A)\tr(B)
\end{equation}
Equation \eqref{e.sl2traceid} plays a crucial role for us - indeed, the Kauffman bracket skein relation is a quantization of this relation \cite{Bu1997}.

Another important fact is that $\tr(A)=\tr(A^{-1})$ for all $A \in SL_2\CC$. This gives us the following. We saw in Subsection \ref{sec.gtracefun} that oriented knots give well-defined trace functions on $X(M,G)$. For $G=SL_2\CC$ (and $SU(2),SL_2\RR$), $\tr(A)=\tr(A^{-1})$ tells us that our knots do not need an orientation - thus, to every unoriented knot $K$ there is a $\gamma \in \pi_1(M)$ such that $\chi_\rho(K)=\chi_\rho(\gamma)$. Conversely, given $\gamma \in \pi_1(M)$, there is a (non-unique) $K$ such that $\chi_\rho(K)=\chi_\rho(\gamma)$.

More generally, let $L$ be a link in $M$ (e.g. a framed link but with no framing) with components $K_1,\ldots,K_n$, and $\rho \in \Hom(\pi_1(M),G)/G$. Then we may define
\begin{equation}
\tr_{L,G}([\rho]) = \prod_i (-\tr(\rho(K_i)).
\end{equation}
For the case $G=SL_2\CC$ the Kauffman bracket skein module at quantum parameter $q=-1$ is actually an algebra and is isomorphic to this ring \cite{Bu1997} (see Subsection \ref{sec.kauffmanroots}).

$X(M,SL_2\CC)$ contains several important substructures. The real points of $X(M,SL_2\CC)$, that is,
\[
X(M,SL_2\CC)|_\RR := \{ [\rho] \in X(M,SL_2\CC): \tr(\rho(\gamma)) \in \RR \text{ for all }\gamma \in \pi_1(M)\}
\]
has two non-disjoint components: $X(M,SL_2\RR)$ and $X(M,SU(2))$. 

If $M$ is closed, $X(M,SL_2\RR)$ has finitely many connected components, one of which is the Teichm\"uller space of the surface \no{(see Section \ref{sec.hyperbolic})}. In Section \ref{sec.chekhovfock} we will quantize this component as a Chekhov-Fock algebra and embed the Kauffman bracket skein algebra into it via the quantum trace map of Bonahon and Wong.

Also if $M$ is closed, $X(M,SU(2))$ is the classical phase space for the original topological quantum field theory, $SU(2)$ Chern-Simons theory \cite{Wi1989}. $X(M,SU(2))$ is a singular symplectic manifold whose geometric quantization is then the phase space for quantum $SU(2)$ Chern-Simons theory. This fact points towards numerous deep connections between the Kauffman bracket skein algebra of $M$ and quantum Chern-Simons theory, which has relevance to e.g. the Volume Conjecture \cite{DS2011} and topological quantum computing (see e.g. \cite{Wa2010} and our conjecture in Subsection \ref{s.tqc}).

\no{
If $K$ is a knot in $S^3$ and $M=S^3 \backslash V(K)$ (where $V(K)$ is an open tubular neighborhood of $K$), then $X(M,SL_2\CC)$ is frequently an affine algebraic curve whose ideal points correspond to surfaces in $M$. $X(M,SL_2\CC)$ has many conjectures relating it to asymptotics of quantum invariants such as the colored Jones polynomial, see \red{...}
}

%% file: algebra.tex
\chapter{Algebra}

In this chapter we review some essential algebraic background for our main results. In Section \ref{sec.algterms} we fix definitions for some standard algebraic terms. In Section \ref{sec.algchebyshev} we introduce the Chebyshev polynomials of the first kind and investigate a generalization of the DeMoivre identity to $q$-commuting variables. In Section \ref{sec.filtergrade} we review the theory of $\BN^p$-filtered and -graded algebras. In Seciton \ref{sec.ncdivision} we review the theory of division in noncommutative algebras, focusing on the all-important Ore condition. In Section \ref{sec.algreps} we briefly discuss representation theory of noncommutative algebras and the unicity theorem, a powerful result about representation varieties which we refine for skein algebras of surfaces in this dissertation. In Section \ref{sec.maxorderintro} we introduce maximal orders, which are a noncommutative generalization of integrally closed commutative domains. In Section \ref{sec.abstracttori} we define quantum tori, the most important algebraic structure we utilize to understand skein algebras and modules in this dissertation. Here we prove several important structural results about quantum tori, including defining the Frobenius homomorphism that ultimately descends to the Chebyshev-Frobenius homomorphism of skein modules, as well as showing that finitely-generated monomial subalgebras are maximal orders, which ultimately is used to show that skein algebras of surfaces are maximal orders.

Throughout the chapter we fix a commutative Noetherian domain $R$ containing a distinguished invertible element $q^{1/2}$ and a positive integer $p$. When we say {\em algebra} we mean an associative $R$-algebra unless otherwise specified.

\section{Original results}

The chapter contains a few original results. The first is Proposition \ref{r.chebygen}, which is a generalized DeMoivre's identity that handles $q$-commuting variables. Second is Proposition \ref{r.gradedorderlift}, which show that for certain filtered algebras such that the associated graded algebra is a maximal order, so too is the filtered algebra. Lastly is \ref{r.ALamaximal}, which shows that finitely-generated monomial subalgebras of quantum tori are maximal orders.

\section{Basic definitions}\lbl{sec.algterms}

\subsection{Common terminology}

Let $S$ be a ring. The {\em center} $Z(S)$ of $S$ is
\[
Z(S) = \{s \in S: rs=sr \text{ for all } r \in S\}.    
\]
A {\em regular element} of $S$ is any element $x\in S$ such that $xy\neq 0$ and  $yx\neq 0$ for all non-zero $y\in S$. If every $x\in S \setminus \{0\}$ is regular, we call $S$ a {\em domain}.

Given a root of unity $\xi \in \Cx$, the {\em order} $\ord(\xi)$ of $\xi$ is the smallest positive integer $N$ such that $\xi^N = 1$.

A {\em central simple algebra} $A$ over a field $K$ is a finite-dimensional associative $K$-algebra which is simple and such that the center of $A$ is precisely $K$.

\subsection{Weyl Normalization}\lbl{sec.weylnormalization}

Suppose $A$ is an algebra, not necessarily commutative. Two elements $x,y \in A$ are said to be {\em $q$-commuting} if there exists $\cC(x,y) \in \ZZ$ such that $xy=q^{\cC(x,y)}yx$. Suppose $x_1,x_2,\ldots,x_n \in A$ are pairwise $q$-commuting. Then the {\em Weyl normalization} of $\prod_i x_i$ is defined by
\[
[x_1x_2 \ldots x_n]:=q^{-\frac{1}{2}\sum_{i<j}C(x_i,x_j)}x_1x_2 \ldots x_n.
\]
It is easy to show that the normalized product does not depend on the order, i.e. given an element $\sigma$ of the symmetric group on $n$ letters, $[x_1x_2 \ldots x_n] = [x_{\sigma(1)}x_{\sigma(2)} \ldots x_{\sigma(n)}]$.

\subsection{Submonoids of \texorpdfstring{$\BZ^p$}{Zp}}\lbl{s.submonoids}

Recall that a set $\La$ equipped with an associative binary operation that has a neutral element is called a {\em monoid}. Throughout this dissertation, we will consider $\BZ^p$, $\BN^p$ to be monoids under componentwise addition.

Let $\La$ be a submonoid of $\BZ^p$. 
We denote by $\uLa$ the {\em group generated by $\La$}. This is given by
\[
\uLa = \La - \La = \{ \bk-\bk': \bk,\bk' \in \La\}.
\]

The {\em rank} of $\La$ is by definition the rank of $\uLa$, which will be at most $p$. 

The {\em cone} of $\La$ is the $\BR_{\ge 0}$-span of $\Lambda$ in $\BR^p= \BZ^p \otimes \BR$, i.e.	
\[
\Cone(\La) = \left\{ \sum_{i=1}^n c_i \bk_i: c_i \in \BR_{\geq 0}, \bk_i \in \La \right\}.
\]
\begin{remark}
$\Cone(\La)$ is closed if and only if $\La$ is finitely generated.
\end{remark}

\section{Chebyshev polynomials}\lbl{sec.algchebyshev}

The Chebyshev polynomials of the first kind $T_n(z) \in \BZ[z]$ are defined recursively by
\begin{equation}\lbl{e.chebyshev}
T_0 =2, \ \ T_1(z) = z, \ \ T_n(z) = zT_{n-1}(z)-T_{n-2}(z) \text{ for } n \geq 2.  
\end{equation}

It is easy to show that for any invertible element $X$ in a ring,
\begin{align}\lbl{e.demoivre}
T_n(X + X^{-1})& = X^n + X^{-n}.
\end{align}
This is known as {\em DeMoivre's identity}.

We want to generalize~\eqref{e.demoivre} and calculate $T_n(X+ X^{-1} +Y)$, where $Y$ is a new variable which $q$-commutes with $X$.

For $n\in \BZ$ and $k \in \BN$, we define the quantum integer $[n]_q$ and the quantum binomial coefficient $\qbinom nk_q$, which are elements of $\BZ[q^{\pm 1}]$, by
\[
[n]_q:= \frac{q^n - q^{-n}}{q - q^{-1}}, \quad   \qbinom nk_q = \prod_{j=1}^k \frac{[n-j+1]_q}{[j]_q}.
\]

\begin{proposition}\lbl{r.chebygen}
Suppose $X$ and $Y$ are variables such that $XY= q^2 YX$ and $X$ is invertible.
Then for any $n \ge 1$,
\be
\lbl{eq.38}
T_n(X+ X^{-1} +Y) = X^n + X^{-n} + Y^n + \sum_{r=1}^{n-1}\sum_{j=0}^{n-r}  c(n,r,j) \left[ Y^r X^{n-2j-r}\right]
\ee
where $c(n,r,j) \in \BZ[q^{\pm 1}]$ and is given by
\be
\lbl{eq.cnrj}
 c(n,r,j)= \frac{[n]_q}{[r]_q} \qbinom{n-j-1}{r-1}_q \qbinom{r+j-1}{r-1}_q.
 \ee
\end{proposition}
Here $[Y^aX^b]$ is the Weyl normalization of of $Y^aX^b$, see Subsection \ref{sec.weylnormalization}.

\begin{proof}
One can easily prove the proposition by induction on $n$.
\end{proof}
\begin{corollary}[\cite{BW2011,Bo2017}]\lbl{c.32}
Suppose $q^2$ is a root of unity of order $n$, then
$$
T_n(X+ X^{-1} +Y) = X^n + X^{-n} + Y^n.
$$
\end{corollary}
\begin{proof}
When $q^2$ is a root of unity of order $N$, then $[N]_q=0$ but $[r]_q \neq 0$ for any $ 1\le r \le N-1$. Equation \eqref{eq.cnrj} shows that $c(N,r,j)=0$ for all $ 1\le r \le N-1$. From \eqref{eq.38} we get the corollary.
\end{proof}

\begin{remark} \lbl{rem.34}
If $n,k\in \BN$, then $\qbinom nk_q\in \BN[q^{\pm 1}]$, the set of Laurent polynomials with non-negative integer coefficients. 
From \eqref{eq.cnrj} it follows that  $c(n,r,j) \in \BN[q^{\pm 1}]$.
\end{remark}

\section{Filtered and graded algebras}\lbl{sec.filtergrade}

Consider $\BN^p$ as a totally ordered monoid under the lexicographic ordering.

\begin{remark}\lbl{r.dccorder}
The order structure on $\BN^p$ satisfies a {\em descending chain condition}. That is, let $a_1 \geq a_2 \geq ...$ be an infinite descending chain in $\BN^p$. Then there exists some $k \in \BN$ such that $a_k = a_l$ for all $l \geq k$.
\end{remark}

An {\em $\BN^p$-filtered algebra}, or just {\em filtered algebra}, is an algebra $A$ with a family $\{F_\bm, \ \bm \in \BN^p \}$ of free $R$-modules such that
\begin{enumerate}[(i)]
\item for each $\bm,\bn \in \BN^p$, $F_\bm F_\bn \subseteq F_{\bm + \bn}$,
\item for $\bm < \bn$, $F_\bm \subseteq F_\bn$, and
\item $\bigcup_{\bm \in \BN^p} F_\bm = A$.
\end{enumerate}
Given nonzero $a \in A$, we define the {\em degree} of $a$, written $\deg a$, to be the minimum $\bm$ such that $a \in F_\bm$.

An {\em $\BN^p$-graded algebra}, or just {\em graded algebra}, is an algebra $\cA$ together with a family $\{\cA_\bm, \ \bm \in \BN^p\}$ of free $R$-modules of $\cA$ such that
\begin{enumerate}[(i)]
\item $\cA_\bm \cA_\bn \subseteq \cA_{\bm+\bn}$, and
\item $\cA = \bigoplus_\bm \cA_\bm$ as an abelian group.
\end{enumerate}
The family $\{\cA_\bm\}$ is called a {\em grading} of $\cA$, and we call $a \in \cA_\bm$ a {\em homogeneous element}. $\cA$ is also a filtered algebra, with filtration $\{F_\bm\}$ given by $F_\bm = \sum_{\bn \leq \bm} \cA_\bn$. Given homogeneous elements $a,b$, $\deg ab = \deg a + \deg b$ if $ab \neq 0$.

Let $A$ be a filtered algebra with filtration $\{F_\bm\}$. The {\em associated graded algebra} $\gr A$ of $A$ is constructed as follows. For $\bm \in \BN^p$, we define the following subset of $A$:
\[
F_{< \bm} = \bigcup_{\bn \in \BN^p: \bn < \bm} F_\bn.
\]
Set $A_\bm = F_\bm / F_{< \bm}$ and $\cA = \bigoplus A_\bm$ as an abelian group. Then $\gr A$ has $\cA$ as its additive group. Given $a \in F_\bm \setminus F_{<\bm}$, we write $\bar{a} = a + F_{<\bm} \in \cA_\bm$ and call this map $A \to \gr A$ the {\em bar mapping}. Given $b \in F_{\bn} \setminus F_{<\bn}$, we define $\bar{a}\bar{b}$ to be $ab + F_{<\bm+\bn}$. This makes $\gr A$ into an algebra.
\be\lbl{r.goodfiltration}
\bar{a}\bar{b} = \left\{ \begin{array}{ll}
\overline{ab} & \text{if } \deg ab = \bm+\bn \\
0 & \text{otherwise} \end{array} \right.
\ee
We say that $\{F_\bm\}$ is a {\em good filtration} if $\bar{a}\bar{b} = \overline{ab}$ for all $a,b \in A$. Note that $\overline{a+b} \neq \bar{a}+\bar{b}$ in general, so the bar mapping is not a ring homomorphism.

\section{Division in noncommutative rings}\lbl{sec.ncdivision}

\subsection{Division algebra}

An algebra $A$ is called a {\em division algebra} if every nonzero element of $A$ is a unit. The following proposition is well-known.
\begin{proposition}\lbl{r.divisionalgebras}
Let $A$ be a (possibly non-commutative) domain which is finitely-generated as a module over its center $Z$. Let $\tZ$ be the field of fractions of $Z$. Then $A \otimes_Z \tZ$ is a division algebra.
\end{proposition}
\begin{proof}
$A \otimes_Z \tZ$ is a central simple algebra by Posner's theorem, see Theorem 13.6.5 \cite{MRS2001}. Then by the Wedderburn structure theorem, there exists a unique positive integer $n$ and ring $S$ where all nonzero elements are units such that $A \otimes_Z \tZ$ is isomorphic as an algebra to the algebra of $n \times n$ matrices with entries in $S$. Since $A$ is a domain, so is $A \otimes_Z \tZ$, and thus $n=1$. Therefore $A \otimes_Z \tZ$ is a division algebra.
\end{proof}

\subsection{Ore condition}\lbl{sec.orecondition}

A multiplicatively closed subset $\fU$ of a ring $S$ satisfies the {\em right Ore condition} if, for each $s \in S$, $u \in \fU$, there exists $s' \in S$, $u' \in \fU$ such that $su'=us'$. The {\em left Ore condition} is defined analogously. If $\fU$ satisfies the left and right Ore conditions, it is called a {\em two-sided Ore set}.

For a multiplicatively closed subset $\fU\subset S$ consisting of regular elements of $S$, a ring $E$ is called a {\em ring of fractions for $S$ with respect to $\fU$} if $S$ is a subring of $E$ such that
\begin{enumerate}[(i)]
\item every $u\in \fU$ is invertible in $E$,
\item every $e\in E$ has presentation $ e= s u^{-1} = (u')^{-1}(s') $ for $s,s'\in D $ and $u,u'\in \fU$.
\end{enumerate}
Then $S$ has a ring of fractions with respect to $\fU$ if and only if $\fU$ is a two-sided Ore set, and in this case the left Ore localization $\fU^{-1}S$ and the right Ore localization $S\fU^{-1}$ are the same and are isomorphic to $E$, see \cite[Theorem 6.2]{GW2004}. If $S$ is a domain and $\fU=S \setminus \{0\}$ is a two-sided Ore set, then $S$ is called an {\em Ore domain}, and $\fU^{-1}S = S\fU^{-1}$ is a division algebra, called the {\em division algebra of $S$}.

\begin{proposition}\lbl{r.sandwich0}  Suppose $\fU$ is a two-sided Ore set of a ring $S$ and $S \subset S'\subset S\fU^{-1}$, where $S'$ is a subring of $S\fU^{-1}$.
\begin{enumerate}[(a)]
\item The set $\fU$ is a two-sided Ore set of $S'$ and $S'\fU^{-1}= S \fU^{-1}$.
\item If $S$ is an Ore domain then so is $S'$, and both have the same division algebra.
\end{enumerate}
\end{proposition}
\begin{proof}
\begin{enumerate}[(a)]
\item Since $S\fU^{-1}$ is also a ring of fractions for $S'$ with respect to $\fU$, one has that $\fU$ is a two-sided Ore set of $S'$ and $S'\fU^{-1}= S \fU^{-1}$.
\item Let $\fU'=S\setminus \{0\}$. Since $S\subset S' \subset S\fU^{-1} \subset S(\fU')^{-1}$, the result follows from (a).
\end{enumerate}
\end{proof}

\subsection{Weak generation}\lbl{sec.weakgeneration}

A subset $D$ of an algebra $A$ is said to {\em weakly generate} $A$ if $A$ is generated as an algebra by $D$ and the inverses of all invertible elements in $D$.
For example, if $\fU \subset A$ is a two-sided Ore subset then $A$ weakly generates its Ore localization $D \fU^{-1}$.  
Clearly an algebra homomorphism $A \to A'$ is totally determined by its values on a set weakly generating $A$.

\section{Representation theory}\lbl{sec.algreps}

\subsection{Unicity theorem}

The unicity theorem was proven in 2017 by Frohman, Kania-Bartoszynska, and L\^e  \cite{FKL2017} and is a powerful result about the representations of prime affine algebras that are finite over their center. A ring $S$ is {\em prime} if all $a,b \in S$ satisfy the following condition: if $asb=0$ for all $s \in S$ then $a=0$ or $b=0$. $S$ is an {\em affine $C$-algebra}, where $C$ is a commutative ring, if $C \subset Z(S)$ and $S$ is finitely generated as a $C$-algebra. This theorem applies to Kauffman bracket skein algebras of surfaces, and Theorem \ref{r.refinedunicity} is a refinement of this theorem, but we cite it in this chapter as it is more general.

\begin{theorem}\lbl{r.unicity1}\cite{FKL2017}
Let $k$ be an algebraically closed field and $R$ be a prime affine $k$-algebra. Suppose $R$ is generated as a module over its center $Z(R)$ by a finite set of $r$ elements.
\begin{enumerate}[(i)]
\item Every element of $\Hom_{k-\text{Alg}}(Z(R),k)$ is the central character of at least one irreducible $k$-representation and at most $r$ non-equivalent irreducible $k$-representations.
\item Every irreducible $k$-representation of $R$ has dimension $\leq N$, which is the PI degree of $R$, and also the square root of the rank of $R$ over $Z(R)$.
\item There exists a Zariski open and dense subset of the form $V_c$ of $\Hom_{k-\text{Alg}}(Z(R),k)$, where $0 \neq c \in Z(R)$, such that every $\tau \in V_c$ is the central character of a unique (up to equivalence) irreducible $k$-representation $\rho_\tau$. Moreover all representations $\rho_\tau$ with $\tau \in V_c$ have dimension $N$.
\end{enumerate}
\end{theorem}

\section{Maximal orders}\lbl{sec.maxorderintro}

Maximal orders are noncommutative generalization of integrally closed commutative domains. One of the nice properties of this fact is that they have a well-behaved representation theory, as we illustrate in Subsection \ref{sec.maxorderrep}. There are a number of slightly different notions of maximal order in the literature, most of which are equivalent in the context of our results, but see e.g. \cite{MRS2001} for more detail.

Let $A$ be an algebra that is finitely generated as a module over its center $Z$ and also a domain, and write $\tZ$ for the field of fractions of $Z$. Then $\tA:=A \otimes_Z \tZ$ is a division algebra (Proposition \ref{r.divisionalgebras}). We say $A$ is a {\em maximal order} if for any subring $B \subset \tA$ such that $A \subset B \subset (z^{-1})A$ for some nonzero $z \in Z$, we have that $A=B$. Here we write $(z^{-1})A$ instead of $z^{-1}A$ to emphasize that it is not a localization, it is merely a set, i.e.
\be\lbl{e.notlocal}
(z^{-1})A = \{ z^{-1} A \mid a\in A\}.
\ee

 We will use the following well-known fact, see \cite{MRS2001}, Proposition 5.1.10(b).

\begin{proposition}
Let $A$ be an algebra that finitely-generated as a module over its  center $Z$ and also a domain. If $A$ is a maximal order then the center of $A$ is integrally closed.
\end{proposition}

\subsection{Maximal orders and filtrations}\lbl{sec.maxorderfilter}

The following proposition is likely known to experts but does not appear in the literature to our knowledge.
\begin{proposition}\lbl{r.gradedorderlift}
Let  $\La \subset \BZ^r$ be a finitely-generated submonoid that is primitive in its group completion $\uLa$. Let $A$ be an associative $\BC$-algebra  which is finitely-generated as a module over its center $Z$ and  as a $\BC$-vector space $A$ has basis $S=\{s_\bm\mid  \bm \in \La\}$. Assume there exists a monoid homomorphism $\deg: \La \to \BN^p$ such that
\be\lbl{e.productformula}
s_\bm \cdot s_{\bm'} = q^{\cC(\bm,\bm')} s_{\bm+\bm'} + F_{<\deg(\bm)+\deg(\bm')},
\ee
where $q \in \BC^\times$ is a root of unity, $\cC(\bm,\bm') \in \BZ$, and for $\bn \in \BN^p$, $F_{<\bn}$ is the $\BC$-span of $\{s_\bm \in S: \deg(\bm) < \deg(\bn)\}$. Here $\BN^p$ is equipped with the lexicographic order. Then $A$ is a domain and
  a maximal order.
\end{proposition}

\begin{proof}
For $\bn \in \BN^p$, let $F_\bn$ be the $\BC$-span of $\{s_\bm  \mid  \deg(\bm) \leq \bn\} \subset A$. Then $A$ is an $\BN^p$-filtered algebra with filtration $\{F_\bn\}$, and we write $\cA:=\gr A$ for the associated graded algebra. We also write $Z$ for the center of $A$, $\tZ$ for the field of fractions of $Z$, $\tA:=A \otimes_Z \tZ$, $\cZ$ for the center of $\cA$, $\tcZ$ for the field of fractions of $\cZ$, and $\tcA:=\cA \otimes_\cZ \tcZ$.

It follows directly from  ~\eqref{r.goodfiltration}, ~\eqref{e.productformula} that $\{F_\bn\}$ is a good filtration, i.e. that the bar mapping preserves multiplication. 

We see from ~\eqref{e.productformula} that for all $s_\bm, s_{\bm'} \in S$, $\overline{s_\bm} \overline{s_{\bm'}} = q^{\cQ(\bm,\bm')}\overline{s_{\bm'}}\overline{s_{\bm}}$, where $Q(\bm,\bm')=\cC(\bm,\bm')-\cC(\bm',\bm)$. This implies that $Q$ is antisymmetric - that is, $Q(\bm,\bm')=-Q(\bm',\bm)$. Then by construction $\cA$ is a monomial subalgebra of a quantum torus associated to an $r \times r$ antisymmetric matrix constructed from $Q$ given by a finitely-generated primitive submonoid of $\BZ^r$, and thus (i) is a domain, (ii) is finitely generated over $\cZ$, and (iii) is a maximal order by Proposition \ref{r.ALamaximal}.

It is well-known that if the associated graded algebra $\cA$ is a domain, then $A$ is a domain.

Suppose $B$ is a subring of $\tA$ such that there exists some $z \in Z$ where $A \subset B \subset (z^{-1})A$. Then $zA \subset zB \subset A$. We will write $D:=zB$ (which we emphasize is merely a set). Now consider the following chain of inclusions under the bar mapping: $\overline{zA} \subset \bar{D} \subset \bar{A} (\subset \tcA)$. Since the bar mapping preserves multiplication we have that
\be\lbl{e.barinclusions}
\bar{A} \subset (\bar{z})^{-1} \bar{D} \subset (\bar{z})^{-1}\bar{A}
\ee
where we emphasize that these are inclusions of sets, and not localizations of any sort. It also follows that $(\bar{z})^{-1}= \overline{(z^{-1})}$, so we will just write $\bar{z}^{-1}$. Consider the $\BC$-vector subspace of $\tcA$ generated by each of the sets in ~\eqref{e.barinclusions}. It is clear that $\BC \cdot \bar{A} = \cA$ by how the associated graded algebra is constructed, and furthermore that $\BC \cdot (\bar{z}^{-1}) \bar{D} = (\bar{z}^{-1}) (\BC \cdot \bar{D})$ and $\BC \cdot (\bar{z}^{-1}) \bar{A} = (\bar{z}^{-1}) \cA$ since $\bar{z}^{-1} \in \tcA$. In that vein, we write $\cD := \BC \cdot \bar{D}$. Thus we have that 
\[
\cA \subset (\bar{z}^{-1}) \cD \subset (\bar{z}^{-1})\cA.
\]

\begin{lemma}\lbl{r.zcDRalg}
$\bar{z}^{-1}\cD$ is a ring.
\end{lemma}

\begin{proof}
Let $d \in \bar{z}^{-1}\cD$. Then we may write $d = \bar{z}^{-1}\overline{zb}$ for some $b \in B$. Again because $\{F_\bm\}$ is a good filtration, the bar mapping preserves multiplication, so we have that $d = \bar{z}^{-1}\bar{z}\bar{b} = \bar{b}$.

Let $d_1,d_2 \in \bar{z}^{-1}\cD$. That $d_1+d_2 \in \bar{z}^{-1}\cD$ follows immediately from the fact that $\bar{z}^{-1}\cD$ is a $\BC$-vector space, so we only need to show that $d_1d_2 \in \bar{z}^{-1}\cD$. Let $b_1,b_2 \in B$ such that $d_1=\overline{b_1}$, $d_2 = \overline{b_2}$. Then because the bar mapping preserves multiplication, $d_1d_2 = \bar{z}^{-1}\bar{z}\overline{b_1}\overline{b_2} = \bar{z}^{-1} \overline{zb_1b_2} \in \bar{z}^{-1}\cD$.
\end{proof}

Since $\cA$ is a maximal order and $\bar{z}^{-1}\cD$ is a ring by Lemma \ref{r.zcDRalg}, we then have that $\cA = \bar{z}^{-1} \cD$, or equivalently, that $\bar{z}\cA = \cD$.

We now wish to show that $zA=D$. We already know that $zA \subset D$ by assumption.

So let $x \in D$. Since $D \subset A$, we may write $x$ uniquely as a linear combination of elements of $S$. Let $\bm_x = \max_{\bm \in \supp(x)}\{\bm\}$. Then
\[
x = \sum_{a_i \in S_{\bm_x}} c_ia_i + \text{lower degree terms}, \ \ \ \ c_i \in \BC,
\]
and so
\[
\bar{x} = \sum_{a_i \in S_{\bm_x}} c_i\bar{a_i}.
\]
Since $\bar{z}\bar{A} = \bar{D}$, $\bar{x} = \overline{(za)}$ for some $a \in A$.

Consider $y = x - za$. Then $x = za+y$, and since $\bar{x} = \overline{(za)}$, we have that $\deg(\bar{y}) < \deg(\bar{x})$. Since $x \in D$ and $za \in D$, we have that $y \in D$ as well. If $y \in zA$ we are done, otherwise we can repeat the above process to write $y$ as a sum of an element of $zA$ and an element of $D$ of degree lower that $y$. Since the order structure on $\BN^p$ satisfies a descending chain condition (see Remark \ref{r.dccorder}), the recursive process will eventually terminate with the remainder an element of $zA$ or $\BC \subset A$, showing that $x \in zA$.

Thus $zA = D$ and so $A = z^{-1}D = B$ so $A$ is a maximal order.
\end{proof}

\subsection{Representation theory of maximal orders}\lbl{sec.maxorderrep}

Maximal orders have a well-behaved representation theory in comparison to most classes of noncommutative algebras, and we illustrate that in this subsection.

Suppose $A$ is a $\BC$-algebra with center $Z$. We assume that $A$ is a domain,  finitely generated as a $\BC$-algebra, and finitely generated as a $Z$-module. Let $\Sp(A)$ denote the set of all equivalence classes of finite dimensional irreducible representations of $A$ over $\BC$, so that if $A$ is commutative then $\Sp(A)=\Hom_{\BC-alg}(A,\BC)$ is the maximal spectrum, which comes equipped with the Zariski topology. By Schur's lemma we have the central character map
$$ \chi : \Sp(A) \to \Sp(Z),$$
defined so that if 
$ \rho: A \to M_n(\BC)$ 
is an irreducible representation then 
$\rho(a) = \chi(\rho)(a) \, \id$. 

Let $\tZ$ be the field of fractions of $Z$ and $\tA:= A \ot_Z \tZ$ be the division algebra of $A$. 
As a vector space over $\tZ$ the dimension of $\tA$ is $d^2$, a perfect square, and $d$ is called the {\em degree of $A$}.

If $A$ is a maximal order then it is an algebra with trace in the sense of Procesi \cite{Pr1976} and we have the following.
\begin{theorem}[\cite{DKP93}] \lbl{r.dCKP}
Let $A$ be a finitely generated $\BC$-algebra which is finitely generated as a module over its center $Z$. Assume that $A$ is a domain and a maximal order. Let $d$ be the degree of $A$.

(i) The points of $\Sp(Z)$ parameterize equivalence classes of $d$-dimensional semisimple representations of $A$.

(ii) The central character map $\chi: \Sp(A) \to \Sp(Z)$ is surjective and each fiber consists of all those irreducible representations of $A$ which are irreducible components of the corresponding semisimple representation. In particular, each irreducible representation of $A$ has dimension at most $d$.

(iii) The set
$$ \Omega_A=\{ a\in \Sp(Z) \mid \chi^{-1}(a) \ \text{has dimension}\ d \}$$
is a non-empty Zariski open set.

\end{theorem}

\section{Quantum tori}\lbl{sec.abstracttori}

In this section we survey the basics of quantum tori and present the Frobenius homomorphism of quantum tori and its relationship with Chebyshev polynomials.
 
Throughout this section we presume that $R$ has a distinguished invertible element $q^{1/2}$. The reader should have in mind the example $R=\ZZ[q^{\pm 1/2}]$. We also fix a finite set $I$ for the remainder of the chapter.

\subsection{Abstract quantum torus}\lbl{sec.aqt}

Denote by $\Mat(I \times I, \BZ)$ the set of all $I\times I$ matrices with entries in $\BZ$, i.e. $U \in \Mat(I \times I, \BZ)$ is a function $U: I \times I \to \BZ$. We write $U_{ij}$ for $U(i,j)$. Occasionally we will use the same notation for a $|I| \times |I|$ matrix with integer values when the underlying set is totally ordered.

Let $U \in \Mat(I \times I, \BZ)$ be antisymmetric, i.e. $U_{ij}= - U_{ji}$. 
Define the {\em quantum torus over $R$ associated to $U$ with basis variables $x_i, i \in I$} by
\begin{align*}
\BT(U;R):= R\la x_i^{\pm 1} , i\in I\ra /(x_i x_j = q^{U_{ij}} x_j x_i).
\end{align*}
When $R$ is fixed, we write $\BT(U)$ instead of $\BT(U;R)$. $\BT(U;R)$ is sometimes called the algebra of skew Laurent polynomials or twisted Laurent polynomials.

Let $\BT_+(U)\subset \BT(U)$ be the $R$-subalgebra generated by $x_i, i\in I$. We call $\BT_+(U)$ the {\em positive quantum torus over $R$ associated to $U$ with basis variables $x_i, i \in I$}. $\BT_+(U)$ is also sometimes called a quantum plane.

Let $\BZ^I$ be the set of all maps $\bk: I \to \BZ$. For $\bk \in \BZ^I$ define the {\em normalized monomial} $X^\bk$ using the Weyl normalization:
\[
X^\bk := \left[ \prod_{i \in I} x_i^{\bk(i)} \right].
\]
The set $\{X^\bk \mid \bk \in \BZ^I\}$ is an $R$-basis of $\BT(U)$, i.e. we have the direct decomposition
\begin{equation}\lbl{eq.grading}
\BT(U) = \bigoplus_{\bk \in \BZ^I} R \cdot X^\bk.
\end{equation}
Similarly, $\BT_+(U)$ is  free over $R$ with basis $\{X^\bk \mid \bk \in \BN^I\}$.

Define an anti-symmetric $\BZ$-bilinear form on $\BZ^I$ by
\[
\langle \bk, \bn \rangle_U:= \sum_{i,j\in I} U_{ij}\, \bk(i)\bn(j).
\]
The following well-known fact follows easily from the definition:

For $\bk,\bn \in \BZ^I$, one has
\begin{equation}\lbl{e.normalizedtorus}
X^\bk X^\bn = q^{\frac{1}{2} \langle \bk, \bn \rangle_U} X^{\bk+\bn}  = q^{\langle \bk,\bn \rangle_U}X^{\bn}X^{\bk},
\end{equation}

In particular, for $n \in \BZ$ and $\bk \in \BZ^I$, one has 
\begin{equation}
\lbl{eq.power}
(X^\bk)^n = X^{n\bk}.
\end{equation}
The first identity of \eqref{e.normalizedtorus} shows that the decomposition \eqref{eq.grading} is a $\BZ^I$-grading of the $R$-algebra $\BT(U)$.

\subsection{Monomial subalgebras}\lbl{sec.monoidsubalgebra}

For a subset $\La \subset \BZ^p$ let $\BA(\La):= \BA_q(U,\La)\subset \BT$ be the  $\BC$-subspace of $\BT$ spanned by $X^\bk, \bk \in \Lambda$. From \eqref{e.normalizedtorus} it is easy to see that $\BA(\La)$ is a subalgebra if and only if $\La$ is a submonoid of $\BZ^p$. In that case we call $\BA(\La)$ a {\em monomial subalgebra of $\BT$}.
 If $\La$ is a subgroup of $\BZ^p$, then $\BA(\La)$ is isomorphic to a quantum torus.

\subsection{Quantum tori at a root of unity}

The following fact is more or less well-known.

\begin{proposition}\lbl{r.qtorusmaximal} Assume that $\xi \in \BC^\times$ is a root of unity and let $U \in \Mat(I \times, I, \BZ)$ be antisymmetric. Then $\BT(U)$ is finitely generated over its center and is a maximal order.
\end{proposition}
\begin{proof} Suppose $\ord(\xi)= N$. Let 
\[
\Gamma_N:= \{ \bk \in \BZ^p \mid \la \bk, \bn \ra_U \in N \BZ \ \forall \bn \in \BZ^p \}.
\]

From \eqref{e.normalizedtorus} it is easy to see that the center of $\BT$ is $\BA(\Gamma_N)$. Since $N\bk \in \Gamma_N$ for all $\bk\in \BZ^p$, we see that $\BT$ is finitely-generated over the center $\BA(\Gamma_N)$.

Since $\Gamma_N$ is isomorphic to a free abelian group, $\BA(\Gamma_N)$ is isomorphic to the ring of Laurent polynomial in several variables, and hence is integrally closed.

An $R$-algebra $A$ is an {\em Azumaya algebra} over $R$ if $A$ is free and finite rank as an $R$-module such that $A \otimes_R A^{\text{op}}$ is isomorphic to $\text{End}_R(A)$ via the map $a\otimes b \mapsto (x \mapsto axb)$. Here, $A^{\text{op}}$ is the opposite algebra of $A$.

 Proposition 2.3 of \cite{DKP93} states that $\BT$ is an Azumaya algebra over its center when $\xi$ is a root of unity. By Theorem 9.4 of 
 \cite{Sa99}, if $A$ is an Azumaya algebra and its center is integrally closed, then $A$ is a maximal order. 
\end{proof}

\subsection{Two-sided Ore domain}

 Both $\BT(U;R)$ and $\BT_+(U;R)$ are two-sided Noetherian domains, see \cite[Chapter 2]{GW2004}. As any two-sided Noetherian domain is a two-sided Ore domain (see \cite[Corollary 6.7]{GW2004}), both $\BT(U;R)$ and $\BT_+(U;R)$ are two-sided Ore domains.

Corollary \ref{r.sandwich} follows immediately from Proposition \ref{r.sandwich0}.
\begin{corollary} \lbl{r.sandwich}  Suppose  $\BT_+(U)\subset S\subset \BT(U)$, where $S$ is a subring of $\BT(U)$. Then $S$ is an Ore domain and the embedding $ S \hookrightarrow \BT(U)$ induces an algebra isomorphism from the division algebra of $S$ to that of $\BT(U)$.
\end{corollary}

\subsection{Reflection anti-involution}\lbl{sec.reflection}

Suppose $U \in \Mat(I \times I, \BZ)$ is antisymmetric. The following is easy to prove, see \cite{Le2017}. 

\begin{proposition}\lbl{r.reflection}   Assume that there is a $\BZ$-algebra homomorphism $\eta: R \to R$ such that $\eta(q^{1/2})= q^{-1/2}$ and $\eta^2= \id$, the identity map.
There exists a unique $\BZ$-linear isomorphism $\heta: \BT(U) \to \BT(U)$ such that $\heta(r X^\bk )= \eta(r) X^\bk$ for all $r\in R$ and $\bk\in \BZ^I$, which is an anti-homomorphism, i.e. $\heta(ab) = \heta(b) \heta(a)$. In addition, $\heta^2=\id$.
\end{proposition}
We call $\heta$ the {\em reflection anti-involution}.

An element $z \in \BT(U)$ is {\em reflection invariant} if $\heta(z)=z$. From the definition, we see that $\heta(X^\bk)=X^\bk$ for all $\bk \in \BZ^I$.

\no{
\red{The below is currently unused}. Assume that $q \in R$ is an indeterminate.
\begin{lemma}\lbl{r.qtorusreflection}
Suppose $z\in \BT(U)$ is reflection-invariant and
\be\lbl{e.qtorusreflection}
z = \sum_{j=1}^m q^{r_j} X^{\bk_j},
\ee
where $r_n \in \BQ$ and $\bk_j$ are pairwise distinct. Then all $r_j=0$, i.e. $z = \sum_{j=1}^m X^{\bk_j}$.
\end{lemma}
\begin{proof}
Applying $\heta$ to~\eqref{e.qtorusreflection}, we have $z = \sum_{j=1}^m q^{-r_j}X^{\bk_j}$. Since the $\bk_j$ are pairwise distinct, the presentation of $z$ as an $R$-linear combination of $X^{\bk_j}$ is unique. Hence we must have $q^{r_j}=q^{-r_j}$, or $r_j=0$.
\end{proof}
}

\subsection{Frobenius homomorphism}

\begin{proposition} \lbl{r.Frobenius}
Suppose $U \in \Mat(I \times I, \BZ)$ is an antisymmetric matrix. There is a unique algebra homomorphism, called the {\em Frobenius homomorphism},
\[
F_n: \BT(p^2U) \to \BT(U)
\]
such that $F_p(x_i) = x_i^p$. Furthermore, $F_p$ is injective.
\end{proposition}
\begin{proof} Since the $x_i$ weakly generate $\BT(p^2U)$, the uniqueness is clear. If we define $F_p$ on the generators $x_i$ by $F_p(x_i)= x_i^p$, it is easy to check that the defining relations are respected by $F_p$. Hence $F_p$ gives a well-defined algebra map. 

By \eqref{eq.power} one has $F_p(X^\bk)= X^{p \bk}$. This shows $F_p$ maps the $R$-basis 
$\{ X^\bk \mid \bk \in \BZ^I\}$ of $\BT(p^2U)$ injectively onto a subset of an $R$-basis of $\BT(U)$. Hence $F_p$ is injective.
\end{proof}

\subsection{Maximal order}

Suppose $\Lambda$ is a submonoid of $\BZ^p$ which generates the subgroup $\uLa$.
We say that $\La$ is {\em primitive in $\uLa$} if, given $c\bk \in \La$ for some positive integer $c$ and $\bk \in \uLa$, then $\bk \in \La$ as well.

The following is a main technical lemma.

\begin{proposition}\lbl{r.ALamaximal}
Suppose $U$ is a $p \times p$ anti-symmetric matrix with integer entries and $\xi$ is a primitive root of unity of order $N$. Let $\La \subset \BZ^p$ be a submonoid such that (i) $\Cone(\La)$ is closed, and (ii) $\La$ is primitive in its group completion $\uLa$. Then $\BA_q(U,\La)$ is a maximal order.
\end{proposition}

\begin{proof}
Since $\BA(\uLa)$ is a quantum torus, without loss of generality we can assume that $\uLa=\BZ^p$. Thus $\BA(\uLa)= \BT_\xi(U)$ which will be abbreviated as $\BT$.

As $\{ X^\bk \mid \bk\in \BZ^p\}$ is a $\BC$-basis of $\BT$, every non-zero $a\in \BT$ has a unique presentation
\be
a =\sum_{\bk \in \supp(a)} c_\bk X^\bk, \quad 0 \neq c_\bk \in \BC,
\ee
where $\supp(a) \subset \BZ^p$ is a finite non-empty set. The Newton polygon $\Newt(a) $ of $a$ is defined to be the convex hull of $\supp(a)$ in $\BR^{p} = \BZ^p \otimes \BR$, that is,

\[
\Newt(a) = \left\{ \sum_{\bk \in \supp(a)} \al_\bk \bk : \al_\bk \in \BR_{\geq 0} \text{ for all } \bk \in \supp(a), \ \sum_{\bk \in \supp(a)} \al_\bk = 1 \right\}.
\]

We recall that given $C,C' \subset \BR^p$, the {\em Minkowski sum} $C+C'$ is
\[
C+C' = \{c+c': c \in C, c' \in C'\} \subset \BR^p.
\]

We have the following analog of the Ostrowski lemma \cite{Os21}.
\begin{lemma}\lbl{r.ostrov}
Let $a,a' \in \BT$ be non-zero. Then $\Newt(aa')$ equals the Minkowski sum $\Newt(a) + \Newt(a')$. In particular, if $\bv$ is a vertex of $\Newt(a)$, then $k\bv \in \supp(a^k)$ for all positive integers $k$.
\end{lemma}

The proof is almost the same as the one presented in commutative case, Proposition 19.4 of \cite{Gr07}.

\begin{proof}
Let $C \subset \BR^p$ be compact and convex. A hyperplane $H \subset \BR^p$ is a {\em support hyperplane} of $C$ if $H \cap \partial C \neq \emptyset$ (where $\partial C$ is the boundary of $C$) and $C$ is contained in one of the two closed halfspaces determined by $H$. We write $H^-$ for the halfspace containing $C$, and $H^+$ for the other one. Given $\bw \in H \cap \partial C$, we represent $H$ in the form $H = \{\by \in \BR^p: \bu \cdot \by = \bu \cdot \bw\}$, where $\bu \in \BR^p$ is a unit length vector pointing into $H^+$ called the {\em exterior normal vector} and $\cdot$ is the standard Euclidean dot product. By Theorem 4.1 of \cite{Gr07}, for each unit length vector $\bu \in \BR^p$ there is a unique support hyperplane of $C$ with exterior normal vector $\bu$ and so we may use the notation $H_C(\bu)$ to represent that hyperplane unambiguously. Given a support hyperplane $H_C(\bu)$ of $C$, the intersection $H_C(\bu) \cap C$ is called a {\em support set} of $C$
. By convexity, it is easy to show that vertices of $C$ live in singleton support sets of $C$, and conversely that every singleton support set of $C$ consists of a vertex of $C$.

By~\eqref{eq.grading}, we may write $a,a'$ uniquely as
\[
a = \sum_{\bk \in \supp(a)} c_\bk X^\bk \ \ \ \ a' = \sum_{\bk' \in \supp(a')} c_{\bk '}X^{\bk '}, \ \ c_\bk,c_{\bk '} \in \BC.
\]
Then with~\eqref{e.normalizedtorus} we have
\be\lbl{e.newt}
aa' = \sum_{(\bk,\bk') \in (\supp(a),\supp(a'))} c_\bk c_{\bk '}X^\bk X^{\bk '} = \sum_{(\bk,\bk')} c_\bk c_{\bk '}\xi^{\frac{1}{2}\la \bk, \bk' \ra_U} X^{\bk + \bk'}.
\ee
Write $C:=\Newt(a)$, $C':=\Newt(a')$, $D:=C+C'$, and $E:=\Newt(aa')$. Since each monomial of $aa'$ is a sum of products of a monomial of $a$ and a monomial of $a'$, $E \subset D$ follows immediately. To show that $D \subset E$, it suffices to show that each vertex $\bv$ of $D$ is contained in $E$.

Let $\bu \in \BR^p$ be a unit length vector. Then Lemma 6.1 of \cite{Gr07} states that
\be\lbl{e.supportset}
D \cap H_D(\bu) = C \cap H_C(\bu) + C' \cap H_{C'}(\bu).
\ee

Let $\bv$ be a vertex of $D$, and $\bu \in \BR^p$ such that $D \cap H_D(\bu)=\{\bv\}$. Then since $\{\bv\}$ is a singleton, by~\eqref{e.supportset} so too must be $C \cap H_C(\bu)=:\{\bx\}$ and $C' \cap H_{C'}(\bu)=:\{\by\}$. Then $\bv = \bx+\by$ and this is the only possible representation of $\bv$ as a sum of a point of $C$ and a point of $C'$. Correspondingly, we have that the product $c_\bx c_\by \xi^{\frac{1}{2} \la \bx,\by \ra_U}X^{\bx+\by}$ that appears in the right hand side of~\eqref{e.newt} will not be canceled by any other term since no other term will have the factor $X^{\bx+\by}$. Thus $D \subset E$.

\end{proof}

Now we continue with the proof of Proposition \ref{r.ALamaximal}.

We will write $\cA:= \BA_\xi(U,\La)$, $\cZ$ for the center of $\cA$, and $\tcZ$ for the field of fractions of $\cZ$. Then $\tBT =\cA \otimes_\cZ \tcZ$ is the division algebra of $\BT$.

Suppose there is a subring $\cB$ of $\tcA$ and an element $z \in \cZ$ such that $\cA \subset \cB \subset (z^{-1})\cA$. We recall that $(z^{-1})\cA$ is merely a set, not a localization of $\cA$, see ~\eqref{e.notlocal}.

Assume there exists some $b \in \cB \setminus \cA$. We claim that $b \in \BT $. Let $\fM \subset \cA$ be the multiplicative set generated by $\{X^{\bk}, \bk \in \La\}$. Then $\fM$ is a two-sided Ore subset of $\cA$ and $\cB$ by $\xi$-commutativity, and so we have that
\[
\cA\fM^{-1} \subset \cB \fM^{-1} \subset (z^{-1}) \cA \fM^{-1}.
\]
Clearly $\cA\fM^{-1} \cong \BT $. It is easy to show that $z \in Z(\BT)$. By Proposition \ref{r.qtorusmaximal}, $\BT$ is a maximal order, and so we have that $\BT  \cong \cB\fM^{-1}$. Thus $b\in \cB \subset \cB\fM^{-1} =\BT$. This completes the proof of the claim.

Since $b$ is an algebra, $b^m \in \cB$ for $m$ a positive integer, and
\be\lbl{e.btzinA}
b^mz \in \cA \ \ \ \ \text{for all positive integers m.}
\ee

Now since $b \neq 0$, the support $\supp(b)$ is defined. We have that $b \notin \cA$ by assumption, and so $\supp(b) \not\subset \La$.

\no{
Let $\Cone(\La)$ be the $\BR_{\ge 0}$-span of $\Lambda$, i.e. 
\[
\Cone(\La) = \left\{ \sum_{i=1}^n c_i \bk_i: c_i \in \BR_{\geq 0}, \bk_i \in \La \right\}.
\]
}

We will use the fact that $\La$ is primitive in the following.

\begin{lemma}\lbl{r.primcone}  There exists some vertex $\bv \in \Newt(b)$ such that $\bv \notin \Cone(\La)$.
\end{lemma}
\begin{proof} Assume the contrary that all vertices $\bv_1, \dots, \bv_r$ of $\Newt(b)$ are in $\Cone(\Lambda)$.
As $\supp(b) \not \subset \Lambda$ there is a point $\bv\in \supp(b) \setminus \Lambda$. The point $\bv$, being in the convex hull of the vertices of $\Newt(b)$, can be presented as a convex linear combination of the vertices. As $\bv$ and $\bv_i$ are integer points, the coefficients of the linear combination can be chosen to be rational.
Thus there are non-negative integers $c_1,\ldots,c_r$  with positive sum $\sum c_i$ such that
$\bv = (\sum c_i \bv_i)/(\sum c_i)$. Then
\be\lbl{e.primlat}
\left( \sum_{i=1}^r c_i \right) \bv = \sum_{i=1}^r c_i \bv_i.
\ee
Since the right hand side of~\eqref{e.primlat} is in $\La$, the primitiveness implies that $\bv\in \Lambda$, a contradiction.
\end{proof}

Given $\bv \in \BR^p$, define the distance function
\[
d(\bv,\Cone(\La)) = \inf_{\by \in \Cone(\La)} \{\|\by-\bv\|\},
\]
where $\| \cdot \|$ is the Euclidean norm. We have that $d(k\bv,\Cone(\La)) = k\, d(\bv,\Cone(\La))$. Since $\Cone(\La)$ is closed and $\bv \notin \Cone(\La)$, $\lim_{k \to \infty} d(k \bv,\Cone(\La)) = \infty$.

Thus there exists some large $k_1$ such that $k_1\bv + \Newt(z)$ does not intersect $\Cone(\La)$. Hence $\Newt(b^{k_1}z)=k_1\Newt(b)+\Newt(z)$ (by Lemma \ref{r.ostrov}) cannot be a subset of $\Cone(\La)$. Since $b^{k_1}z \in \cA$ if and only if $\Newt(b^{k_1}z) \subset \Cone(\La)$, this means that $b^{k_1}z \notin \cA$, a contradiction.

Thus $\cA=\cB$, concluding the proof of the proposition.
\end{proof}

%% file: kauffman.tex
\chapter{Kauffman bracket skein modules}

In this chapter we introduce our primary objects of study: the Kauffman bracket skein modules of marked 3-manifolds and Kauffman bracket skein algebras of surfaces. Our goal here is to give a focused picture of the current state of the art of the structure of these modules as they relate to our primary results and is by no means exhaustive.

We do this by highlighting several aspects.

The first is the relationship between skein algebras of marked surfaces to skein modules of marked 3-manifolds in Sections \ref{s.def00} and \ref{s.skeinsurface}. Namely, we will construct skein algebras of surfaces by taking a cylinder over the surface to obtain a marked 3-manifold, define the skein module for the marked 3-manifold, and then adding in the algebra structure by hand (though we note that this algebra structure also arises naturally by less artificial means). The resulting algebra is isomorphic to what would be obtained by working directly with multicurves drawn on the surface, and we show that our approach is equivalent to this, but has the advantage of allowing us to use the same formalism for both surfaces and 3-manifolds. In these sections we also show that the skein algebra of a surface is a domain, introduce the important technical tool of the reflection anti-involution, and prove the existence of the division algebra of the skein algebra.

\no{
In Section \ref{sec.quantizationcharacter} we review the construction of the skein module as a deformation quantization of the $SL_2\BC$-character variety along the Goldman-Weil-Petersson Poisson form. Among the wisdom revealed by this construction is the fact that the transformation from multicurves embedded in a surface to framed tangles embedded in a thickened surface is precisely the outcome of this quantization procedure - indeed, one can think of the distance between the bottom and top strand of a crossing as being proportional to $e^\hbar$. 
}

Understanding the skein module as a quantization of the $SL_2\BC$-character variety helps in illuminating the special role played by the skein module at 4th roots of unity, discussed in Section \ref{s.rootsof1}, which serve to describe characteristics of the $SL_2\BC$-character variety and also serve as the domain for the Chebyshev-Frobenius homomorphism. The enormous advantage of the quantum point of view on the character variety will ultimately be illustrated by a compact proof of the fact that character variety is normal, a consequence of the results of Chapter \ref{c.maxorder}.

\no{

The tractability of the representation theory of skein algebras of surfaces is then illustrated with the unicity theorem, which in short shows that the set of characters corresponding to a single irreducible representation forms a Zariski open dense subset of the character variety.

Next we introduce the miraculous quantum trace map of Bonahon and Wong, which embeds the difficult to understand skein algebra of a punctured marked surface into a Chekhov-Fock algebra, which is a type of easy-to-understand quantum torus corresponding to quantization of the shear coordinates of Teichm\"uller space.

Side-by-side with the quantum trace map is Muller's skein coordinate map, which is another embedding of the skein algebra of a marked surface into a quantum torus corresponding to a quantization of the Penner coordinates of Teichm\"uller space. This embedding turns out to have transfer isomorphisms between quantum tori corresponding to different quasitriangulations that are easier to work with that the ones for the Chekhov-Fock algebra.

Finally we look at the shear-to-skein map of L\^e, giving the direct connection between Chekhov-Fock algebra and Muller algebra associated to a triangulable surface as well as giving a simplified construction for the quantum trace map.

}

Throughout the chapter we fix a commutative Noetherian domain $R$ with a distinguished invertible element $q^{1/2}$ and a $\BZ$-algebra involution $\eta:R\to R$ such that $\eta(q^{1/2})= q^{-1/2}$.

\section{Skein modules of marked 3-manifolds}\lbl{s.def00}

In this section we define the Kauffman bracket skein module of marked 3-manifolds, closely following \cite{Le2017}. Throughout this section we fix a marked 3-manifold $\MN$.

Recall that $\cT\MN$ is the $R$-module freely spanned by the $\cN$-isotopy classes of $\cN$-tangles in $M$. The {\em Kauffman bracket skein module}  $\cS\MN$ is the quotient $\cS\MN= \cT\MN/\Rel_q$ where $\Rel_q$ is the $R$-submodule of $\cT\MN$ spanned by the skein relation elements, the trivial loop relation elements, and the trivial arc relation elements, where
\begin{enumerate}[(i)]
 \item  $R$ is identified with the  $R$-submodule of $\cT\MN$ spanned by the empty tangle, via $c \to c\cdot \emptyset$,
 \item  a {\em skein relation element} is any element of the form  $T  -q T_+ - q^{-1} T_-$, where $T ,T_+,T_-$  are $\cN$-tangles  identical everywhere except in a ball in which they look like Figure \ref{fig:figures/skein1},

 \FIGc{figures/skein1}{From left to right:  $T, T_+, T_-$.}{1.2cm}
\item a {\em trivial loop relation element} is any element of the form $\beta + q^2 + q^{-2}$, where $\beta$ is a trivial knot, i.e. a loop bounding a disk in $M$ with framing perpendicular to the disk.

\item   a {\em trivial arc relation element} is any $\cN$-tangle $T$ containing an $\cN$-arc $a$ for which there exists an arc $b \subset \cN$ such that $a \cup b$ bounds a continuous disk $D$ in $M$ such that $D \cap (T \setminus a) = \emptyset$. This situation is depicted in Figure \ref{fig:figures/trivialarc}.

\FIGc{figures/trivialarc}{The gray shaded area is a region of $\partial M$, the striped region is a continuous disk with boundary $a \cup b$, and $b \subset \cN$.}{1.5cm}
\end{enumerate}

\no{\red{These relation elements are also depicted in Figure \ref{fig:figures/skein} in the Introduction.}}

The following was first shown in \cite[Proposition 3.1]{Le2017}. We reproduce the proof here for the convenience of the reader.
\begin{proposition}\lbl{r.reorderingrel}
In $\cS\MN$ the {\em reordering relation} depicted in Figure \ref{fig:figures/boundary} holds.
\end{proposition}
\FIGc{figures/boundary}{Reordering relation: Here $\cN$ is perpendicular to the page and its perpendicular projection onto the page (and in the shaded disk) is the bullet denoted by $\cN$. The vector of orientation of $\cN$ is pointing to the reader. There are two strands of the tangle coming to $\cN$ near $N$, with the lower one being depicted by the broken line}{2cm}

\begin{proof}
The proof is given in Figure \ref{fig:figures/reordering}. Here the first identity is an isotopy, the second is the skein relation, the third follows from the trivial arc relation.
\FIGc{figures/reordering}{Proof of Proposition \ref{r.reorderingrel}}{1.8cm}
\end{proof}

\begin{remark} Muller \cite{Mu2012} introduced Kauffman bracket skein modules for marked surfaces. Here we use a generalization of Muller's construction to marked 3-manifolds, introduced in \cite{Le2017}. \end{remark}

\subsection{Functoriality}\lbl{sec.func}

By a {\em morphism} $f: \MN \to (M',\cN')$ between marked 3-manifolds we mean an orientation-preserving embedding $f: M \embed M'$ such that $f$ restricts to an orientation preserving embedding on $\cN$. Such a morphism induces an $R$-module homomorphism $f_*: \cS\MN \to \cS(M',\cN')$ by $f_*(T)=T$ for any $\cN$-tangle $T$.

Given marked 3-manifolds $(M_i, \cN_i)$, $i=1,\dots, k$, such that $M_i\subset M, \cN_i \subset \cN$, and the $M_i$ are pairwise disjoint, then there is a unique $R$-linear map, called the {\em union map}
\[
\text{Union}:\prod_{i=1}^k \cS(M_i, \cN_i) \to \cS(M,\cN),
\]
such that if $T_i$ is an $\cN_i$-tangle in $M_i$ for each $i$, then
\be\lbl{e.unionmap}
\text{Union} (T_1,\dots, T_k) = T_1 \cup \dots \cup T_k.
\ee
For $x_i\in \cS(M_i,\cN_i)$ we also denote $\text{Union} (x_1,\dots, x_k)$ by $x_1 \cup \dots \cup x_k$.

\section{Skein algebras of surfaces}\lbl{s.skeinsurface}

In this section we will define skein algebras of finite type and marked surfaces by describing them as a skein module of a thickened surface with an algebra structure coming from stacking in the vertical direction. Throughout this section we fix a marked surface $\SP$ and a finite type surface $\fS$.

\subsection{Skein modules of marked surfaces}

Let $M$ be the cylinder over $\Sigma$ and $\cN$ the cylinder over $\cP$ as described in Subsection \ref{sec.cylinders}, i.e. $M=\Sigma \times (-1,1)$ and $\cN=\cP \times (-1,1)$. We consider $\MN$ as a marked 3-manifold, where the orientation on each component of $\cN$ is given by the natural orientation of $(-1,1)$. We will consider $\Sigma$ as a subset of $M$ by identifying $\Sigma$ with $\Sigma \times \{0\}$. There is a vertical projection map $\pr:M \to \Sigma$ which maps $(x,t)$ to $x$. We define $\cS\SP:= \cS\MN$.

\no{We define $\cS\SP:= \cS\MN$. Since $\SP$ is fixed, we will use the notation $\cS:=\cS\SP$ for the remainder of this subsection.}

An $\cN$-tangle $T$ in $M$ is said to have {\em vertical framing} if the framing vector at every point $p\in T$ is vertical, i.e. it is tangent to $p \times (-1,1)$ and has direction agreeing with the positive orientation of $(-1,1)$. 
 
 Suppose $T \subset \Sigma$ is a $\cP$-tangle. Technically $T$ may not be an $\cN$-tangle in $M$ since several strands of $T$ may meet at the same point in $\cP$, which is forbidden in the definition of an $\cN$-tangle. We modify $T$ in a small neighborhood of each point $p\in \cP$ by vertically moving the strands of $T$ in that neighborhood, to get an $\cN$-tangle $T'$ in $M$ as follows:

Equip $T$ with the vertical framing. If at a marked point $p$ there are $k=k_p$ strands $a_1, a_2, \dots, a_k$ of $T$ (in a small neighborhood of $p$) incident to $p$ and ordered in clockwise order, then we $\cN$-isotope these strands vertically so that $a_1$ is above $a_2$, $a_2$ is above $a_3$, and so on, see Figure \ref{fig:figures/simul}. The resulting $T'$ is an $\cN$-tangle whose $\cN$-isotopy class only depends on the $\cP$-isotopy class of $T$. Define $T$ as an element in $\cS\SP$ by
\begin{equation}\lbl{e.simul}
   T: = q^{\frac 14 \sum_{p\in \cP} k_p(k_p-1)} T' \in \cS.
\end{equation}
 
  \FIGc{figures/simul}{Left: There are 3 strands $a_1, a_2, a_3$ of $T$ coming to $p$, ordered clockwise. Right: The corresponding strands $a'_1, a'_2, a'_3$ of $T'$, with $a'_1$ above $a'_2$, and $a'_2$ above $a'_3$. Arrowed edges are part of the boundary, not part of the $\cP$-tangles}{2cm}

The factor which is a power of $q$ on the right hand side is introduced so that $T$ is invariant under the reflection involution, see Subsection \ref{sec.involution} below.

Recall that we write $B_\SP$ for the set of all $\cP$-isotopy classes of essential $\cP$-tangles in $\SP$. The following was first shown in \cite[Lemma 4.1]{Mu2012}.
\begin{lemma}\lbl{r.markedbasis}
$B_\SP$ is a basis of the free $R$-module $\cS\SP$.
\end{lemma}
We call $B_\SP$ the {\em preferred basis} of $\cS$.

For $0\neq x\in \cS$ one has the finite presentation 
\[
x= \sum _{i \in I} c_i x_i,  \quad   c_i \in R\setminus \{0\}, \ x_i \in B_\SP,
\]
and we define the {\em support} of $x$ to be the set $\supp(x)= \{ x_i \mid i\in I\}$. For $z\in B_\SP$ define
\begin{equation}
 \mu(z,x)=   \max_{x_i \in \supp (x) } \mu (z, x_i).
\end{equation}
Here $\mu(z,x_i)$ is the geometric intersection index defined in Subsection \ref{sec.intersectionindex}.
 
\begin{remark}
Equation \eqref{e.simul} describes the isomorphism between Muller's definition of the skein algebra of a totally marked surface $\SP$ in terms of multicurves of knots and arcs in $\SP$ and our definition of a skein algebra of a marked surface $\SP$ in terms of $\cP \times (-1,1)$-tangles in $(\Sigma \times (-1,1), \cP \times (-1,1))$.
\end{remark}

\subsection{Skein modules of finite type surfaces}

Recall that a framed link $L$ in $\fS \times (-1,1)$ consists of a compact 1-dimensional non-oriented submanifold of $\fS \times (-1,1)$ equipped with a smooth normal vector field such that each connected component of $L$ is diffeomorphic to $S^1$. By convention, the empty set is considered to be a framed link with no components that is isotopic only to itself. A framed link $L$ has {\em vertical framing} if the framing vector at every point $p \in L$ is vertical, i.e. is tangent to $p \times (-1,1)$ and has direction agreeing with the positive orientation of $(-1,1)$.

The {\em Kauffman bracket skein module} of $\fS$ at $q$, denoted by $\cS_q(\fS)$, is the free $R$-module spanned by all isotopy classes of framed links in $\fS \times (-1,1)$ subject to the {\em skein relation} $T=qT_++q^{-1}T_-$ with $T,T_+,T_-$ depicted in Figure \ref{fig:figures/skein1}.

A simple diagram $\alpha$ in $\fS$ defines a unique isotopy class of framed link in $\fS \times (-1,1)$ with vertical framing, which we will also denote as $\alpha$. Recall that $B_\fS$ denotes the set of all isotopy classes of simple diagrams in $\fS$.

\begin{lemma}\cite{PS2000}
The set $B_\fS$ is a basis of $\cS(\fS)$ over $R$.
\end{lemma}

\subsection{Algebra structure}

We define the $R$-algebra product structure in $\cS\SP$ as follows. For $\cN$-tangles $T_1, T_2$ in $(M,\cN)=(\Sigma \times (-1,1), \cP \times (-1,1))$ define the product $T_1 T_2$ as the result of stacking $T_1$ atop $T_2$ using the cylinder structure of $\MN$. More precisely, this means the following. Let $\iota_1: M \embed M$ be the embedding $\iota_1(x,t)= (x, \frac{t+1}2)$ and $\iota_2: M \embed M$ be the embedding $\iota_2(x,t)= (x, \frac{t-1}2)$. Then $T_1 T_2:= \iota_1(T_1) \cup \iota_2(T_2)$. This product makes $\cS\SP$ an $R$-algebra, which is non-commutative in general.

The product structure on $\cS(\fS)$ is defined identically. Here, since $\cP = \emptyset$, all $\emptyset$-tangles in $(\fS \times (-1,1),\emptyset)$ are the same as framed links.

\no{
We endow $\cS(\fS)$ with an $R$-algebra structure similarly. Let $L_1,L_2$ be two framed links in $\fS \times (-1,1)$. We define the product $L_1L_2$ by first isotoping $L_1$ into $\fS \times (0,1)$ and $L_2$ into $\fS \times (-1,0)$ and then taking the union of the two. Then $\cS(\fS)$ is an $R$-algebra which is typically non-commutative.
}
 
\subsection{Reflection anti-involution}\lbl{sec.involution}

Let $\heta: \cS\SP \to \cS\SP$ be the bar homomorphism of \cite{Mu2012}, i.e. the $\BZ$-algebra anti-homomorphism defined by
\begin{enumerate}[(i)]
\item $\heta(x) = \eta(x)$ if $x\in R$, and
\item if $T$ is an $\cN$-tangle with framing $v$ then $\heta(T)$ is $\refl(T)$ with the framing $-\refl(v)$, where $\refl$ is  the reflection which maps $(x,t)\to (x, -t)$ in $\Sigma \times (-1,1).$\end{enumerate}
It is clear that $\heta$ is an anti-involution. An element $z\in \cS$ is {\em reflection invariant} if $\heta(z)=z$.

The prefactor on the right hand side of \eqref{e.simul} was introduced so that every $\cP$-tangle $T$ is reflection invariant as an element of $\cS$. The preferred basis $B_\SP$ consists of reflection invariant elements.

Suppose $T$ is a $\cP$-tangle with components $x_1,\dots, x_k$. By the reordering relation (see Figure \ref{fig:figures/boundary}), any two components $x_i, x_j$ are $q$-commuting as elements of $\cS$ (and note that $x_ix_j=x_jx_i$ if at least one is a $\cP$-knot), and
\[
T= [x_1 x_2 \dots x_k]  \quad \text{in }\ \cS,
\]
where on the right hand side we use the Weyl normalization, see Subsection \ref{sec.weylnormalization}.

\subsection{Functoriality}\lbl{sec.surfunc}

Let $(\Sigma',\cP')$ be a marked surface such that $\Sigma'\subset \Sigma$ and $\cP'\subset \cP$. The morphism $\iota: (\Sigma',\cP') \embed \SP$ given by the natural embedding induces an $R$-algebra homomorphism $\iota_*:\cS(\Sigma', \cP')\to \cS\SP$.

\begin{proposition}\lbl{r.func}
Suppose $\cP'\subset \cP$. Then $\iota_*:\cS(\Sigma, \cP')\to \cS\SP$ is injective.
\end{proposition}
\begin{proof}This is because the preferred basis $B_{(\Sigma,\cP')}$ is a subset of $B_\SP$.
\end{proof}

\subsection{Domain}

The following was first shown in \cite{PS2018}, with special cases shown in \cite{Bu1997,PS2000}.
\begin{lemma}\lbl{r.finitedomain}
Let $\fS$ be a finite type surface. Then $\cS(\fS)$ is a domain.
\end{lemma}

Muller \cite{Mu2012} showed that $\cS\SP$ is a domain in the case where $\SP$ is totally marked. Extending this result to the case with unmarked boundary components is a simple application of functoriality.

\begin{lemma}\lbl{s.markeddomain}
Let $\SP$ be a marked surface. Then $\cS\SP$ is a domain.
\end{lemma} 
\begin{proof}
Let $\cP' \supset \cP$ be a larger set of marked points such that $(\Sigma,\cP')$ is totally marked. By Proposition \ref{r.func}, $\cS\SP$ embeds into $\cS (\Sigma,\cP')$ which is a domain by Muller's result. Hence $\cS\SP$ is a domain.
\end{proof}
Muller's proof that $\cS\SP$ is a domain when $\SP$ is totally marked depends on an arbitrary choice of total ordering on the set of $\cP$-knots that is shown not to matter. We give an alternate (though more difficult) proof of this result for the case where $\SP$ is a marked surface in Appendix \ref{a.domain} that does not utilize any arbitrary choice.

\subsection{Division algebra}\lbl{r.skeindivisionalg}

Let $\SP$ be a marked surface. It is shown in the proof of Theorem \ref{r.torus} that $\SP$ is an Ore domain, and so it has a division algebra which we denote $\tcS\SP$.

Let $\fS$ be a finite type surface. It is shown in \cite{FKL2017} that $\cS(\fS)$ is finitely generated as a module over its center, and by Lemma \ref{r.finitedomain}, $\cS(\fS)$ is a domain. Write $Z$ for the center of $\cS(\fS)$ and $\tZ$ for the field of fractions of $Z$. Then by Proposition \ref{r.divisionalgebras}, $\cS(\fS) \otimes_Z \tZ$ is a division algebra, which we will also denote as $\tcS(\fS)$ since a division algebra is unique when it exists by Proposition \ref{r.sandwich0}.

\subsection{\texorpdfstring{$D$}{D}-monomials}\lbl{s.dmonomials}

Suppose $\D$ is a $\cP$-quasitriangulation of a quasitriangulable marked surface $\SP$. Given a subset $D\subset \D$, a {\em $D$-monomial} is an element in $\cS$ of the form $\D^\bn$, where $\bn \in \BN^\D$ has that $\bn(a)=0$ if $a\not \in D$.

\no{
\subsection{Quantization of character variety}\lbl{s.quantizationcharacter}

Consider a marked 3-manifold $\MN$ with empty marked set: $\cN = \emptyset$, which we will just abbreviate as $M$. Consider the base ring $R=\BC$ and let $\xi \in \BC^\times$. In this section we will understand the Kauffman bracket skein module $\cS_\xi(M):=\cS_\xi(M,\emptyset)$ for $M$ over $\BC$ with quantum parameter $q$ set to $\xi$ and how it can be seen as a quantization of the $SL_2\BC$-character variety of $M$, denoted $\cR(M)$. We will abbreviate $\cS:=\cS_\xi(M)$.}

\section{Roots of unity}\lbl{s.rootsof1}

In this section we explore the structure and meaning of the skein algebra at different roots of unity. For a more detailed discussion see \cite{Si2004,Ma2011}.

When the quantum parameter $q$ is a root of unity the skein module behaves in a distinct fashion from the case where $q$ has infinite order. Furthermore, the behavior and meaning of the skein algebra has a lot of variance within possible choices of the root of unity. Our main results - the existence of the Chebyshev-Frobenius homomorphism, and the skein algebra being a maximal order - are only applicable when $q$ is a root of unity.

Throughout this section we fix an unmarked 3-manifold $(M,\emptyset)$, though in certain places we may assume that $\MN$ is a thickened surface. We write $\cS_q:=\cS_q(M,\emptyset)$. By convention, when $q$ is a complex number, we will write it as a Greek letter such as $\xi,\ve,\zeta,$ or $\nu$.

\subsection{\texorpdfstring{$\cS$}{S} at roots of unity}\lbl{sec.kauffmanroots}

Let $M:=(M,\emptyset)$ be a marked 3-manifold. The first skein module to be understood was when the quantum parameter $\xi$ was equal to $-1$. Under these conditions, the skein relation in $\cS(M)$ no longer sees crossings, see Figure \ref{fig:figures/plusone}.

\FIGc{figures/plusone}{Crossings don't matter when $\xi=-1$. This also means that $\cS_{-1}(M)$ has an algebra structure.}{3cm}
Under these conditions, the union map ~\eqref{e.unionmap} is a well-defined product. Bullock showed \cite{Bu1997} that, modulo its nilradical, $\cS_{-1}(M)$ is naturally isomorphic to the ring of regular functions $\cR(M,SL_2\BC)$ on $SL_2\BC$-character variety $X(M,SL_2\BC):=X(M)$ of $M$ (see Subsection \ref{sec.topcharvar}). This is one way to understand that $\cS(\fS)$ is a quantization of the $SL_2\BC$-character variety. The nilradical was later shown to be trivial \cite{PS2000,CM2012}. As discussed in Subsection \ref{sec.topcharvar}, unoriented knots in $M$ can be used to define trace functions on $X(M)$. To be more precise, given a link $L$ (e.g. a framed link with no framing) in $M$ with components $K_i$, the trace function $\tr_L: X(M) \to \BC$ associated to $L$ is given by
\[
\tr_L([\rho]) = \prod_i (-\tr(\rho(K_i)).
\]

It is easy to see that when $\xi=1$, crossings may also be ignored in the skein relation. It is then easy to find an isomorphism between $\cS_1$ and $\cS_{-1}$, however it is not natural.

The structure of $\cS_\xi$ when $\xi=\pm i$ is less straightforward than when $\xi = \pm 1$, is still closely related to the character variety.

This case was first understood by Sikora \cite{Si2004} in the case where $M$ is a cylinder over a finite type surface or certain submanifolds of a rational homology sphere. Additional insights were found by March\'e \cite{Ma2011} several years later.

\no{
Sikora found
}
\no{

\section{Unicity theorem}

The unicity theorem \cite{FKL2017} was presented in more general terms in Theorem \ref{r.unicity1}. We present here the specialization of this important theorem to the Kauffman bracket skein algebra of a finite type surface here. This first requires the introduction of the notion of the variety of classical shadows.

Let $\SP$ be a finite type surface, $R=\BC$ the base ring and $\xi \in \BC^\times$ a root of unity of order $n$, $m=\ord(\xi^4)$, and $\ve = \xi^{m^2}$. Note that an irreducible representation $\rho: \cS_\xi\SP \to M_d(\BC)$ defines a central character $\chi_\rho: Z(\cS_\xi\SP) \to \BC$ given by the trace. \red{finish later?}

\begin{theorem}\cite{FKL2017}
Let $\SP$ be a finite type surface and $\zeta$ a root of unity. There is a Zariski open dense subset $U$ of the variety of classical shadows $\mathscr{Y}_\zeta\SP$ such that each point of $U$ is the classical shadow of a unique (up to equivalence) irreducible representation of $\cS_\zeta\SP$. All irreducible representations with classical shadows in $U$ have the same dimension $N$ which is equal to the square root of the rank of $\cS_\zeta\SP$ over $Z(\cS_\zeta\SP))$. If a classical shadow is not in $U$, then it has at most $r$ non-equivalent irreducible representations, and each has dimension $\leq N$. Here $r$ is a constant depending on the surface $\SP$ and the root $\zeta$.
\end{theorem}

}

%% file: quantumtopology.tex
\chapter{Quantum Teichm\"uller theory}

In this chapter we quickly review some basic aspects of quantum Teichm\"uller theory. We present only aspects of this theory directly relevant to our work, for more background see e.g. \cite{Ka1998,Liu2009,ChFo1999,Pe2012,Le2017}.

In Section \ref{sec.qttheory} we give a brief overview of quantum Teichm\"uller theory and define the Chekhov-Fock algebra (a version of quantum Teichm\"uller space) associated to a triangulable marked surface equipped with a triangulation. The Chekhov-Fock algebra is thought of as a quantization of the shear coordinates of the Teichm\"uller space of the surface. In Section \ref{sec.qtmuller} we discuss the Muller algebra associated to a quasitriangulation of a quasitriangulable marked surface. The Muller algebra may be thought of as a quantization of the Penner coordinates of decorated Teichm\"uller space. In Chapter \ref{c.embed} we will show how to embed the skein algebra of a finite type or marked surface into a Chekhov-Fock algebra or a Muller algebra, respectively. Lastly, in Section \ref{sec.sheartoskein}, we briefly explain the connection between these two embeddings via the shear-to-skein map of L\^e.

\no{
\section{Quantum invariants}

\subsection{Overview}

\subsection{Relationship with classical invariants}
}
\no{
\section{Deformation quantization}
}

\section{Chekhov-Fock algebra}\lbl{sec.qttheory}

Quantum Teichm\"uller theory is a wide topic with origins in 3D quantum gravity. Considering a 3-manifold with boundary as a spacetime, the classical phase space of Einstein gravity on this manifold is the Teichm\"uller space of its 2D boundary according to Verlinde and Verlinde \cite{VeVe1989} and Witten \cite{Wi1990}. This can be understood as corresponding to the idea that lengths of geodesics are observables in Einstein gravity - this then is sometimes known as Liouville field theory because the Einstein field equations reduce to Liouville's equations here. Unlike 4D Einstein gravity which has no clear quantization, 3D quantum gravity has several candidate quantizations arising from the quantization of Teichm\"uller space.

The first two versions of quantum Teichm\"uller space were constructed by Chekhov and Fock \cite{ChFo1999} and Kashaev \cite{Ka1998}. They are closely related, and considered to be quantizations of the shear coordinates of Teichm\"uller space. More recently, the Muller algebra \cite{Mu2012} is considered to be a quantization of the Penner coordinates of decorate Teichm\"uller space. The relationship between these quantizations was investigated by L\^e \cite{Le2017} and we review this in Section \ref{sec.sheartoskein}. In this section we review the Chekhov-Fock construction.

One of the reasons for introducing the Chekhov-Fock algebra here is that it is the target of Bonahon and Wong's quantum trace map \cite{BW2011}, which L\^e used to show that the skew field of the skein algebra of a surface can be thought of as a coordinate-free version of quantum Teichm\"uller space of that surface \cite{Le2017}. We review this material because while we only utilize the skein coordinate map and the Muller algebra in the course of the proofs of many of our results, those results could in principle have also been proven using the quantum trace map and Chekhov-Fock algebras.

Throughout this section we fix a triangulable marked surface $\SP$ and a $\cP$-triangulation $\D$ (see Subsection \ref{sec.triangulations}).

\no{
\subsection{Deformation quantization of shear coordinates}
}

\subsection{Face matrix}\lbl{sec.facematrix}

In the case that $\SP$ is a triangulable marked surface, we may define the quantum Teichm\"uller space of $\SP$ in terms of the \emph{face matrix} $Q_\D$ associated to a $\cP$-triangulation $\D$. This will be be a quantum torus defined from $Q_\D$ known as a Chekhov-Fock algebra. The Chekhov-Fock algebra may also be defined for finite type surfaces, but we do not include a discussion of that here since we do not utilize it, but see e.g. \cite{Le2017} for a discussion of this case.

Let $\tau \in \cF(\D)$. We may assume that $\tau$ is not self-folded (see Remark \ref{r.noselffold}). Thus $\tau$ has three distinct edges which we denote $a,b,c$ in counterclockwise order. Then we define an anti-symmetric matrix $Q_\tau \in \text{Mat}(\D \times \D,\BZ)$ as follows.
\[
Q_\tau(a,b) = Q_\tau(b,c) = Q_\tau(c,a) = 1
\]
\[
Q_\tau(e,e')=0 \ \text{ if one of } e,e' \text{ is not in }\{a,b,c\}.
\]
In other words, $Q_\tau \in \text{Mat}(\D \times \D, \BZ)$ is the 0-extension of the following $\{a,b,c\}\times\{a,b,c\}$ matrix
\[
\left( \begin{array}{ccc}
0 & 1 & -1 \\
-1 & 0 & 1 \\
1 & -1 & 0 \end{array} \right).
\]

Define the \emph{face matrix} $Q=Q_\D \in \text{Mat}(\D \times \D, \BZ)$ by
\[
Q= \sum_{\tau \in \cF(\D)}Q_\tau.
\]

\begin{remark}
Our $Q$ is the same as $Q^\D$ of \cite{Mu2012} or the same as $-B$ of \cite{FST2008}, and is also known as the signed adjacency matrix. We say ``face matrix'' to emphasize the duality with ``vertex matrix''. This duality, and the ``Fourier transform'' between the Chekhov-Fock algebra and the Muller algebra is elaborated on in Subsection \ref{sec.sheartoskein}.
\end{remark}

\subsection{Chekhov-Fock algebra associated to a triangulable marked surface}\lbl{sec.chekhovfock}

In this subsection we assume that $\SP$ is either a triangulable marked surface with triangulation $\D$, or a finite type surface with an ideal triangulation $\D$. Let $Q=Q_\D$ be the face matrix associated to $\D$.\no{(see Subsection \ref{sec.facematrix})} Let $\fY_q(\D)$ be the quantum torus $\BT(-2Q)$ with basis variables $Z_a$, $a \in \D$, i.e.
$$
\fY(\D)=R \langle Z_a^{\pm 1}, a \in \D \rangle/(Z_aZ_b = q^{-2Q(a,b)}Z_bZ_a).
$$
This is known as (one version of) the \emph{Chekhov-Fock algebra} associated to $\D$. This is the same as $\fY^{(2)}(\D)$ given in \cite{Le2017}.

\section{Muller algebra}\lbl{sec.qtmuller}

In this section we introduce the Muller algebra associated to a quasitriangulation $\D$ of a quasitriangulable marked surface $\SP$. This construction is another quantum torus that, in some cases, could be considered to be a quantization of the Penner coordinates on decorated Teichm\"uller space \cite{Pe2012,Le2017}. Again, we will ultimately embed the skein algebra of a marked surface into this algebra in order to simplify the study of the skein algebra. This tactic effectively reduces the exponential complexity of evaluating skein relations to polynomial complexity.

Throughout this section we fix a quasitriangulable marked surface $\SP$.

\subsection{Vertex matrix}\lbl{sec.vmatrix}

Suppose $a$ and $b$ are $\cP$-arcs which do not intersect in $\Sigma \setminus \cP$.  We define a number $P(a,b)\in \BZ$ as follows.  Removing an interior point of $a$ from $a$, we get two {\em half-edges} of $a$, each of which is incident to exactly one vertex in $\cP$. Similarly, removing an interior point of $b$ from $b$, we get two half-edges of $b$. Suppose $a'$ is a half-edge of $a$ and $b'$ is a half-edge of $b$, and $p\in \cP$. If one of $a', b'$  is not incident to $p$, set $P_p(a',b')=0$. If both $a', b'$ are incident to $p$, define
$P_p(a',b')$ as in Figure \ref{fig:figures/vmatrix2}, i.e.
\[
P_p(a',b')= \begin{cases}
1  &\text{if $a'$ is clockwise to $b'$ (at vertex $p$)}\\
-1  &\text{if $a' $ is counter-clockwise to $b'$ (at vertex $p$)}.
\end{cases}
\]
\FIGc{figures/vmatrix2}{$P_p(a',b')=1$ for the left case, and $P_p(a',b')=-1$ for the right one. Here the shaded area is part of $\Sigma$, and the arrow edge is part of a boundary edge. There might be other half-edges incident to $p$, and they maybe inside and outside the angle between $a'$ and $b'$. }{2.5cm}

Now define 
\[
P(a,b)= \sum P_p(a',b'),
\]
where the sum is over all $p\in \cP$, all half-edges $a'$ of $a$, and all half-edges $b'$ of $b$. 

Suppose $\D$ is a $\cP$-quasitriangulation. Two distinct $a, b \in \D$ do not intersect in $\Sigma \setminus \cP$, hence we can define $P(a,b)$.

Let $P_\D \in \Mat(\D \times \D, \BZ)$, called the {\em vertex matrix} of $\D$, be the anti-symmetric $\D\times \D$ matrix defined by $P_\D(a,b)= P(a,b)$, with 0 on the diagonal. 

\begin{remark}
The vertex matrix was introduced in \cite{Mu2012}, where it is called the orientation matrix.
\end{remark}

\subsection{Muller algebra associated to a quasitriangulable marked surface}\lbl{sec.mulleralgebra}

Recall that an unmarked component is a connected component of $\pS$ not containing any marked points, and we write $\cH$ for the set of unmarked components. 
Note that two distinct elements of $\cH$ are not $\cP$-isotopic since otherwise $\Sigma$ is an annulus with $\cP=\emptyset$, which is ruled out since $\SP$ is quasitriangulable. For $\bk \in\BN^\cH$, define the following element of $\cT\SP$: 
\[
\cH^\bk:= \prod_{\beta \in \cH} \beta^{\bk(\beta)} \in \cT\SP.
\]

Let $\D$ be a quasitriangulation of $\SP$ and $P \in \Mat(\D \times \D, \BZ)$ the associated vertex matrix. Let $\fX(\D)$ be the quantum torus over $R[\cH]$ associated to $P$ with basis variables $X_a$, $a \in \D$.
That is,
\begin{align*}
\fX(\D) & = R[\cH]\langle x_a^{\pm 1}, a \in \D\ra
/(x_ax_b=q^{P(a,b)}x_bx_a).
\end{align*}
We call $\fX(\D)$ the {\em Muller algebra} associated to $\SP$ and $\D$.

\begin{remark}
  By definition, each $a\in \Delta$ is a $\cP$-arc, and will be considered as an element of the skein algebra $\cS\SP$. We will see from the reordering relation (Figure \ref{fig:figures/boundary}) that for each pair of $\cP$-arcs $a,b\in \D$,
\begin{equation}\lbl{eq.35}
ab = q^{P(a,b)} ba.
\end{equation}
\no{where $P\in \Mat(\D\times \D,\BZ)$ is the vertex matrix. (see Subsection \ref{sec.vmatrix}).}
\end{remark}

As a free $R[\cH]$-module, $\XD$ has a basis given by $\{X^\bn \mid \bn \in \BZ^\D\}$, where $X^\bn$ is a normalized monomial as defined in Subsection \ref{sec.aqt}. As a free $R$-module, $\XD$ has a basis given by $\{\cH^\bk \, X^\bn \mid \bk\in \BN^\cH, \bn \in \BZ^\D\}$.

Let $\fX_+(\D)$ be the $R[\cH]$-subalgebra of $\XD$ generated by $X_a, a\in \D$. Then $\fX_+(\D)$ is a free $R[\cH]$-module with a basis given by $\{ X^\bn \mid \bn \in \BN^\D\}$ and a free $R$-module with preferred basis $B_{\D,+}:=\{\cH^\bk \, X^\bn \mid \bk\in \BN^\cH, \bn \in \BN^\D\}$. Furthermore, $\fX_+(\D)$ has the following presentation as an algebra over $R[\cH]$: 
\[
\fX_+(\D)= R[\cH]\la X_a , a\in \D\ra /( X_a X_b = q^{P(a,b)}X_b X_a).
\]

The involution $\eta:R \to R$ of Proposition \ref{r.reflection} extends to an involution $\eta: R[\cH] \to R[\cH]$ by $\eta(rx)= \eta(r) x$ for all $r\in R$ and $x = \cH^\bk$ for all $\bk \in \BN^\cH$. As explained in Subsection \ref{sec.reflection}, $\eta$ extends to 
an anti-involution $\heta:\XD \to \XD$ so that $\heta(x)=\eta(x)$ for $x\in R[\cH]$ and $\heta(X^\bk)= X^\bk$.

\begin{remark}
In Subsection \ref{sec.flipandtransfer} we will investigate the behavior of the transfer isomorphisms between Muller algebras induced by a change of quasitriangulation - i.e. a change of coordinate map. These maps turn out to be simpler than the transfer isomorphisms for the Chekhov-Fock algebra, which is one main reason why we choose to use the Muller algebra formalism instead.
\end{remark}

\section{Shear-to-skein map}\lbl{sec.sheartoskein}

Let $\SP$ be a triangulable marked surface and $\D$ a $\cP$-triangulation with face matrix $Q$ (note that $\cH = \emptyset$ in this case, and that there are no monogon edges). Write $H$ for the $\D_\inn \times \D$ submatrix of $Q$. In \cite{Le2017}, the {\em shear-to-skein map}
\[
\psi_\D: \cY(\D) \to \fX(\D)
\]
is defined, where $\cY(\D)$ is the Chekhov-Fock algebra associated to $\D$ and $\fX(\D)$ is the Muller algebra associated to $\SP$ and $\D$. It is an injective \no{multiplicatively linear}$R$-algebra homomorphism defined uniquely by
\be\lbl{e.sheartoskein}
\psi_\D(Z^\bk) = X^{\bk H}.
\ee
Here $H$ is the $\oD \times \D$ submatrix of $Q$. Bonahon has pointed out that this relationship between the Chekhov-Fock algebra and the Muller algebra corresponds to the classical relationship between the shear and Penner coordinates of Teichm\"uller space. From this one may derive a duality between the face matrix and vertex matrix, see \cite{Le2017}.

%% file: embedding.tex
\chapter{Skein algebra embeddings}\lbl{c.embed}

In this chapter we introduce a powerful technical tool in the study of skein algebras: embeddings of the skein algebra into quantum tori. This technique has revolutionized the study of skein algebras, making many previously intractable questions answerable by studying simple properties of $q$-commuting Laurent polynomial algebras.

This strategy was first discovered by Bonahon and Wong in \cite{BW2011} with the introduction of the miraculous quantum trace map, which we review in Section \ref{sec.quantumtrace}. This map embeds the complicated skein algebra of a triangulable marked surface into a Chekhov-Fock algebra introduced in Section \ref{sec.chekhovfock}, which is a type of quantum torus corresponding to quantization of the shear coordinates of Teichm\"uller space and was the first instance in which the exponential complexity of resolving skein relations was reduced to the polynomial complexity of multiplying $q$-commuting Laurent polynomials. This opened a new avenue of attack for understanding skein algebras and all of our methods in this dissertation were inspired by this method.

While we do not make use of the quantum trace map in proving the existence of the Chebyshev-Frobenius homomorphism in Chapter \ref{c.cf}, L\^e and the author had previously utilized it to prove its existence in earlier drafts, and it is the more natural method to use for finite type surfaces.

\no{
This strategy was first discovered by Bonahon and Wong in \cite{BW2011}, in which they found an embedding of the Kauffman bracket skein algebra of a finite type surface in the Chekhov-Fock version of quantum Teichm\"uller space discussed in Section \ref{sec.qttheory}. We review in Section \ref{sec.quantumtrace} the definition of the quantum trace map.

Muller expanded this strategy with an easier-to-study embedding of the skein algebra of totally marked surfaces into the Muller algebra, a quantum torus defined in terms of an anti-symmetric matrix also defined by a triangulation much like the Chekhov-Fock algebra.
}

Side-by-side with the quantum trace map is Muller's skein coordinate map introduced in \cite{Mu2012}, which is another embedding of the skein algebra of a totally marked surface into a quantum torus corresponding to a quantization of the Penner coordinates of Teichm\"uller space. This embedding turns out to have transfer isomorphisms between quantum tori corresponding to different quasitriangulations that are easier to work with that the ones for the Chekhov-Fock algebra, which was ultimately why we chose to work with this quantization. In Section \ref{sec.torus} we generalize the Muller algebra to allow for marked surfaces equipped with a quasitriangulation, and introduce a further generalization in Chapter \ref{c.surgery} which allows one to define a surgery theory.

\section{Quantum trace map}\lbl{sec.quantumtrace}

The skein algebra of a finite type surface $\cS(\fS)$ can be found as an essential subalgebra of the Chekhov-Fock algebra associated to an ideal triangulation of $\fS$. This is via the Bonahon-Wong quantum trace map $\Trq^\D$ (see \cite{BW2011}[Proposition 29]). $\Trq^\D$ is very difficult to work with in practice. In \cite{Le2017}[Section 6], making use of the ``Fourier transform'' relationship between the Muller algebra and the Chekhov-Fock algebra of a triangulation, L\^e develops a map $\varkappa_\D: \cS \to \fY(\D)$ equal to $\Trq^\D$ but with a simple combinatorial definition in terms of ``$\D$-simple'' knots. We describe the minimum ingredients necessary to define $\varkappa_\D$ here.

Let $\SP$ be a triangulable marked surface, and suppose $\D$ is a triangulation of $\SP$. Let $x$ be a $\cP$-arc or $\cP$-knot. We say that $x$ is $\D${\em -simple} if $\mu(x,a)\leq 1$ for all $a \in \D$.

Suppose $x$ is $\D$-simple. After a $\cP$-isotopy we can assume that $x$ is $\D${\em -taut}, i.e. $\mu(x,a)$ is equal to the number of internal common points of $x$ and $a$, for all $a \in \D$. Let $\cE(x,\D)$ be the set of all edges $e$ in $\D$ such that $\mu(x,e) \neq 0$, and $\cF(x,\D)$ be the set of all triangles $\tau$ of $\D$ intersecting the interior of $x$.

A {\em coloring of }$(x,\D)$ is a map $C \in \BZ^{\circD}$ such that $C(e) = 0$ if $e \notin \cE(x,\D)$ and $C(e) \in \{1,-1\}$ if $e \in \cE(x,\D)$. A coloring $C$ of $(x,\D)$ is said to be {\em admissable} if for any triangle $\tau \in \cF(x,\D)$ intersecting $x$ at two edges $a$ and $b$, with notations of edges as in Figure \ref{fig:figures/nonadmissable}, one has that the tuple $(C(a),C(b))$ is not equal to $(-1,1)$. Denote by $\Col(\alpha,\D)$ the set of all admissable colorings of $\alpha$.

\FIGc{figures/nonadmissable}{Non-admissable case: $C(a)=-1$, $C(b)=1$.}{2cm}

In the proof of Theorem 6.8 of \cite{Le2017}, a map denoted $\varkappa_\D: \cS_q\SP \to \fY_q(\D)$ is given. It is defined such that for any $\D$-simple knot $\alpha$,
\begin{equation}\lbl{e.varkappa}
\varkappa_\D(\alpha) = \sum_{C \in \Col(\alpha,\D)}\left[ \prod_{a \in \D} Z_a^{C(a)} \right] \in \fY_q(\D),
\end{equation}
where the brackets denote Weyl normalization. It is shown that this map is equal to $\Trq^\D$.

\section{Skein coordinate map}\lbl{sec.torus}

Throughout this section we fix a quasitriangulable marked surface $\SP$ and write $\cS:=\cS\SP$.

The following extends a result of Muller \cite[Theorem 6.14]{Mu2012} for {\em totally} marked surfaces to the case of marked surfaces.

\begin{theorem}\lbl{r.torus}
Suppose the marked surface $\SP$ has a quasitriangulation $\D$
\begin{enumerate}[(a)]
\item There is a unique $R[\cH]$-algebra embedding $\vpD: \cS \embed \fX(\D)$ such that for all $a \in \D$, 
\begin{equation}\lbl{eq.vpD}
\vpD(a) = X_a.
\end{equation}

\item If we identify $\cS$ with its image under $\vpD$ then $\cS$ is sandwiched between $\fX_+(\D)$ and $\fX(\D)$, i.e.
\begin{equation}\lbl{eq.incl2}
\fX_+(\D) \subset \cS \subset \XD.
\end{equation}
Consequently $\cS$ is a two-sided Ore domain, and $\vpD$ induces an $R[\cH]$-algebra isomorphism 
\[
\tvpD: \widetilde \cS\SP \overset \cong \longrightarrow \widetilde \fX(\D),
\]
 where  $\widetilde \cS\SP$ and $ \widetilde \fX(\D)$ are the division algebras of $\cS\SP$ and $\fX(\D)$, respectively.

\item Furthermore, $\vpD$ is reflection invariant, i.e. $\vpD$ commutes with $\heta$.
\end{enumerate}
\end{theorem}

\begin{proof} The proof is a modification of Muller's proof for the totally marked surface case of Muller. To further simplify the proof, we will use Muller's result for the totally marked surface case.

(a) We first prove a couple lemmas.  
   
Relation \eqref{eq.35} shows that there is a unique $R[\cH]$-algebra homomorphism $\tau: \fX_+(\D) \to \cS$ defined by $\tau(X_a)=a$. Then for $\bn \in \BZ^\D$,
\[
\tau(X^\bn)= \D^\bn := \left [  \prod_{a\in \D} a^{\bn(a)} \right].
\]

Note that $\tau$ is injective since $\tau$ maps the preferred $R$-basis $B_{\D,+}$ of $\fX_+(\D)$ bijectively onto a subset of the preferred $R$-basis $B_{\SP}$ of $\cS$. We will identify $\fX_+(\D)$ with its image under $\tau$.

Recall that $\D=\D_\inn \cup \D_\bd$, where $\D_\inn$ is the set of inner edges and $\D_\bd$ is the set of boundary edges. Recall from Subsection \ref{s.dmonomials} the notion of $D$-monomial for $D \subset \D$.
\begin{lemma}\lbl{r.monomial} %
Let $x\in \cS$.
\begin{enumerate}[(i)]
\item If $D\subset \D$ there is an $D$-monomial $\fm$ such that $\mu(a, x \fm)=0$ for all $a\in D$.

\item There is an $\D_\inn$-monomial $\fm$ such that $ x\fm \in \fX_+(\D)$.
\end{enumerate}
\end{lemma}
\begin{proof}
(i) The following two facts are respectively \cite[Lemma 4.7(3)]{Mu2012}  and \cite[Corollary 4.13]{Mu2012}:
\begin{align}
&\mu(a, yz) \le \mu(a,y) + \mu(a,z) \quad \text{for all} \ a\in \D,\ y,z\in \cS
\lbl{eq.mu1}\\
&\mu(a, y\,  a^{\mu(a,y)})=0  \quad \text{for all} \ a\in \D,\ y\in \cS. \lbl{eq.mu2}
\end{align}
These results are formulated and proved for general marked surfaces in \cite{Mu2012}, not just totally marked surfaces. Besides, since any two edges in $\D$ have intersection index 0, we have
\begin{equation}\lbl{eq.mu3}
\mu(a, \fm) = 0 \quad \text{for all $\D$-monomials $\fm$ and all} \ a\in \D.
\end{equation}

Let $\bn\in \BZ^\D$ be given by $\bn(a)= \mu(a,x)$ for $a\in D$ and $\bn(a)= 0$ for $a\not\in D$. %
Then $\fm = \D^\bn$ is a $D$-monomial. Suppose $a\in D$.
By taking out the factors $a$ in $\fm$ and using~\eqref{e.normalizedtorus}, we have
\[
\fm = q^{k/2} a^{\mu(a,x)} \, \fm'
\]
where $\fm'$ is another $D$-monomial and $k\in \BZ$. Using \eqref{eq.mu2} and then \eqref{eq.mu1} and \eqref{eq.mu3}, we have
\[
\mu(a, x \fm) \le \mu(a, x a^{\mu(a,x)}) + \mu (a, \fm') =0,
\]
which proves $\mu(a, x \fm)=0$ for all $a\in D$.

(ii) Choose $\fm$ of part (a) with $D=\D_\inn$. Let $x_i \in \supp(x\fm)$. Clearly $\mu(a,x_i)=0$ for all $a \in \D_\bd$.
 Since $\mu(a,x\fm)=0$ for all $a\in \D_\inn$, one can find a $\cP$-tangle $x_i'$ which is $\cP$-isotopic to $x_i$ such that $x_i$ and $a$ are taut (see Lemma~\ref{l.FHS}), so that $x_i'\cap a=\emptyset$ in $\Sigma \setminus \cP$ for each $a\in \D$. The maximality in the definition of quasitriangulation shows that each component of $x_i'$ is $\cP$-isotopic to one in $\D \cup \cH$. It follows that $x_i=\cH^\bk \D^\bn$ in $\cS$ for some $\bk\in \BN^\cH$ and $\bn \in \BN^\D$. This implies $x\fm \in \fX_+(\D)$.
\end{proof}

\begin{lemma}\lbl{l.oresubset} The multiplicative set $\fM$ generated by  $\D$-monomials is a 2-sided Ore subset of $\cS$. Similarly the multiplicative set $\fM_\inn$ generated by $\D_\inn$-monomials is a 2-sided Ore subset of $\cS$.
\end{lemma}
\begin{proof}

By definition, $\fM$ is right Ore if for every  $x\in \cS$ and every $ u \in \fM$, one has $x \fM \cap u\cS \neq \emptyset$.

 By \eqref{e.normalizedtorus}, one has $u=q^{k/2} \D^\bn$ for some $k\in \BZ$, $\bn\in \BN^\D$. 
By Lemma \ref{r.monomial}, there is a $\D$-monomial $\fm$ such that $x\fm \in \fX_+(\D)$. Since $B_{\D,+}=\{ \cH^\bk \D^\bn \mid \bk \in \BN^\cH, \bn \in \BN^\D \}$ is the preferred $R$-basis of $\fX_+(\D)$, we have a finite sum presentation $x \fm= \sum_{i\in I} c_i  \cH^{\bk_i} \D^{\bn_i}$ with $ c_i \in R$. It follows that
\begin{align*}
 x \fM \ni x \fm u &= q^{k/2} \sum_{i\in I} c_i \cH^{\bk_i} \D^{\bn_i} \D^\bn= q^{k/2} \sum_{i\in I} c_i q^{\la \bn_i, \bn\rangle_P}\cH^{\bk_i} \D^\bn \D^{\bn_i}   \\
 &= q^{k/2}  \D^\bn \sum_{i\in I} c_i q^{\la \bn_i, \bn\rangle_P}\cH^{\bk_i} \D^{\bn_i} = u \sum_{i\in I} c_i q^{\la \bn_i, \bn\rangle_P}\cH^{\bk_i} \D^{\bn_i} \in u \cS,
\end{align*}
where the second equality follows from \eqref{e.normalizedtorus}. This proves $\fM$ is right Ore. Since the reflection anti-involution $\heta$ reverses the order of the multiplication and fixes each $\D$-monomial, $\fM$ is also left Ore. The proof that $\fM_\inn$ is Ore is identical, replacing $\fM$ by $\fM_\inn$ everywhere.
\end{proof}
Let us prove Theorem \ref{r.torus}(a). Since $\cS$ is a domain, the natural map $\cS \to \cS\fM^{-1}$, where $\cS\fM^{-1}$ is the Ore localization of $\cS$ at $\fM$, is injective.  Since Ore localization is flat, the inclusion $\tau :\fX_+(\D) \embed \cS$ induces an inclusion
\begin{equation}\lbl{eq.iso1}
  \tilde{\tau}:\fX_+(\D) \fM^{-1} \embed  \cS\fM^{-1}.
\end{equation}
 Note that $\fX_+(\D) \fM^{-1}= \fX(\D)$.
Let us prove $\tilde{\tau}$ is surjective.  After identifying $\fX_+(\D) \fM^{-1}$ as a subset of $\cS\fM^{-1}$ via $\tilde{\tau}$, it is enough to show that $\cS\subset \fX_+(\D) \fM^{-1}$.  This is guaranteed by Lemma \ref{r.monomial}, thus $\tilde{\tau}$ is an isomorphism.

Let $\vpD$ be the restriction of $\tilde{\tau}^{-1}$ onto $\cS$, then we have an embedding of $R[\cH]$-algebras $\vpD: \cS \embed \fX(\D)$ such that $\vpD \circ \tau$ is the identity on $\fX_+(\D)$. Any element $x\in\cS$ can be presented as $y \fm^{-1}$ with $y \in \fX_+(\D), \fm \in \fM$. This shows $\fX_+(\D)$ weakly generates $\cS$, and thus the uniqueness of $\vpD$ is clear.
 This proves (a).
 
(b) Inclusion \eqref{eq.incl2} follows from \eqref{eq.vpD}, and part (b) follows from Corollary \ref{r.sandwich}.

(c)  Let us prove that $\vpD$ is reflection invariant, i.e. for every $x \in \cS$ , we have
\begin{equation}\lbl{eq.4d}
  \vpD(\heta(x))= \heta(\vpD(x)).
\end{equation}
  Identity \eqref{eq.4d} clearly holds for the case when $x \in \D \cup \cH$. Hence it holds for $x$ in the $R$-algebra generated by $\D \cup \cH$, which is $\fX_+(\D) \subset \cS$. Since every element $x \in \cS$ can be presented in the form $y \fm^{-1}$, where $y \in \fX_+(\D), \fm \in \fM$, we also have \eqref{eq.4d} for $x$. This completes the proof of Theorem~\ref{r.torus}. \end{proof}

\begin{remark}\lbl{rem.Deltain} Lemma \ref{r.monomial} shows that $\vpD(\cS)$ lies in $\fX_+(\D) (\fM_\inn) ^{-1}$.
\end{remark}

\subsection{Flip and transfer}\lbl{sec.flipandtransfer}
For a quasitriangulation $\D$ of $\SP$, the map $\vpD: \cS\embed \sX(\D)$ will be called the {\em skein coordinate map}. We wish to understand how these coordinates change under change of quasitriangulation.
 
Let us first introduce the notion of a flip of a $\cP$-quasitriangulation.
\FIGc{figures/mutation}{Flip $a \to a^*$. Case 1}{3cm}
\FIGc{figures/quasiflip}{Flip $a \to a^*$. Case 2}{2.5cm}

Suppose $\D$ is a quasitriangulation of $\SP$ and $a$ is an inner edge in $\D$. The {\em flip of $\D$ at $a$} is a new quasitriangulation given by $\D' := \D \setminus \{a\} \cup \{ a^*\}$, where $a^*$ is the only $\cP$-arc not $\cP$-isotopic to $a$ such that $\D'$ is a quasitriangulation.
There are two cases:
\begin{itemize}
\item[\em Case 1.] $a$ is the common edge of two distinct triangles, see Figure \ref{fig:figures/mutation}.
\item[\em Case 2.] $a$ is the common edge of a holed monogon and a triangle, see Figure \ref{fig:figures/quasiflip}.
\end{itemize}
 In both cases, $a^*$ is depicted in Figures \ref{fig:figures/mutation} and \ref{fig:figures/quasiflip}.

Any two $\cP$-quasitriangulations are related by a sequence of flips, see eg \cite{Pe2012}, where a flip of case 2 is called a quasi-flip. If $\SP$ is a totally marked surface, then there is no flip of case 2. 

Suppose $\D, \D'$ are two quasitriangulations of $\SP$. Let
\[
\Theta_{\D,\D'}:= \tilde{\vp}_{\D'} \circ (\tvpD)^{-1}:  \tiX(\D)  \to \tiX(\D').
\]
By Theorem \ref{r.torus}, $\Theta_{\D,\D'}$ is an $R[\cH]$-algebra isomorphism from $\tiX(\D)$ to $\tiX(\D')$. We call $\Theta_{\D,\D'}$ the {\em transfer isomorphism from $\D$ to $\D'$}.

\begin{proposition} \lbl{r.42}
\begin{enumerate}[(a)]
\item The transfer isomorphism $\Theta_{\D,\D'}$ is natural. This means
\[
\Theta_{\D, \D}= \Id, \quad \Theta_{\D, \D''} = \Theta_{\D', \D''} \circ \Theta_{\D, \D'}.
\]

\item The skein coordinate maps $\vpD: \cS \embed \tiX(\D)$ commute with the transfer isomorphisms, i.e.
\[
\vp_{\D'} = \Theta_{\D,\D'} \circ \vpD.
\]

\item Suppose $\D'$ is obtained from $\D$ by a flip at an edge $a$, with $a$ replaced by $a^*$ as in  Figure \ref{fig:figures/mutation} (Case 1) or Figure \ref{fig:figures/quasiflip} (Case 2).
Identify $\cS$ as a subset of $\XD$ and $\tiX(\D')$. Then, with notations of edges as in Figure \ref{fig:figures/mutation} or Figure \ref{fig:figures/quasiflip}, we have 
\begin{align}\lbl{eq.aa0}
\Theta_{\Delta,\Delta'}(u)&= u \quad \text{ for} \ u\in \D \setminus \{a\},\\
\lbl{eq.aa1}
\Theta_{\Delta,\Delta'}(a) &=
 \begin{cases}
 [ce(a^*)^{-1}] + [ bd(a^*)^{-1}] \quad &\text{in Case 1}\\
 [b^2 (a^*)^{-1}] + [c^2 (a^*)^{-1}] + \beta [bc(a^*)^{-1}] &\text{in Case 2}.
 \end{cases}
\end{align}
\end{enumerate}
\end{proposition}
\begin{proof}
Parts (a) and (b) follow right away from the definition. Identity \eqref{eq.aa0} is obvious from the definition. 
Case 1 of \eqref{eq.aa1} is proven in \cite[Proposition 5.4]{Le2017}.

For case 2 of \eqref{eq.aa1}, we have that $\Theta_{\Delta,\Delta'}(a) = \tilde{\vp}_{\D'}(a)$. To compute this, we note that in $\cS$, $aa^*=q^2b^2+q^{-2}c^2+\beta bc$. Then
\begin{align*}
\tilde{\vp}_{\D'}(aa^*) & = \tilde{\vp}_{\D'}(q^2b^2) + \tilde{\vp}_{\D'}(q^{-2}c^2) + \tilde{\vp}_{\D'}(\beta bc), \\
\tilde{\vp}_{\D'}(a)\tilde{\vp}_{\D'}(a^*) &= q^2b^2+q^{-2}c^2+\beta bc, \\
\tilde{\vp}_{\D'}(a)a^* &=  q^2b^2+q^{-2}c^2+\beta bc
\end{align*}
Multiply both sides on the right by $(a^*)^{-1}$ and note that the $q$ factors agree with Weyl normalization. Therefore,
\[
\tilde{\vp}_{\D'}(a) = q^2b^2(a^*)^{-1} + q^{-2}c^2(a^*)^{-1} + \beta bc(a^*)^{-1} = [b^2(a^*)^{-1}] + [c^2(a^*)^{-1}] + \beta [bc(a^*)^{-1}].
\]
\end{proof}

%% file: surgerytheory.tex
\chapter{Surgery theory}\lbl{c.surgery}

In this chapter we introduce a new surgery theory for skein algebras of marked surfaces via embedding in a modified version of the Muller algebra we call the surgery algebra. The inclusion of unmarked boundary components in the theory allows us to describe how the skein coordinates change under modifications of surfaces. We consider two possible modifications: adding a marked point and plugging an unmarked boundary component with a disk. The results of this chapter will be used in the proof of the existence of the Chebyshev-Frobenius homomorphism in Chapter \ref{c.cf}, particularly for Proposition \ref{r.knot}.

In Section \ref{sec.suralg} we introduce the surgery algebra, a subalgebra of the Muller algebra which is still large enough for the skein algebra to embed into but which replaces the inverse of monogon edges with the flip of that monogon edge. In Section \ref{sec.addmarked} we study the simple case of adding a marked point and the associated transition maps between surgery algebras. In Section \ref{sec.holeplug} we discuss the more complicated case of plugging an unmarked boundary component with a disk and associated transition maps between surgery algebras.

Throughout this chapter, a quasitriangulable marked surface $\SP$ is fixed and we write $\cS:=\cS\SP$.

\section{Surgery algebra}\lbl{sec.suralg}

Let $\D$ be a quasitriangulation of $\SP$. 
Identify $\cS$ as a subset of $\XD$ using the skein coordinate map $\vD$. 
Recall that $\D_\mon$ is the set of all monogon edges. Let $\Dess:= \D \setminus \D_\mon$.

Suppose $\beta\in \cH$ is an unmarked component whose monogon edge is $a_\beta$. Let $\Sigma'$ be the result of gluing a disk to $\Sigma$ along $\beta$. We will say that $(\Sigma', \cP)$ is obtained from $\SP$ by {\em plugging 
 the unmarked $\beta$}.   Then $a_\beta$ becomes 0 in $\cS(\Sigma', \cP)$, while it is invertible in $\XD$. For this reason we want to find a subalgebra $\fXsurD$ of $\XD$ in which $a_\beta$ is not invertible, but we still have $\cS \subset \fXsurD$.

For $a\in \D_\mon$ choose a $\cP$-arc $a^*$ such that $\D\setminus \{a\} \cup \{a^*\}$ is a new quasitriangulation, ie it is the result of the flip of $\D$ at $a$. Let $\D_\mon^*:= \{ a^* \mid a \in \D_\mon\}$. 
For $a\in \D_\mon$ let $(a^*)^*=a$.

The {\em surgery algebra}  $\fXsurD$ is the $R[\cH]$-subalgebra of $\sX(\D)$ generated by $a^{\pm 1}$ with $a\in \Dess$, and  all $a\in  \D_\mon \cup \D_\mon^*$. Thus in $\fXsurD$ we don't have $a^{-1}$ (for  $a\in \D_\mon$) but we do have $a^*$, which will suffice in many applications. With the intention to replace $a^{-1}$ by $a^*$, we introduce the following definition: for $a\in \D$ and $k\in \BZ$ define
$$ a^{\{k\}} = \begin{cases} (a^*)^{-k} \qquad & \text{if  $a\in \D_\mon$ and $k <0$ }\\
a^k & \text{in all other cases.}\end{cases}
$$
For $\bk\in \BZ^\D$ define
$$ \D^{\{\bk\}} := \left [ \prod_{a\in \D} a^{\{\bk(a)\}}   \right].$$

\begin{proposition}\lbl{p.ztox}

(a)  As an $R[\cH]$-algebra, $\fXsurD$ is 
generated by $\D\cup \D_\mon^*$ and $a^{-1}$ for $a\in
\Dess$,  
subject to the following relations: 
 \begin{align}
xy &= q^{P(x,y)}yx \quad &\text{if}& \  (x,y) \neq (a, a^*) \ \text{for all } \ a\in (\D_\mon \cup \D_\mon^*)   \lbl{eq.PP} \\
a a^*& = q^2b^2+q^{-2}c^2+\beta bc, &\text{if}& \ a\in (\D_\mon \cup \D_\mon^*).
 \lbl{e.zrelations}
 \end{align}
Here, for  the case where $a \in \D_\mon \cup \D_\mon^*$, we denote by $\beta$  the unmarked boundary component surrounded by $a$, and the edges $b,c$  are  as in Figure \ref{fig:figures/quasiflip}.

(b) The skein algebra $\cS$ is a subset of $\fXsurD$ for any quasitriangulation $\D$.

(c) As an $R[\cH]$-module, $\fXsurD$ is free with basis
$$\BDs:= \{ \D^{\{\bk\}} \mid \bk \in \BZ^\D\}.$$
\end{proposition}

It should be noted that $P(x,y)$  in \eqref{eq.PP} is well-defined  since $x,y$ do not intersect in $\Sigma \setminus \cP$.

\begin{proof} (a) Let us redefine $\fXsurD$ so that it has the presentation given in the proposition. Recall that $\cS\subset \XD$, hence $a^*\in \XD$. Define an  $R[\cH]$-algebra homomorphism $\tsur: \fXsurD \to \XD$ by $\tsur(a)= a$ for all $a\in \D \cup \D_\mon^*$. Since all the defining relations of $\fXsurD$ are satisfied in $\XD$, the homomorphism $\tsur$ is well-defined. To prove (a) we need   to show that $\tsur$ is injective. 

 Let $\fXsurDp \subset \fXsurD$ be the $R[\cH]$-algebra generated by all $a\in \D \cup \D_\mon^*$.

\begin{lemma}\lbl{r.sinczp}
Let $T \subset \cS$ be an essential $\cP$-tangle such that $\mu(T,b)=0$ for all $b \in \Dess$. Then $T \in \fXsurDp$.
\end{lemma}

\begin{proof} After a $\cP$-isotopy we can assume $T$ does not intersect any $b\in \Dess$ in $\Sigma \setminus \cP$ by Lemma \ref{l.FHS}.
Cutting $\Sigma$ along $\Dess$, one gets a collection of ideal triangles and {\em eyes}. Here  an {\em eye} is a bigon with a small open disk removed, see Figure \ref{fig:figures/eye}. Each eye has  two  quasitriangulations. 
If $x$ is a component of $T$, then $x$, lying inside of a triangle or an eye, must be $\cP$-isotopic to an element in $\D \cup \D_\mon^* \cup \cH$, which implies $x \in \fXsurDp$. Hence $T \in \fXsurDp$.
\end{proof}
 
\FIGc{figures/eye}{An {\em eye} (left), and its two quasitriangulations.}{2cm}

Recall from Subsection \ref{s.dmonomials} the notion of $D$-monomial for $D \subset \D$.

\begin{lemma}\lbl{l.surgerybasis} 
(i) The  set 
$$
\BDps := \{ \D^{\{\bk\}} \mid \bk \in \BN^{\Dess} \times \BZ^{\D_{\mon}}\}.
$$
is an $R[\cH]$-basis for $\fXsurDp$.
The map $\tsur$ maps $\fXsurDp$ injectively into $\cS$.

(ii)  The multiplicative set $\fN$ generated by all $\Dess$-monomials is a two-sided Ore set of $\fXsurDp$ and $\fXsurD= \fXsurDp \fN^{-1}$.
\end{lemma}

\begin{proof} 
(i) Clearly $\fXsurDp$ is $R[\cH]$-spanned by monomials in elements of $\D \cup \D_\mon^*$. 
Since each $b\in \Dess$ will $q$-commute with any element of $\D \cup \D_\mon^*$, 
every monomial in elements of $\D \cup \D_\mon^*$ is equal to, up to a factor which is a power of $q$, an element of the form
\be
\lbl{eq.1q}
x= a_1 \dots a_l \,  \Dess^\bk, \quad \bk\in \BN^{\Dess},
\ee
where $a_i \in \D_\mon\cup \D_\mon^*$ and $\Dess^\bk$ is understood to be a $\Dess$-monomial. Each $a \in \D_\mon\cup \D_\mon^*$ commutes with every element of $\D_\mon\cup \D_\mon^*$ except for $a^*$. 
If for any $a\in \D_\mon \cup \D_\mon^*$, $\{a,a^*\} \not\subset \{a_1,\dots, a_l\}$, then  $x\in \BDps$, up to a factor which is a power of $q$.

If $\{a,a^*\} \subset \{a_1,\dots, a_l\}$, then we can permute the product to bring one $a$ next to one $a^*$, and relation \eqref{e.zrelations} shows that $x$ is equal to an $R[\cH]$-linear combination of elements of the form \eqref{eq.1q} each of which have a smaller number of $a,a^*$. Induction shows that elements $x$ of the form \eqref{eq.1q} are linear combinations of elements of the same form \eqref{eq.1q} in which not both $a$ and $a^*$ appear for every $a\in \D_\mon \cup \D_\mon^*$. This shows $\BDps$ spans $\fXsurDp$ as an $R[\cH]$-module.

The geometric realization of elements in $\BDps$ (i.e. their image in $\cS$) shows that $\tsur$ maps $\BDps$ injectively into the preferred basis $B_{\SP}$ of $\cS$.
 This shows that $\BDps$ must be $R[\cH]$-linearly independent, that $\BDps$ is an $R[\cH]$-basis of $\fXsurDp$, and that $\tsur$ maps $\fXsurDp$ injectively into $\cS$.
 
(ii) As $\fXsurDp$ embeds in $\cS$, which is a domain, $\fN$ contains only regular elements.
To get $\fXsurD$ from $\fXsurDp$ we need to invert all $a\in \Dess$. 
 As every $a\in \Dess$ will $q$-commute with any other generator, 
every element  of $\fXsurD$ has the form $x \fn^{-1}$ and also the form $(\fn')^{-1} x'$, where $x,x'\in\fXsurDp$ and $\fn, \fn'$ are $\Dess$-monomials. Thus $\fXsurD$ is a  ring of fractions of $\fXsurDp$ with respect to $\fN$. This shows that $\fN$ is a two-sided Ore set of $\fXsurDp$ and $\fXsurD = \fXsurDp \fN^{-1}$.
\end{proof}

Suppose $\tsur(x)=0$ where $x\in \cZ$. Then $x=y \fn^{-1}$ for some $y\in \fXsurDp$ and  $\fn\in \fN$. Hence $\tsur(y) = \tsur(x) \tsur(\fn)=0$. 
 Lemma \ref{l.surgerybasis} shows $y=0$. Consequently $x=0$. This proves the injectivity of~$\tsur$.

(b) Suppose $x \in \cS$. Using Lemma \ref{r.monomial} with $D=\Dess$, there is a $\Dess$-monomial $\mathfrak{n}$ such that   $\mu(a, x \mathfrak{n}) =0$ for all $a\in \Dess$. Then $x \mathfrak{n}$ is an $R$-linear combination of essential $\cP$-tangles which do not intersect any edge in $\Dess$ by Lemma \ref{l.FHS}. By Lemma \ref{r.sinczp}, it follows that $x \mathfrak{n} \in  \fXsurDp$.
Hence $x = (x \mathfrak{n}) \mathfrak{n}^{-1} \in \fXsurDp  \fN^{-1}  = \fXsurD$. This proves $\cS \subset \fXsurD$.

(c) As any element of $\fXsurD$ may be written as $x \fn^{-1}$, where $x\in\fXsurDp$ and $\fn$ is a $\Dess$-monomial, and $\BDps$ spans $\fXsurDp$ as an $R[\cH]$-module, we have that $\BDs$ spans $\fXsurD$ as an $R[\cH]$-module. On the other hand, suppose we have a  $R[\cH]$-linear combination of
$\BDs$ giving 0:
$$ \sum_i c_i\,  \D^{\{\bk_i\}} =0, \quad c_i \in R[\cH].$$
Multipliying on the right by $(\Dess)^\bk$ where $\bk(a)$ is sufficiently large for each $a\in \Dess$, we get
$$ \sum_i q^{l_i/2}\, c_i \, \D^{\{\bk'_i\}} =0, \quad l_i \in \BZ,$$
where $\bk'_i(a) \ge 0$ for all $a\in \Dess$. This means each $\D^{\{\bk'_i\}}$ is in $\BDps$, an $R[\cH]$-basis of $\fXsurDp$. It follows that $c_i=0$ for all $i$. Hence, $\BDs$ is 
linearly independent over 
$R[\cH]$, and consequently an $R[\cH]$-basis of $\fXsurD$.
\end{proof}

\begin{lemma}\lbl{l.skeinsurgeryextension} An $R$-algebra homomorphism
$g: \cS \to A$  extends to an $R$-algebra homomorphism $\fXsurD \to A$ if and only if $g(a)$ is invertible for all $a \in \Dess$, and furthermore the extension is unique.

\end{lemma}

\begin{proof} We have $\fXsurDp\subset \cS\subset \fXsurD= \fXsurDp \fN^{-1}$. By Proposition \ref{r.sandwich0}, $\fN$ is a two-sided Ore subset of $\cS$ and $\fXsurD= \cS \fN^{-1}$. The lemma follows from the universality of the Ore extension. 
\end{proof}

\section{Adding marked points}\lbl{sec.addmarked}

Let $p\in \pS \setminus \cP$, and $\cP'= \cP \cup \{ p\}$. The identity map $\iota: \Sigma\to \Sigma$ induces an $R$-algebra embedding $\iota_*:\cS\SP \to \cS(\Sigma,\cP')$, see Proposition \ref{r.func}. After choosing how to extend a $\cP$-quasitriangulation $\D$ to be a $\cP'$-quasitriangulation $\D'$, we will show  that $\iota_*$ has a unique extension to an $R$-algebra homomorphism $\Psi: \fXsurD \to \fXsur(\D')$ which makes the following diagram commute. The map $\Psi$ describes how the skein coordinates change.

\be \lbl{eq.cd1}
\begin{tikzcd}
\cS(\Sigma,\cP) \arrow[d,"\iota_*"] \arrow[r,hook,"\vpD"] & \fXsurD \arrow[d,"\Psi"] \\
\cS(\Sigma,\cP') \arrow[r,hook,"\vpDp"] & \fXsur(\D')
\end{tikzcd}
\ee

There are two scenarios to consider: adding a marked point to an unmarked boundary component or to a boundary edge.

\subsection{Adding a marked point to a boundary edge}\lbl{sec.addmarkedptboundary}

Suppose $a\subset \pS$ is a boundary $\cP$-arc of $\SP$ and $p$ is a point in the interior of $a$.
Let $\cP'=\cP \cup \{p\}$.  The set $\cH'$  of unmarked boundary components of the new marked surface $(\Sigma,\cP')$ is equal to~$\cH$.

Let $\D$ be a $\cP$-quasitriangulation of $\SP$. 
 Define a $\cP'$-quasitriangulation $\D'$ of $\Sigma$ by $\D':=\D \cup \{a_1,a_2\}$ as shown in Figure \ref{fig:figures/boundarypoint} (where we have $\cP'$-isotoped $a$ away from $\partial \Sigma$ in $\D'$).

\FIGc{figures/boundarypoint}{Adding a marked point to a boundary edge. The shaded part is a subset of the interior of $\Sigma$.}{2cm}

Recall that $\fXsurD$ is weakly generated by $\D \cup \D_\mon^*$, and that $\Dess= \D\setminus \D_\mon$.

\begin{proposition}\lbl{prop.addptsurgerytorus} There exists a unique $R[\cH]$-algebra homomorphism $\Psi: \fXsurD \to \fXsur(\D')$ which makes diagram \eqref{eq.cd1} commutative. Moreover, $\Psi$ is given by  $\Psi(a)=a$ for all $a \in \D \cup \D_\mon^*$.
\end{proposition}
\begin{proof} Identify $\cS\SP$ with its image under $\vpD$ and $\cS\SPp$ with its image under $\vpDp$.
Note that $\Dess \subset \D'_{\mathrm{ess}}$. Hence if $a \in \Dess$ then $\iota_*(a)=a$ is invertible in $\fXsur(\D')$.
That 
$\Psi$ exists uniquely follows from Lemma \ref{l.skeinsurgeryextension} because $\vpDp\circ \iota_*(a)$ is invertible for all $a \in \Dess$. That $\Psi(a) = a$ for all $a \in \D \cup \D_\mon^*$ follows immediately.
\end{proof}

\subsection{Adding a marked point to an unmarked component}

Suppose $\beta\in \cH$ is an unmarked boundary component of $\SP$. Choose a point $p \in \beta$ and set $\cP' = \cP \cup \{p\}$. We call the new marked surface $(\Sigma, \cP')$ and write $\cH'= \cH \setminus \beta$ for its set of unmarked boundary components. 

Suppose $\D$ is a quasitriangulation of $\SP$. 
Let $a \in \D_\mon$ be the monogon edge bounding the eye containing the unmarked boundary component $\beta$ (as defined in Figure \ref{fig:figures/eye}), and $b,c \in \D$ be the edges immediately clockwise and counterclockwise to $a$ as depicted on the left in Figure \ref{fig:figures/holepoint}. To get a triangulation of the eye containing $\beta$ with the added marked point $p$, we need to add 3 edges $d,e,f$ as depicted on the right side of Figure \ref{fig:figures/holepoint}. Here 
 $f$ is the boundary $\cP'$-arc whose ends are both $p$. By relabeling, we can assume that $e$ is counterclockwise to $d$ at $p$.
 Up to isotopy of $\Sigma$ fixing every point in the complement of the  monogon, there is only one choice for such $d$ and $e$. Then $\D'=\D \cup \{d,e,f\}$ is a quasitriangulation of $(\Sigma,\cP')$.

\FIGc{figures/holepoint}{From $\D$ to $\D'$.}{3.1cm}

Since $\cH\neq \cH'$,  it is not appropriate to consider modules and algebras over $R[\cH]$. 
Rather we will consider both $\cS\SP$ and $\cS\SPp$ as algebras over $R$. As an $R$-algebra, $\fXsurD$ is weakly generated by $\D \cup \D_\mon^*\cup \cH$.

\begin{proposition}\lbl{lem.markedpthole} There exists a unique $R$-algebra homomorphism $\Psi: \fXsurD \to \fXsur(\D')$ which makes diagram \eqref{eq.cd1} commutative. Moreover, for
$z \in \D \cup \D_\mon^*\cup \cH$ we have
\be \lbl{eq.1f}
\Psi(z) = 
\begin{cases} z  \quad &\text{if } \ z \notin \{\beta,a^*\}, \\
[d^{-1}e]+[ad^{-1}e^{-1}f]+[de^{-1}]   &\text{if } \ z = \beta,\\ 
[a^{-1}b^2] + [a^{-1}c^2] + [a^{-1}bcd^{-1}e] + [bcd^{-1}e^{-1}f] + [a^{-1}bcde^{-1}] &\text{if } \ z = a^*.
\end{cases} 
\ee
\end{proposition}

\begin{proof} Again $\Dess\subset \D'_{\mathrm{ess}}$, so the existence and uniqueness of $\Psi$ follows.
Formula  \eqref{eq.1f} follows from a simple calculation of the value of $\iota_*(z)$.
\end{proof}

\section{Plugging a hole}\lbl{sec.holeplug}

The more interesting operation is plugging a hole.

Fix an unmarked boundary component $\beta \in \cH$. Let $\Sigma'$ be the result of gluing a disk to $\Sigma$ along $\beta$. Then $(\Sigma', \cP)$ is another marked surface. The natural morphism
$ \iota: \SP \to \SpP$ gives rise to an $R$-algebra homomorphism $\iota_*: \cS\SP\to \cS\SpP$. Since $\iota_*$ maps the preferred $R$-basis $B_{\SP}$ of $\cS\SP$ onto a set containing the preferred $R$-basis $B_{(\Sigma',\cP)}$ of $\cS\SpP$, the map $\iota_*$ is surjective.
\FIGc{figures/surgery}{From $\D$ to $\D'$}{2.3cm}

Suppose $\D$ is a quasitriangulation of $\SP$.
Let $a\in \D$ be the monogon edge bounding the eye containing the unmarked boundary $\beta$ and $\tau$ be the triangle having $a$ as an edge. Let $a,b,c$ be the edges of $\tau$ in counterclockwise order, as in Figure \ref{fig:figures/surgery}. Let $\D'= \D \setminus \{a,b\}$. Then $\D'$ is a $\cP$-quasitriangulation of $\Sigma'$.

We cannot extend $\iota_*: \cS\to \cS'$ to an $R$-algebra homomorphism  $ \sX(\D) \to \sX(\D')$, since $\iota_*(a)=0$ but $a$ is invertible in $\XD$. %
This is the reason why we choose to work with the smaller algebra $\fXsurD$ which does not contain $a^{-1}$.

Recall that as an $R$-algebra, $\fXsurD$ is weakly generated by $\cH \cup \D_\mon^* \cup \D$.

\begin{proposition}\lbl{t.holetrick} There exists a unique $R$-algebra homomorphism 
$\Psi: \fXsurD \to \fXsur(\D')$ such that the following diagram
\be\lbl{d.holesurgery}
\begin{tikzcd}
\cS(\Sigma,\cP) \arrow[d,"\iota_*"] \arrow[r,hook,"\vpD"] & \fXsurD \arrow[d,"\Psi"] \\
\cS(\Sigma',\cP) \arrow[r,hook,"\vpDp"] & \fXsur(\D')
\end{tikzcd}
\ee
is commutative. Explicitly, $\Psi$ is defined on the generators in $\cH \cup \D_\mon^* \cup \D$ as follows: 
\begin{align}
\Psi(e)& = e  \quad \text{if} \ e  \in (\cH \cup \D_\mon^* \cup \D)\setminus \{a,a^*, b, \beta\}
\lbl{eq.2a}\\
\Psi(a)&= \Psi(a^*)=0, \ 
 \Psi(b)=c, \ 
 \Psi(\beta)= -q^2 - q^{-2}.
 \lbl{eq.2b}
\end{align}
The map  $\Psi$ is surjective and its  kernel is  the ideal $J$ of $\fXsurD$ generated by $a, a^*, b-c, \beta + q^2 + q^{-2}$. 
\end{proposition}
\begin{proof} Identify $\cS\SP$ with its image under $\vpD$ and $\cS\SPp$ with its image under $\vpDp$.

The existence and uniqueness of 
 $\Psi$ follows from Lemma \ref{l.skeinsurgeryextension}, since $ \iota_*(x)$ is invertible in $\fXsur(\D')$ for all $x \in \D \setminus \D_\mon$. By checking the value of $ \iota_*(x)$ for $x\in \cH \cup \D_\mon^* \cup \D$, we get \eqref{eq.2a} and \eqref{eq.2b}. 
 It follows that $J$ is in the kernel $\ker\Psi$. Hence $\Psi$ descends to an $R$-algebra homomorphism
$$ \bPsi: \fXsurD/J \to \fXsur(\D').$$
We will prove that $\bPsi$ is bijective by showing that there is  an $R$-basis $Y'$ of $\fXsur(\D')$ and an $R$-spanning set $\bY$  of $ \fXsurD/J$ such that $\bPsi$ maps $\bY$ bijectively onto $Y'$. Then $\bPsi$ is an isomorphism. Let 
\begin{align*}
Y &=  \{ \D^{\{\bk\}} \, (\cH)^\bn \mid \bk \in \BZ^{\D}, \bn \in \BN^{\cH}  \},   \\
Y_0 &=  \{ \D^{\{\bk\}} \, (\cH)^\bn \in X \mid \bk(a)= \bk(b)=\bn(\beta)=0   \} \subset X, \\
Y' &=\{ (\D')^{\{\bk\}} \, (\cH')^\bn \mid \bk \in \BZ^{\D'}, \bn \in \BN^{\cH'}  \}.   
\end{align*}
By Proposition \ref{p.ztox}(c), the sets $Y$ and $Y'$ are respectively $R$-bases of $\fXsurD$ and $\fXsur(\D')$. As $\D'= \D \setminus \{a,b\}$ and $\cH'= \cH\setminus \{\beta\}$, Formula \eqref{eq.2a} shows that $\Psi$ maps $Y_0$ bijectively onto $X'$. Consequently, the projection $\pi: \fXsurD\to \fXsurD/J$ maps $Y_0$ bijectively onto a set $\bY$ and
$\bPsi$ maps $\bY$ bijectively onto $Y'$.

It remains to be shown that  the $R$-span $R\la \bY\ra$ of $\bY$ equals $\fXsurD/J$. Suppose $x= \D^{\{\bk\}} \cH^\bn  \in \Psi(X) \setminus Y_0$. 
Then either $\bk(a)\neq 0$, or $\bk(b)\neq 0$, or $\bn(\beta)\neq 0$.

If $\bk(a)\neq 0$, then in $x$ there is factor of $a$ or $a^*$ which is in $J$, and hence $\pi(x)=0$.
Because $b-c$ and $\beta + q^2 + q^{-2}$ are in $ J$, in $\fXsurD/J$ we can replace $b$ by $c$  and $\beta$ by the scalar $-q^2 - q^{-2}$. Thus, $\pi(x) \in R\la \bY\ra$. As $Y$ spans $\fXsurD$, this shows $\bY$ spans $\fXsurD/J$. The proposition is proven.

\end{proof}

%% file: chebyshevfrobenius.tex
\chapter{Chebyshev-Frobenius homomorphism}\lbl{c.cf}

For the case when the marked set is empty,  Bonahon and Wong  \cite{BW2016-2} constructed a remarkable algebra homomorphism, called the Chebyshev homomorphism, from the skein algebra with quantum parameter $q=\xi^{N^2}$ to the skein algebra with quantum parameter $q=\xi$, where $\xi$ is a complex root of unity, and $N$ is the order (as a root of unity) of $\xi^4$.
In \cite{BW2016-2} the proof of the existence of the Chebyshev homomorphism is based on the theory of the quantum trace map \cite{BW2011}. Since the result can be formulated using only elementary skein theory, Bonahon and Wong asked for a skein theoretic proof of their results. Such a proof was offered in \cite{Le2015}.

Here we extend the result of Bonahon and Wong to the case of marked 3-manifolds. Our proof  is different  from the two above mentioned proofs even in the case of the marked set being empty; it does not rely on many computations but rather on the functoriality of the skein algebras.

\section{Setting}

Throughout this section we fix a marked 3-manifold  $(M,\cN)$. The ground ring $R$ is $\BC$. Let $\Cx$ denote the set of non-zero complex numbers. Recall that a root of unity is a complex number $\xi$ such that there exists a positive integer $N$ such that $\xi^N=1$, and the smallest such $N$ is called the order of $\xi$, denote by $\ord(\xi)$.

The skein module $\cS\MN$ depends on the choice of  $q=\xi \in \Cx$, and we denote the skein module with this choice by $\cS_\xi(M,\cN)$.  To be precise, we also  choose and fix one of the two square roots of $\xi$ for the value of $\xi^{1/2}$, but the choice of $\xi^{1/2}$ is not important due the symmetry of complex conjugation.

Similarly, if $\SP$ is a marked surface with quasitriangulation $\D$, then we use the notations $\cS_\xi\SP$, $\sX_\xi(\D), \fXsur_\xi(\D)$ to denote what were respectively the $\cS\SP, \sX(\D), \fXsurD$ of Subsections \ref{sec.torus} and \ref{sec.suralg} with ground ring $\BC$ and $q=\xi$. We will always identify $\cS_\xi\SP$ as a subset of $\sX_\xi(\D)$, its division algebra $\widetilde {\sX}_\xi(\D)$, and $\fXsur_\xi(\D)$.

\section{Formulation of the result}\lbl{sec.formulationmain} 

For $\xi\in \mathbb{C}^\times$ recall that $\cS_\xi\MN= \cT\MN/\Rel_\xi$, where $\cT\MN$ is the  $\BC$-vector space with basis the set of all $\cN$-isotopy classes of $\cN$-tangles in $M$  and $\Rel_\xi$ is the subspace spanned by the trivial loop relation elements, the trivial arc relation elements, and the skein relation elements, see Subsection \ref{s.def00}. For $x\in \cT\MN$ denote by $[x]_\xi$ its image in $\cS_\xi\MN= \cT\MN/\Rel_\xi$.

For an $\cN$-arc or an $\cN$-knot  $T$ in  $(M,\cN)$ and $k \in \BN$ let $T^{(k)}$ be {\em the $k$th framed power of $T$}, which is $k$ parallel copies of $T$ obtained using the framing, which will be considered as an $\cN$-tangle lying in a small neighborhood of $T$. The $\cN$-isotopy class of $T^{(k)}$ depends only on the $\cN$-isotopy class of $T$.

Given a polynomial $P(z) = \sum c_i z^i \in \BZ[z]$, and an $\cN$-tangle $T$ with a single component we define an element $P^{\text{fr}}(T) \in \cT\MN$ called the {\em threading of $T$ by $P$} by $P^{\text{fr}}(T) = \sum c_i T^{(i)}$. If $T$ is a $\cP$-knot in a marked surface $\SP$ then $P^{\text{fr}}(T)=P(T)$. Using the definition \eqref{e.simul} one can easily check that $P^{\text{fr}}(T)=P(T)$ for the case when $T$ is a $\cP$-arc as well.

Fix  $N\in \BN$. Suppose  $T_N(z)= \sum c_i z^i$ is the $N$th Chebyshev polynomial of type 1 defined by~\eqref{e.chebyshev}.
Define a $\BC$-linear map
$$ \hPhi_{N}: \LMN \to \LMN$$ so that  if $T$ is an $\cN$-tangle then  $\hPhi_N(T)\in \LMN$ is the union of $a^{(N)}$ and $(T_N)^{\text{fr}}(\al)$ for each $\cN$-arc component $a$ and each $\cN$-knot component $\al$ of $T$.  See Subsection \ref{sec.func} for the precise definition of union for skein modules. In other words, $\hPhi_N$ is given by threading each $\cN$-arc by $z^N$ and each $\cN$-knot by $T_N(z)$. More precisely, if
the $\cN$-arc components of $T$ are $a_1, \dots, a_k$ and the $\cN$-knot components  are $\al_1, \dots, \al_l$, then

\be\lbl{eq.def1}
\hPhi_N(T) = \sum_{0\le j_1, \dots, j_l\le N} c_{j_1} \dots c_{j_l}  a_1^{(N)}
 \cup \dots \cup  a_k^{(N)} \cup \, \al_1^{(j_1)} \cup \cdots \cup \al_l^{(j_l)} \ \in \  \LMN.
 \ee
See Section \ref{sec.func} for the precise definition of the union map in~\eqref{eq.def1}.

\begin{theorem}\lbl{t.ChebyshevFrobenius}
Let $(M,\cN)$ be a marked 3-manifold and  $\xi$ be a complex root of unity. Let $N:=\text{ord}(\xi^4)$ and $\ve:=\xi^{N^2}$. 

There exists
a unique $\BC$-linear map $\Phi_\xi: \cS_\ve(M,\cN) \to \cS_\xi(M,\cN)$ such that if 
$x\in \cS_\ve(M,\cN)$ is presented by an $\cN$-tangle $T$ then 
$
\Phi_\xi(x) = [\hPhi_{N}(T)]_\xi
$
in $ \Sx\MN$. 

In other words, the map $\hPhi_N: \cT\MN \to \cT\MN$ descends to a well-defined map $\Phi_\xi: \cS_\ve(M,\cN) \to \cS_\xi(M,\cN)$, meaning that the following diagram commutes. 

\[
\begin{tikzcd}\lbl{dia.Phixi}
\LMN  \arrow[d,"{[\,\cdot\,]_\ve}"] \arrow[r,"\hPhi_N"] & \LMN \arrow[d,"{[\, \cdot \, ]_\xi}"] \\
\cS_\ve\MN  \arrow[r,"\Phi_\xi"]& \cS_\xi\MN
\end{tikzcd}
\]
\end{theorem}
 Note that if $\ord(\xi^4)=N$ and $\ve=\xi^{N^2}$, then $\ve\in \{\pm 1, \pm i\}$. If $\cN=\emptyset$, then the skein module $\Se\MN$ with $\ve\in \{\pm 1, \pm i\}$
has an interpretation in terms of classical objects and is closely related to the $SL_2$-character variety, see \cite{Tu1991,Bu1997,PS2000,Si2004,Ma2011} or Section \ref{s.rootsof1}.

We call $\Phi_\xi$ the {\em Chebyshev-Frobenius homomorphism}. As mentioned, for the case when $\cN=\emptyset$ (where there are no arc components), Theorem \ref{t.ChebyshevFrobenius} was proven in \cite{BW2016-2} with the help of the quantum trace map, and was reproven in \cite{Le2017} using elementary skein methods. 
 We will prove Theorem \ref{t.ChebyshevFrobenius} in Subsection \ref{sec.mainthm}, using a result on skein algebras of triangulable marked surfaces discussed below, which is also of independent interest.

\section{Proof of theorem}

\subsection{Independence of triangulation problem}

Let $\D$ be a triangulation of a (necessarily totally) marked surface $\SP$. 
Suppose $N$ is a positive integer and $\xi \in \Cx$ not necessarily a root of unity.  For now we do not require $N=\ord(\xi^4)$.  Let $\ve=\xi^{N^2}$.

By Proposition~\ref{r.Frobenius}, we have a $\BC$-algebra embedding (the Frobenius homomorphism):
\begin{align}
F_N: \sX_{\ve}(\D) \to \sX_\xi(\D),\quad
\lbl{eq.defPhi1}
F_N(a)& = a^N\quad \text{for all} \ a \in \D.
\end{align}

Consider the embedding $\vpD: \cS_\xi \SP\embed \sX_\xi(\D)$ as a coordinate map depending on a triangulation. If we try to define a function on $\cS_\xi\SP$ using the coordinates, then we have to ask if the function is well-defined, i.e. it does not depend on the chosen coordinate system. Let us look at this problem for the Frobenius homomorphism. 

Identify $\cS_\nu \SP$ as a subset of $\sX_\nu(\D)$ via $\vpD$, for $\nu=\ve,\xi$. We investigate when a dashed arrow exists in the following diagram that makes it commute.
\be \lbl{dia.sx}
\begin{tikzcd}
\cS_\ve \SP  \arrow[hookrightarrow]{r} \arrow[dashed,"?"]{d}& \sX_\ve(\D)  \arrow[d, "F_N"] \\
\cS_\xi \SP \arrow[hookrightarrow]{r} & \sX_\xi(\D)
\end{tikzcd}
\ee

We answer the following questions about $F_N$:

A. For what $\xi \in \Cx$ and $N\in \BN$ does $F_N$  restrict to a map from $\cS_{\ve}\SP$ to $\cS_{\xi}\SP$ and the restriction does not depend on the triangulation $\D$?

B.  If $F_N$ can restrict to such a map, can one define the restriction of $F_N$ onto $\cS_{\ve}\SP$ in an intrinsic way, not referring to any triangulation $\D$?

The answers are given in the following two theorems.

Question A is answered by the following theorem.
\begin{theorem} \lbl{thm.A} Suppose $\xi\in \Cx$ and $N\ge 2$ and $\ve=\xi^{N^2}$. Assume that $\SP$ has at least two different triangulations.  
If $F_N: \sX_{\ve}(\D) \to \sX_\xi(\D)$
restricts to a map $\cS_\ve\SP \to \cS_\xi\SP$ for all triangulations $\D$ and the restriction does not depend on the triangulations, then $\xi$ is a root of unity and $N=\ord(\xi^4)$.
\end{theorem}

Question B is answered by the following converse to Theorem \ref{thm.A}.
\begin{theorem}
\lbl{thm.surface}
Suppose $\SP$ is a triangulable surface and $\xi$ is a root of unity. Let  $N =\ord(\xi^4)$ and $\ve=\xi^{N^2}$. Choose a triangulation $\D$  of $\SP$.  

(a) The map $F_N$ restricts to a $\BC$-algebra homomorphism  $F_\xi:\cS_\ve\SP \to \cS_\xi\SP$ which does not depend on the triangulation $\D$.

(b) If $a$ is a $\cP$-arc, then $F_\xi(a) = a^N$, and if $\al$ is a $\cP$-knot, then $F_\xi(\al)= T_N(\al)$.
\end{theorem}

We prove Theorem \ref{thm.A} in Subsection \ref{sec.triind} and Theorem  \ref{thm.surface} in Subsection \ref{sec.trisurfacethm}.

\subsection{Division algebra}\lbl{sec.triind} Assume $\SP$ is triangulable, $\xi \in \Cx$, and $N\in \BN$. Choose a triangulation $\D$ of $\SP$.
Let $\tXeD$ and $\tXzD$  be the division algebras of $\XeD$ and $\XzD$,  respectively. The $\BC$-algebra embedding $F_N:\XeD\to \XzD$ extends to a $\BC$-algebra embedding $$
\tF_N: \tXeD \to \tXzD.
$$ 

For $\nu=\ve,\xi$ let $\tS_\nu\SP$
be the division algebra of   $\cS_\nu\SP$. By Theorem \ref{r.torus} the embedding 
$\vpD : \cS_\nu\SP \embed \sX_\nu(\D)$ induces an isomorphism $\tvpD : \tcS_\nu\SP \overset \cong \longrightarrow \tX_\nu(\D)$. Diagram~\eqref{dia.sx} becomes
\be \notag
\begin{tikzcd}
\tS_\ve \SP  \arrow[rightarrow,"\cong"]{r}[swap]{\tvpD} %
& \tX_\ve(\D)  \arrow[d, "\tF_N"] \\
\tS_\xi \SP \arrow[rightarrow,"\cong"]{r}[swap]{\tvpD} & \tX_\xi(\D)
\end{tikzcd}
\ee
By pulling back $\tF_N$ via $\tvpD$, we get a $\BC$-algebra embedding

\be
\tF_{N,\D}: \tSe\SP \to  \tSz\SP,
\ee
which a priori depends on the $\cP$-triangulation $\D$.

\begin{proposition}\lbl{prop.two} Let $\SP$ be a triangulable marked surface, $\xi\in \Cx$, and $N\in \BN$. %

(a) If $\xi$ is a root of unity and 
 $N:=\text{ord}(\xi^4)$, 
then $\tF_{N,\D}$   does not depend on the triangulation $\D$.

(b)  Suppose $\SP$ has at least 2 different triangulations and $N \ge 2$. Then $\tF_{N,\D}$   does not depend on $\D$ if and only if $\xi$ is a root of unity and $N=\ord(\xi^4)$.

\end{proposition}
\begin{remark} A totally marked surface $\SP$ has at least 2  triangulations if and only if it is not a disk with less than 4 marked points.
\end{remark}

\begin{proof} As (a) is a consequence of (b), let us prove (b).

By Proposition \ref{r.42}, the map $\tF_{N,\D}$   does not depend on $\cP$-triangulations $\D$ if and only if the diagram
\be 
\begin{tikzcd}
\tiX_\ve(\D) \arrow[r,"\Theta_{\D,\D'}" ] \arrow[d, "\tF_N" ]& \tiX_\ve(\D')  \arrow[d, "\tF_N" ] \\
\tiX_\xi(\D) \arrow[r,"\Theta_{\D,\D'}" ] & \tiX_\xi(\D')
\lbl{dia.xx}
\end{tikzcd}
\ee
is commutative for any two $\cP$-triangulations $\D,\D'$. Since any two $\cP$-triangulations are related by a sequence of flips, in \eqref{dia.xx} we can assume that   $\D'$ is obtained from $\D$ by a flip at an edge $a\in \D$, with  the notation as given in Figure \ref{fig:figures/mutation}. Then  $\D'= \D\cup \{a^*\} \setminus \{a\}$.
The commutativity of \eqref{dia.xx} is equivalent to
\be 
\lbl{eq.aa3}
(\tF_N \circ \Theta_{\D,\D'})(x) = (\Theta_{\D,\D'} \circ \tF_N) (x), \quad \text{for all } \ x\in \tiX_\ve(\D).
\ee
 Since  $\D$ weakly generates the algebra $\tiX_\xi(\D)$, it is enough to show \eqref{eq.aa3} for $x\in \D$.

If $x \in \D \setminus \{ a\}$, then by \eqref{eq.aa0} one has $\Theta_{\D,\D'}(x)= x$, and hence we have \eqref{eq.aa3} since both sides are equal to $x^N$ in $\tiX_\ve(\D')$.
Consider the remaining case $x=a$.  By \eqref{eq.aa1}, we know that
\be 
\lbl{eq.aa2}
\Theta_{\D,\D'}(a)  = X + Y, \quad \text{where} \ X=[  b  d  (a^*)^{-1}], \ Y=[  c  e  (a^*)^{-1}].
\ee 
Using the above identity and the definition of $\tF_N$, we  
calculate the left hand side of \eqref{eq.aa3}:
\begin{align}  \lbl{eq.l0}
(\tF_N \circ \Theta_{\D,\D'})(a) & =  \tF_N\left( [  b  d  (a^*)^{-1}] + [  c  e  (a^*)^{-1}] \right) \\
& =  [  b^N  d^N  (a^*)^{-N}] + [  c^N  e^N  (a^*)^{-N}]= X^N + Y^N. \nonumber
\end{align}

Now we calculate the right hand side of \eqref{eq.aa3}:
\begin{align}
\lbl{eq.l1}
(\Theta_{\D,\D'} \circ \tF_N)(a) &= \Theta_{\D,\D'} (a^N)  
= \left(\Theta_{\D,\D'} (a)\right)^N     =  (X+Y)^N.
\end{align} 
Comparing \eqref{eq.l0} and \eqref{eq.l1}, we see that \eqref{eq.aa3} holds if and only if 
\be  (X+Y)^N= X^N + Y^N
\lbl{eq.gauss}
\ee From the $q$-commutativity of elements in $\D'$ one can check that $XY = \xi^4 YX$. By the Gauss binomial formula (see e.g. \cite{KC2001}),
$$ (X+Y)^N= X^N + Y^N + \sum_{k=1}^{N-1} \binom Nk_{\xi^4} Y^k X^{N-k},\ \text{
where} \
\binom Nk_{\xi^4}= \prod_{j=1}^k \frac{1- \xi^{4(N-j+1)}}{1-\xi^{4j}}.$$
Note that $Y^k X^{N-k}$ is a power of $\xi$ times a monomial in $b,c,d,e$, and $(a^*)^{-1}$, and these monominals are distinct for $k=0,1,\dots, N$. As monomials (with positive and negative powers) in edges form a $\BC$-basis of $\sX_\xi(\D')$, we see that $(X+Y)^N= X^N + Y^N$ if and only if 
\be  \binom Nk_{\xi^4} =0 \ \text{for all } k = 1,2, \dots, N-1. \lbl{eq.gauss2}.
\ee
It is well-known, and easy to prove, that \eqref{eq.gauss2} holds if and only if $\xi^4$ is a root of unity of order $N$. %
\end{proof}

As the edge $a$ in the proof of Proposition \ref{prop.two} is in $\cS\SP$,  Theorem \ref{thm.A} follows immediately.

\subsection{Frobenius homomorphism \texorpdfstring{$\tF_\xi:=\tF_{N,\D}$}{Fxi=FND}}

For the remainder of this chapter let $\xi \in \Cx$ be a root of unity, $N=\ord(\xi^4)$, $\ve=\xi^{N^2}$. Suppose $\SP$ is a triangulable marked surface. Since $\tF_{N,\D}$ does not depend on $\D$ and $N=\ord(\xi^4)$, we denote 
\[
\tF_\xi:= \tF_{N,\D}: \tcS_\ve\SP \to \tcS_\xi\SP.
\]

\subsection{Arcs in \texorpdfstring{$\SP$}{(S,P)}}

\begin{proposition} 
\lbl{r.arc5}
Suppose $a\subset \Sigma$ is a $\cP$-arc. Then $\tF_\xi(a)= a^N$. 
\end{proposition}
\begin{proof} Since $a$ is an element of a $\cP$-triangulation $\D$, we have $\tF_\xi(a)= \tF_{N,\D}(a) = a^N$. 
\end{proof}

\subsection{Functoriality}

\begin{proposition}
\lbl{prop.three} Suppose $\SP$ and  $(\Sigma', \cP')$ are triangulable marked surfaces such that  $\Sigma\subset \Sigma'$ and $\cP\subset \cP'$. For any $\zeta \in \Cx$, the embedding $\iota: \SP \embed (\Sigma',\cP')$ induces a $\BC$-algebra homomorphism $\iota_*: \widetilde \cS_\zeta(\Sigma, \cP) \to \widetilde \cS_\zeta(\Sigma', \cP')$.

Let $\xi\in \Cx$ be a root of unity, $N =\ord(\xi^4)$ and $\ve= \xi^{N^2}$. Then  the following diagram commutes.

\[
\begin{tikzcd}
\ttS_\ve\SP \arrow[r,"\iota_*"] \arrow[d, "\tF_\xi"]& \ttS_\ve(\Sigma',\cP') \arrow[d, "\tF_\xi"] \\
\ttS_\xi\SP \arrow[r,"\iota_*"] & \ttS_\xi(\Sigma',\cP')
\end{tikzcd}
\]

\end{proposition}

\begin{proof} If $a\subset \Sigma$ is a $\cP$-arc, then it is also a $\cP'$-arc in $\Sigma'$. Hence by Proposition \ref{r.arc5}, both $\iota_*\circ \tF_\xi(a)$ and $\tF_\xi\circ \iota_*(a)$ are equal to $a^N$ in $\ttS_\xi(\Sigma',\cP')$. Since for a triangulable marked surface $\SP$, the set of all sums of $\cP$-arcs and their inverses generates $\tXeD=\widetilde\cS_\ve\SP  $, we have the commutativity of the diagram.
\end{proof}

\subsection{Knots in \texorpdfstring{$\cS\SP$}{S(S,P)}}

We find an intrinsic definition of $\tF_\xi(\al)$, where $\al$ is a $\cP$-knot. 
\begin{proposition}\lbl{r.knot} Suppose $\SP$ is a triangulable marked surface, $\xi$ is a root of unity, and $N=\ord(\xi^4)$.
 If $\al$ is a $\cP$-knot in $\SP$, then $\tF_\xi(\al) = T_N(\al)$.
\end{proposition}

We break the proof of Proposition \ref{r.knot} into lemmas.
\begin{lemma}\lbl{r.unknot}
(a) Proposition \ref{r.knot} holds if $\al$ is a trivial $\cP$-knot, i.e. $\al$ bounds a disk in $\Sigma$.

(b) If $\xi$ is a root of unity with $\ord(\xi^4)=N$ and $\ve=\xi^{N^2}$, then
\be 
\lbl{eq.vexi}
T_N(-\xi^2-\xi^{-2})= -\ve^2 - \ve^{-2}.
\ee
\end{lemma}
\begin{proof} Let us prove (b) first. The left hand side and the right hand side of \eqref{eq.vexi} are
\begin{align}
\lbl{eq.xx1}
LHS&=T_N(-\xi^2 -\xi^{-2})=(-\xi^2)^N + (-\xi^{-2})^{N}=(-1)^N(\xi^{2N} + \xi^{-2N})\\
RHS&= -\ve^2 - \ve^{-2} = -\xi^{2N^2} -\xi^{-2N^2}. \lbl{eq.xx2}
  \end{align}
  Since $\ord(\xi^4)=N$, either $\ord(\xi^2)=2N$ or $\ord(\xi^2)=N$.
  
  Suppose $\ord(\xi^2)=2N$. Then both  right hand sides of \eqref{eq.xx1} and \eqref{eq.xx2} are equal to $2(-1)^N$, and so they are equal.
  
  Suppose $\ord(\xi^2)=N$. Then $N$ must be odd since otherwise $\ord(\xi^4)=N/2$. 
Then both  right hand sides  of \eqref{eq.xx1} and \eqref{eq.xx2} are equal to $-2$. This completes the proof of (b).

(a) Since $\al$ is a trivial knot, $\al= -\ve^2 - \ve^{-2}$ in $\cS_\ve\SP$ and $\al= -\xi^2 -\xi^{-2}$ in $\cS_\ve\SP$. Hence
\begin{align*}
\tF_\xi(\al) &= \tF_\xi (-\ve^2 - \ve^{-2})= -\ve^2 - \ve^{-2}= T_N(-\xi^2 -\xi^{-2})=
T_N(\al),
\end{align*} 
where the third identity is part (b). Thus $\tF_\xi(\al) = T_N(\al)$.
\end{proof}

\begin{lemma}\lbl{r.markedannulus}  Proposition \ref{r.knot} holds if $\Sigma = \An$, the annulus, and $\cP\subset \partial \An$ consists of two points, one in each connected component of $\partial \An$.

\end{lemma}

\begin{proof} If $\al$ is a trivial $\cP$-knot, then the result follows from  Lemma \ref{r.unknot}. We assume $\al$ is non-trivial. Then $\al$ is the core of the annulus, i.e. $\al$ is a parallel of a boundary component of $\An$.
Let $\D=\{a,b,c,d\}$ be the triangulation of $(\BA,\cP)$ shown in Figure \ref{fig:figures/annulusquasitriangulation}.
\FIGc{figures/annulusquasitriangulation}{Triangulation of $(\BA,\cP)$. $c,d$ are boundary $\cP$-arcs.}{2.2cm}

The Muller algebra $\sX_\xi(\D)$ is a quantum torus with generators $a,b,c,d$ where any two of them commute, except for $a$ and $b$ for which  $ab = \xi^{-2} ba$. 
We calculate $\al$ as an element of $\sX_\xi(\D)$ as follows.

First, calculate $a\al$ by using the skein relation, see Figure \ref{fig:figures/case3eqn2}.
\FIGc{figures/case3eqn2}{Computation of $a\al = \xi b^* + \xi^{-1}b$.}{2cm}

Here $b^*$ is new edge obtained from the flip of $\D$ at $b$ as defined in Figure \ref{fig:figures/mutation}. From Equation (\ref{eq.aa1}) we have that $b^* = [b^{-1}a^2] + [b^{-1}cd]$. Thus,
\begin{align}
\al & = a^{-1} (a\al)= a^{-1} (\xi b^* + \xi ^{-1}b) = a^{-1} (\xi  ([b^{-1}a^2] + [b^{-1}cd] )+ \xi ^{-1}b) \nonumber\\
\lbl{eq.v1}& =   [a^{-1} b^{-1} c d]+[a b^{-1}] +[ a^{-1} b]= X + Y + Y^{-1}.
\end{align}
where $X= [a^{-1} b^{-1} c d]$ and $Y= [a b^{-1}]$. From the commutation relations in $\sX_\xi(\D)$, we get $YX= \xi^4 XY$. 
Since each   of $\{a,b,c,d\}$ is a $\cP$-arc, from Proposition \ref{r.arc5}, we have
\be 
\lbl{eq.v2}
\tF_\xi(\al)= [a^{-N} b^{-N} c^N d^N]+[a^N b^{-N}] +[ a^{-N} b^N]= X^N + Y^N + Y^{-N}.
\ee
Since $\ord(\xi^4)=N$, 
Corollary \ref{c.32} shows that
$$
T_N(\al) = T_N(X + Y + Y^{-1})= X^N + Y^N + Y^{-N},
$$
which is equal to $\tF_\xi(\al)$ by \eqref{eq.v2}. This completes the proof.
\end{proof}

\begin{lemma} \lbl{r.knot1a} Proposition \ref{r.knot} holds if $\al$ is not 0 in $H_1(\Sigma, \BZ)$. 
\end{lemma}

\begin{proof} 
{\em Claim.} If $\al$ is not 0 in $H_1(\Sigma, \BZ)$, then there exists a properly embedded arc $a \subset \Sigma$ such that $|a \cap \al|=1$. \\
{\em Proof of Claim.} Cutting $\Sigma$ along $\al$ we get a (possibly non-connected) surface $\Sigma'$ whose boundary contains two components $\beta_1, \beta_2$ coming from $\al$. That is, we get $\Sigma'$ from $\Sigma$ by gluing $\beta_1$ to $\beta_2$ via the quotient map $\pr: \Sigma' \to \Sigma$, where $\pr(\beta_1)=\pr(\beta_2)=\al$. Choose a point $p\in \al$ and let $p_i\in \beta_i$ such that $\pr(p_i)=p$ for $i=1,2$.

Suppose first that $\Sigma'$ is connected. For $i=1,2$ choose a properly embedded arc $a_i$ connecting  $p_i\in \beta_i$ and a point in a boundary component of $\Sigma'$ which is not $\beta_1$ nor $\beta_2$. We can further assume that $a_1 \cap a_2=\emptyset$ since if they intersect once then replacing the crossing with either a positive or negative smoothing from the Kauffman skein relation (only one will work) will yield arcs that do not intersect and end at the same points as $a_1,a_2$, and the general case follows from an induction argument. Then $a=\pr(a_1 \cup a_2)$ is an arc such that $|a \cap \al|=1$.

Now suppose $\Sigma'$ has 2 connected components $\Sigma_1$ and $\Sigma_2$, with $\beta_i \subset \Sigma_i$. Since $\al$ is not homologically trivial, each $\Sigma_i$ has a boundary component other than $\beta_i$. For $i=1,2$ choose a properly embedded arc $a_i$ connecting  $p_i\in \beta_i$ and a point in a boundary component of $\Sigma'$ which is not $\beta_i$. Then $a=\pr(a_1 \cup a_2)$ is an arc such that $|a \cap \al|=1$. This completes the proof of the claim.

Let $\cQ = \partial a$ and 
$\cP'=\cP \cup \cQ$. Let $\Xi \subset \Sigma$ be the closure of a tubular neighborhood of $\al \cup a$. Then $\Xi$ is an annulus, and $\cQ$ consists of 2 points, one in each connected component of $\partial \Xi$. Let $\tF_{\xi,(\Xi,\cQ)}$, $\tF_{\xi,\SP}$, and $\tF_{\xi,(\Sigma,\cP')}$ be the map $\tF_\xi$ applicable respectively to the totally marked surfaces $(\Xi,\cQ)$, $\SP$, and $(\Sigma,\cP')$. By the functoriality of the inclusion $(\Xi,\cQ)\subset (\Sigma, \cP')$ (see Proposition \ref{prop.three}) we get the first of the following identities 
$$ \tF_{\xi,(\Sigma,\cP')} (\al) = \tF_{\xi,(\Xi,\cQ)}(\al) = T_N(\al)\ \text{in } \cS_\xi(\Sigma,\cP'),$$
while the second follows from Lemma \ref{r.markedannulus}. The functoriality of the inclusion $ \cP \subset \cP'$  gives
$$ \tF_{\xi,(\Sigma,\cP)} (\al) =  \tF_{\xi,(\Sigma,\cP')} (\al) \ \text{in } \cS_\xi(\Sigma,\cP').$$
It follows that $$ \tF_{\xi,(\Sigma,\cP)} (\al)=  T_N(\al) \ \text{in } \cS_\xi(\Sigma,\cP').$$
Since the natural map $\cS\SP \to \cS(\Sigma, \cP')$ is an embedding (Proposition \ref{r.func}), we also have $ \tF_{\xi,(\Sigma,\cP)} (\al)=  T_N(\al) \ \text{in } \cS_\xi(\Sigma,\cP)$, completing the proof.
\end{proof}

Now we proceed to the proof of Proposition \ref{r.knot}.
\begin{proof}[Proof of Proposition \ref{r.knot}]
If $\al\neq 0$ in $H_1(\Sigma,\BZ)$, then the statement follows from Lemma~\ref{r.knot1a}.  Assume  $\al=0$  in $H_1(\Sigma,\BZ)$. The idea we employ is to remove a disk in $\Sigma$ so that $\al$ becomes homologically non-trivial in the new surface, then use the surgery theory developed in Chapter \ref{c.surgery}.

 Since $\al=0$ in $H_1(\Sigma,\BZ)$, there is a surface $\Xi \subset \Sigma$ such that $\al=\partial \Xi$. Let $\Omega \subset \Xi$ be a closed disk in the interior of $\Xi$ and  $\beta=\partial \Omega$. Let $\Sigma'$ be obtained from $\Sigma$ by removing the interior of $\Omega$. Fix a point $p\in \beta$ and let $\cP'=\cP \cup \{p\}$. Since $\SP$ is triangulable, $\SPpp$ is also triangulable and $(\Sigma', \cP)$ is quasitriangulable.

Choose an arbitrary quasitriangulation $\D'$ of $(\Sigma', \cP)$. Via Proposition \ref{t.holetrick}, by plugging the unmarked boundary component $\beta$ we get a triangulation $\D$ of $\SP$ and a quotient map $\Psi_\zeta: \fXsur_\zeta(\D') \to \fXsur_\zeta(\D)$ for each $\zeta \in \Cx$, which we will just call $\Psi$ unless there is confusion. 
Since $\D$ is a triangulation, we have $\fXsur_\zeta(\D) =\sX_\zeta(\D)$.

For each $\zeta \in \Cx$ we have the inclusions  
$$\fXsur_\zeta(\D')\subset \sX_\zeta(\D') \subset \widetilde{\cS}_\zeta(\Sigma', \cP'),$$ where the second one comes from 
$\sX_\zeta(\D') \subset \widetilde{\cS}_\zeta(\Sigma', \cP) \subset \widetilde{\cS}_\zeta(\Sigma', \cP')$.

\noindent {\em Claim 1.} The map $\tF_\xi: \tcS_\ve(\Sigma',\cP') \to \tcS_\xi(\Sigma',\cP')$ restricts to a map from $\fXsur_\ve(\D')$ to $\fXsur_\xi(\D')$. That is, $\tF_\xi(\fXsur_\ve(\D')) \subset \fXsur_\xi(\D')$. In other words, there exists a map $F^\beta_\xi$ corresponding to the dashed arrow in the following commutative diagram.
\[
\begin{tikzcd}
\tcS_\ve(\Sigma',\cP') 
\arrow[r,"\tF_\xi"] & \tcS_\xi(\Sigma',\cP')  
\\
\fXsur_\ve(\D')\arrow[u, hook] \arrow[r,dashed,"F^\beta_\xi"] & \fXsur_\xi(\D') \arrow[u, hook]
\end{tikzcd}
\]
{\em Proof of Claim 1.} Let $a\in \D'$ be the only monogon edge (which must correspond to $\beta$). By definition, the set consisting of \\
(i) elements in $\D'\setminus \{a\}$ and their inverses,  $a$ and $a^*$, and \\
(ii) $\beta$\\  generates the $\BC$-algebra $\fXsur_\ve(\D')$. Let us look at each of these generators. If $x$ is an element of type (i) above, then by Proposition \ref{r.arc5}, we have $\tF_\xi(x)= x^N$ which is in $\fXsur_\xi(\D')$. Consider the remaining case $x=\beta$. Since the class of $\beta$ in $H_1(\Sigma',\BZ)$ is nontrivial, by Lemma \ref{r.knot1a}, we have
\be \lbl{eq.8aa}
\tF_\xi(\beta)= T_N(\beta)
\ee
 which is also in $\fXsur_\xi(\D')$. Claim 1 is proved.

\noindent{\em Claim 2.} The following diagram is commutative.
\be \lbl{eq.cd5}
\begin{tikzcd}
\fXsur_\ve(\D') \arrow[d,"\Psi_\ve"] \arrow[r,"F^\beta_\xi"] & \fXsur_\xi(\D') \arrow[d,"\Psi_\xi"] \\
\sX_\ve(\D) \arrow[r,"F_N"] & \sX_\xi(\D)
\end{tikzcd}
\ee
{\em Proof of Claim 2.} We have to show that
\be \lbl{eq.8a}
(F_N \circ \Psi)(x) = (\Psi \circ F^\beta_\xi)(x) \quad \text{for all} \ x\in \fXsur_\ve(\D').
\ee
It is enough to check the commutativity on the set of generators of $\fXsur_\ve(\D')$ described in (i) and (ii) above. If \eqref{eq.8a} holds for $x$ which is invertible, then it holds for $x^{-1}$. Thus it is enough to check \eqref{eq.8a} for $x\in \D' \cup \{ a^*, \beta\}$. 
Assume the notations $a,b,c$ of the edges near $\beta$ are as in Figure \ref{fig:figures/surgery}. 

First assume $x \not \in \{ a,a^*,b,\beta\}$. By \eqref{eq.2a}, we have $\Psi(x)=x$. Hence the left hand side of \eqref{eq.8a} is
$$F_N(\Psi(x))= F_N(x)= x^N.$$
On the other hand, the right hand side of \eqref{eq.8a} is
$$\Psi(F^\beta_\xi(x))=\Psi(\tF_\xi(x)) = \Psi(x^N)= x^N, $$
which proves \eqref{eq.8a} for $x \not \in \{ a,a^*,b,\beta\}$.

Assume $x=a$ or $x=a^*$. By \eqref{eq.2b}, we have $\Psi(x)=0$. Hence the left hand side of \eqref{eq.8a} is 0. On the other hand, the right hand side is
$$\Psi(F^\beta_\xi(x))= \Psi(\tF_\xi(x)) = \Psi(x^N)=0,$$
which proves \eqref{eq.8a} in this case.

Now consider the remaining case $x=\beta$. By \eqref{eq.2b}, we have $\Psi(\beta)= -\ve^2 -\ve^{-2}$. Hence the left hand side of \eqref{eq.8a} is
\be \nonumber
  F_N(\Psi(\beta))= F_N(-\ve^2 -\ve^{-2}) = -\ve^2 -\ve^{-2} = T_N(-\xi^2 -\xi^{-2}),
  \ee
  where the last identity is \eqref{eq.vexi}.
On the other hand, using \eqref{eq.8aa} and the fact that $\Psi$ is a $\BC$-algebra homomorphism, we have
\be 
\nonumber
  \Psi(F^\beta_\xi(x))=\Psi(\tF_\xi(x)) = \Psi(T_N(\beta))= T_N(\Psi(\beta))= T_N(-\xi^2 -\xi^{-2}).
  \ee
 Thus we always have \eqref{eq.8a}. This completes the proof of Claim 2.
 
 Let us continue with the proof of the proposition. Since the class of $\al$ is not 0 in $H_1(\Sigma',\BZ)$, by Lemma \ref{r.knot1a}, we have $F^\beta_\xi(\al)=\tF_\xi(\al)= T_N(\al)$. The commutativity of Diagram \eqref{eq.cd5} and the fact that $\Psi$ is a $\BC$-algebra homomorphism implies that
 \begin{align} F_N(\Psi_\ve(\al ))&= \Psi_\xi(F^\beta_\xi(\al ))= \Psi_\xi(T_N(\al )) \nonumber\\
 &= T_N(\Psi_\xi(\al )).
 \lbl{eq.9p} 
 \end{align}
 Note that $\al $ defines an element in $\cS_\nu(\Sigma', \cP)$ for $\nu=\ve,\xi$. Following the commutativity of Diagram \eqref{d.holesurgery} in Proposition \ref{t.holetrick}, we have that
 \begin{align}
&\Psi_\ve(\al)=\al \in \cS_\ve\SP \subset \sX_\ve(\D), \lbl{eq.alve}\\
&\Psi_\xi(\al)=\al \in \cS_\xi\SP \subset \sX_\xi(\D). \lbl{eq.alxi}
\end{align}

Then we may compute
\begin{align*}
\tF_\xi(\al) & = F_N(\al), \quad \text{by Proposition~\ref{prop.two}} \\
& = F_N(\Psi_\ve(\al)), \quad \text{by \eqref{eq.alve}} \\
& = T_N(\Psi_\xi(\al)), \quad \text{by \eqref{eq.9p}} \\
& = T_N(\al), \quad \text{by \eqref{eq.alxi}},
\end{align*}
completing the proof of Proposition \ref{r.knot}.
\end{proof}

\subsection{Proof of theorem \ref{thm.surface}}\lbl{sec.trisurfacethm}

\begin{proof} [Proof of Theorem \ref{thm.surface}]

The $\BC$-algebra $\Se\SP$ is generated by $\cP$-arcs and $\cP$-knots. If $a$ is a $\cP$-arc, then by Proposition \ref{r.arc5}, $\tF_\xi(a)= a^N\in \Sx\SP$. If $\al$ is a $\cP$-knot, then by Proposition \ref{r.knot}, $\tF_\xi(\al)= T_N(\al)\in \Sx\SP$. It follows that $\tF_\xi(\Se\SP)\subset \Sx\SP$. Hence  $\tF_\xi$ restricts to a $\BC$-algebra homomorphism $F_\xi: \Se\SP \to \Sx\SP$. Since on $\XD$, $\tF_\xi$ and $F_N$ are the same, $F_\xi$ is the restriction of $F_N$ on $\Se\SP$.
From Proposition \ref{prop.two}, $F_\xi$ does not depend on the triangulation~$\D$. This proves part (a). Part (b) was established in Propositions \ref{r.arc5} and \ref{r.knot}.
\end{proof}

\subsection{Proof of theorem \ref{t.ChebyshevFrobenius}}\lbl{sec.mainthm}

\begin{proof}[Proof of Theorem \ref{t.ChebyshevFrobenius}]   Recall that $\hPhi_N: \LMN \to \LMN$ is the $\BC$-linear map defined so that that if 
$T$ is an $\cN$-tangle with arc components $a_1, \dots, a_k$ and  knot components $\al_1,\dots,\al_l$, then
\be\lbl{eq.def00}
\hPhi_N(T) = \sum_{0\le j_1, \dots, j_l\le N} c_{j_1} \dots c_{j_l}  a_1^{(N)}
 \cup \dots \cup  a_k^{(N)} \cup \,\al_1^{(j_1)} \cup \cdots \cup \al_l^{(j_l)} \ 
 \ee 
where $T_N(z) = \sum_{i=0}^N c_i z^i$ is the $N$th Chebyshev polynomial of type 1, see~\eqref{e.chebyshev}. 
To show that $\hPhi_N: \LMN \to \LMN$ descends to a map $\Se\MN \to \Sx\MN$ we have to show that $\hPhi_N(\Rel_\ve) \subset \Rel_\xi$. Let $\hPhi_\xi: \LMN \to \Sx\MN$ be the composition
$$ \hPhi_\xi: \LMN \overset {\hPhi_N} \longrightarrow \LMN \overset{{[\,\cdot\,]_\xi}}{\longrightarrow} \Sx\MN.$$
Then we have to show that $\hPhi_\xi(\Rel_\ve)=0$.
There are 3 types of elements which span $\Rel_\ve$: trivial arc relation elements, trivial knot relation elements, and skein relation elements, and we consider them separately.

(i) Suppose $x$ is a trivial arc relation element (see Figure \ref{fig:figures/trivialarc}).  The $N$ copies $a^{(N)}$ in $\hPhi_\xi(x)$ have $2N$ endpoints, and by 
reordering the height of the endpoints, from  $a^{(N)}$ we can obtain a trivial arc. Hence, the reordering relation (see Figure \ref{fig:figures/boundary}) and the trivial arc relation show that $\hPhi_{\xi}(x) =0$.

(ii) Suppose $x=\ve^2 + \ve^{-2} +\al $ is a trivial loop relation element, where $\al$ is a trivial loop. Each parallel of $\al$ is also a trivial loop, which is equal to $-\xi^2 -\xi^{-2}$ in $\Sx\MN$. Hence $\hPhi_\xi(\al) = T_N(-\xi^2 -\xi^{-2})$, and
$$ \hPhi_{\xi}(x)= \ve^2 + \ve^{-2} + T_N((\al))=  \ve^2 + \ve^{-2} + T_N(-\xi^2 -\xi^{-2})=0,$$
where the last identity is \eqref{eq.vexi}. 

(iii) Suppose $x = T - \ve T_+ - \ve^{-1} T_-$ is a skein relation element. Here 
 $T, T_+, T_-$ are  $\cN$-tangles which are identical outside a ball $B$ in which they look like in Figure \ref{fig:figures/skein1a}.

 \FIGc{figures/skein1a}{From left to right: the tangles $T, T_+, T_-$}{1.3cm}

 \underline{Case I: the two strands of $T\cap B$ belong to two distinct components of $T$}.
 Let $T_1$ be the component of $T$ containing the overpass strand of $T \cap B$ and $T_2= T \setminus T_1$.
  Let $M'$ be the closure of a small neighborhood of $B \cup T=B\cup T_+=B\cup T_-$.   
 Write
 $\cN':= \cN \cap \partial (M')$. The functoriality 
 of the inclusion $(M', \cN') \to \MN$ implies that it is enough to show
  $\hPhi_\xi(x)=0$ for $(M', \cN')$. Thus now we replace $\MN$ by $(M', \cN')$.

 Note that $M'$ is homeomorphic to $\Sigma \times (-1,1)$ where $\Sigma$ is an oriented surface which is the union of the shaded disk of Figure \ref{fig:figures/skein1a} and the ribbons obtained by thickening the tangle $T_+$. As usual, identify $\Sigma $ with $\Sigma \times \{0\}$. Then, all four $\cN'$-tangles $T_1, T_2, T_+, T_-$ are in $\Sigma$ and have vertical framing. Note that $\Sigma$ might be disconnected, but each of its connected components has non-empty boundary. Let $\cP= \cN \cap \Sigma$. Then $\cS_\nu(M', \cN')= \cS_\nu\SP$ for $\nu=\xi,\ve$. Enlarge $\cP$ to a larger set of marked points $\cQ$ such that $(\Sigma, \cQ)$ is triangulable. Since the induced map $\iota_*: \Sx\SP \to \Sx\SQ$ is injective (by Proposition \ref{r.func}), it is enough to show that $\hPhi_\xi(x)=0$ in $\Sx\SQ= \LSQ/\Rel_\xi$. Here $\LSQ:= \cT(\Sigma \times (-1,1), \cQ \times (-1,1))$.
 
  The vector space $\LSQ$ is a $\BC$-algebra, where the product $\al \beta$ of two ($\cQ\times (-1,1)$)-tangles $\al$ and $\beta $ is the result of placing $\al$ on top of $\beta$. The map $\hPhi_\xi :\LSQ \to \Sx\SQ$ is an algebra homomorphism. Recall that for an element $y\in \LSQ$ we denote by $[y]_\nu$ its image under the projection $\LSQ \to \cS_\nu(\Sigma,\mathcal{Q})= \LSQ/\Rel_\nu$ for $\nu=\xi,\ve$.
 
 As $\SQ$ is triangulable, by Theorem \ref{thm.surface} we have the map $F_\xi: \Se\SQ\to \Sx\SQ$. %

Suppose  $y$ is a component of one of $T_1,T_2,T_+,T_-$, then $y$ is either a $\cQ$-knot (in $\Sigma$) or a $\cQ$-arc (in $\Sigma$) whose end points are distinct, with vertical framing in both cases. It follows that $[y^{(k)}]_\xi= [y^k]_\xi$. If $y$ is a knot component then Proposition \ref{r.knot} shows that $F_\xi([y]_\ve)= T_N([y]_\xi)= \hPhi_\xi(y) $.
Each of $T_1,T_2,T_+,T_-$ is the product (in $\LSQ$) of its components as the components are disjoint in $\Sigma$. Hence from the definition of $\hPhi_\xi$, we have
\be
\notag 
\hPhi_{\xi} (T_i)= F_\xi ([T_i]_\ve) \quad \text{for all } \ T_i \in \{ T_1, T_2, T_+, T_-\}.
\ee
As $T= T_1 T_2$ in $\LSQ$, we  have
 \be 
 \notag
 \hPhi_{\xi}(T)=  \hPhi_{\xi}(T_1 T_2) = \hPhi_{\xi}(T_1) \hPhi_{\xi}(T_2) = F_\xi ([T_1]_\ve) F_\xi([T_2]_\ve)= F_\xi([T]_\ve).
 \ee
As $x= T- \ve T_+ - \ve^{-1} T_-$, we also have $\hPhi_\xi(x) = F_\xi([x]_\ve).$ But $[x]_\ve =0$ because $x$ is a skein relation element.
This completes the proof that $\hPhi(x)=0$ in Case I.

\underline{Case II: Both strands of $T \cap B$ belong to the same component of $T$}. We show that this case reduces to the previous case.

Both strands of $T \cap B$ belong to the same component of $T$ means that some pair of non-opposite points of $T \cap \partial B$ are connected by a path in $T\setminus B$ (see Figure \ref{fig:figures/nonopposite}). Assume that the two right hand points of $T \cap \partial B$ are connected by a path in $T\setminus B$. All other cases are similar.
\FIGc{figures/nonopposite}{Four possibilities for non-opposite points being connected by a path in $T\setminus B$.}{2cm}
Then the two strands of $T_+$ in $B$ belong to two different components of $T_+$.
We isotope $T_+$ in $B$ so that its diagram forms a bigon, and calculate $\hPhi_{\xi}(T_+)$ as follows.
 
 \begin{align}  \hPhi_{\xi} \Lplu&=  \hPhi_{\xi} \Lpluss   \qquad \text{by isotopy} \nonumber\\
  &= \ve\,  \hPhi_{\xi} \uplus + \ve^{-1} \hPhi_{\xi} \uminus  \notag\\
 &= \ve (-\ve^{-3})  \hPhi_{\xi}\Lminus + \ve^{-1} \hPhi_{\xi} \LLL   \notag\\
 & =  -\ve^{-2}\hPhi_{\xi} (T_-) + \ve^{-1} \hPhi_{\xi} (T), \lbl{eq.zz1}
 \end{align} 
 where the second equality follows from the skein relation which can be used since the two strands of $T_+$ in the applicable ball belong to different components of $T_+$  (by case I), and the third equality follows from the well-known identity correcting a kink in the skein module:
$$ \negkinka = \ve\negkinkb + \ve^{-1}\negkinkc = (\ve + \ve^{-1}(-\ve^2-\ve^{-2}))\negkinkd = -\ve^{-3}\negkinkd.
$$
The  identity \eqref{eq.zz1} is equivalent to $\hPhi_\xi(x)=0$. 
This completes the proof of the theorem.
\end{proof}

\section{Consequences for marked and finite type surfaces}

Suppose $\SP$ is a marked surface, with no restriction at all. Apply Theorem \ref{t.ChebyshevFrobenius} to $\MN=(\Sigma \times (-1,1), \cP \times (-1,1))$. Note that in this case $\Phi_\xi$ is automatically an algebra homomorphism. Besides, since the set of $\cP$-arcs and $\cP$-knots generate $\Se\SP$ as an algebra,  we get the following corollary.

\begin{proposition}\lbl{prop.five}
Suppose $\SP$ is a marked surface, $\xi$ is a root of unity, $N=\ord(\xi^4)$, and $\ve = \xi^{N^2}$.  Then there exists a unique $\BC$-algebra homomorphism $\Phi_\xi: \cS_\ve\SP \to \cS_\xi\SP$ such that for $\cP$-arcs $a$ and $\cP$-knots $\al$,
\[
\Phi_\xi(a) = a^N, \ \ \ \ \Phi_\xi(\al) = T_N(\al).
\]
\end{proposition}
\begin{remark}
It follows from uniqueness that in the case where $\SP$ is a triangulable surface, $\Phi_\xi$ is the same as $F_\xi$ obtained in Theorem \ref{thm.surface}.
\end{remark}
\begin{corollary}
Proposition \ref{prop.five} applies equally well to the case where the surface under consideration is a finite type surface $\fS$, since $\cS(\fS)$ is defined in terms of the skein module of the marked 3-manifold $(\fS \times (-1,1),\emptyset)$.
\end{corollary}

%% file: center.tex
\chapter{Center of \texorpdfstring{$\cS\SP$}{S(P)} and (skew-)transparency}

In this chapter we investigate the center of the skein algebra of a surface, as well as the more general notion of (skew)-transparent elements in skein modules of marked 3-manifolds when the quantum parameter $q$ is a root of unity, which are roughly elements such that passing a strand through that element results in either no change or sign flip to its value in the skein module.

In Section \ref{sec.trans} we first review known results on the center of the skein algebra of an unmarked surface when the quantum parameter $q$ is a root of unity. We then show that the image of the Chebyshev-Frobenius homomorphism $\Phi_\xi: \cS_\ve\SP \to \cS_\xi\SP$ discussed in Chapter \ref{c.cf} is either ``transparent'' or ``skew-transparent'' depending on whether $\xi^{2N}=\pm 1$. Our result generalizes known theorems regarding the center of the skein algebra \cite{BW2016-2,Le2015,FKL2017} of an unmarked surface and (skew-)transparent elements in the skein module of an unmarked 3-manifold \cite{Le2015}.

In Section \ref{sec.reducedskein} we find the center of the skein algebra of a marked surface when the quantum parameter $q$ is not a root of unity via a Muller algebra argument.

\section{Image of \texorpdfstring{$\Phi_\xi$}{Phixi} and (skew-)transparency}\lbl{sec.trans}

We fix the ground ring to be $R=\BC$ throughout this section.

\subsection{Center of the skein algebra of an unmarked surface}

Fix a compact oriented surface $\Sigma$ with (possibly empty) boundary. For $\xi \in \Cx$ we write $\cS_\xi := \cS_\xi(\Sigma,\emptyset)$. In this context, and when $\xi$ is a root of unity, the Chebyshev-Frobenius homomorphism $\Phi_\xi: \cS_\ve \to \cS_\xi$ specializes to the Chebyshev homomorphism for the skein algebra of $\Sigma$ given in \cite{BW2016-2}. The image of $\Phi_\xi$ is closely related to the center of $\cS_\xi$.

\begin{theorem}[\cite{FKL2017}]\lbl{r.bwcenter}
Let $\xi$ be a root of unity, $N= \ord(\xi^4)$, $\ve = \xi^{N^2}$, and $\cH$ the set of boundary components of $\Sigma$. Note that $\xi^{2N}$ is either 1 or $-1$.
Then the center $Z(\cS_\xi)$ of $\cS_\xi$ is given by

\be\lbl{e.surfacecenter}
Z(\cS_\xi) = \left\{ \begin{array}{ll}
\Phi_\xi(\cS_\ve)[\cH] & \text{if }\xi^{2N}=1, \\
\Phi_\xi(\cS_\ve^{\text{ev}})[\cH] & \text{if }\xi^{2N}=-1. \end{array} \right.
\ee
\end{theorem}
Here $\cS_\xi^{\text{ev}}$ is the $\BC$-subspace of $\cS_\xi$ spanned by all 1-dimensional closed submanifolds $L$ of $\Sigma$ such that $\mu(L,\al) \equiv 0 \pmod{2}$ for all knots $\al\subset \Sigma$.
In \cite{BW2016-2}, the right hand side of \eqref{e.surfacecenter} was shown to be a subset of the left hand side using methods of quantum Teichm\"uller theory in the case where $\xi^{2N}=1$. This result was reproven in \cite{Le2015} using elementary skein methods. The generalization to $\xi^{2N}=-1$ and the converse inclusion was shown in \cite{FKL2017}.

\subsection{(Skew-)transparency} In the skein module of a 3-manifold, we don't have a product structure, and hence cannot define central elements. Instead we will use the notion of {\em (skew-)transparent elements}, first considered in \cite{Le2015}. Throughout this subsection we fix a marked 3-manifold $\MN$.

Suppose $T'$ and $T$ are disjoint $\cN$-tangles. 
Since $\Phi_\xi(T')$ can be presented by a $\BC$-linear combination of $\cN$-tangles in a small neighborhood of $T'$, one can define $\Phi_\xi(T')\cup T$  as an element of $\cS_\xi\MN$, see Subsection \ref{sec.func}.

Suppose $T_1,T_2$, and $T$ are $\cN$-tangles. We say that $T_1$ and $T_2$ are connected by a {\em  single $T$-pass $\cN$-isotopy} if there is a continuous family of $\cN$-tangles $T_t$, $t \in [1,2]$, connecting $T_1$ and $T_2$ such that $T_t$ is transversal to $T$ for all $t \in [1,2]$ and furthermore that $T_t \cap T=\emptyset$ for $t \in [1,2]$ except for a single $s \in (1,2)$ for which $|T_s \cap T|=1$.

\begin{theorem}\lbl{r.mutransparent}
Suppose $\MN$ is a marked 3-manifold, $\xi$ is a root of unity, $N = \ord(\xi^4)$. Note that $\xi^{2N}$ is either 1 or $-1$.

(a) If $\xi^{2N}=1$ then the image of the Chebyshev-Frobenius homomorphism is {\em transparent} in the sense that if $T_1, T_2$ are $\cN$-isotopic $\cN$-tangles disjoint from another $\cN$-tangle $T$, then in $\cS_\xi\MN$ we have
\be 
\Phi_\xi(T) \cup T_1 = \Phi_\xi(T) \cup T_2.
\ee 

(b) If $\xi^{2N}=-1$ then the image of the Chebyshev-Frobenius homomorphism is {\em skew-transparent} in the sense that if $\cN$-tangles $T_1, T_2$ are connected by a single $T$-pass $\cN$-isotopy, where $T$ is another $\cN$-tangle, then in $\cS_\xi\MN$ we have
\be 
\Phi_\xi(T) \cup T_1 = - \Phi_\xi(T) \cup T_2.
\ee 
\end{theorem}

\begin{proof}
 
(a) and (b) are proven in \cite{Le2015} for the case where $\cN = \emptyset$. That is, given an $\emptyset$-tangle $T$, it is shown that $\Phi_\xi(T)$ is (skew-)transparent in $\cS_\xi(M,\emptyset)$ where we necessarily have that all components of $T$ are knots. By functoriality, $\Phi_\xi(T)$ is (skew-)transparent in $\cS_\xi\MN$ as well when all components of $T$ are knots.

We show that $\Phi_\xi(a)$ is (skew-)transparent when $a$ is a $\cN$-arc. Let $T_1,T_2$ be $\cN$-isotopic $\cN$-tangles connected by a single $a$-pass $\cN$-isotopy $T_t$. Consider a neighborhood $\Xi$ consisting of the union of a small tubular neighborhood of $a$ and a small tubular neighborhood of $T_t$. We may assume that the strands of $\Phi_\xi(a)=a^{(N)}$ are contained in the tubular neighborhood of $a$, and furthermore that $T_t$ is a single $a_i$-pass $\cN$-isotopy for each component $a_i$ of $a^{(N)}$. Write $\cQ = \cN \cap \Xi$. Then both $a^{(N)} \cup T_1$ and $a^{(N)} \cup T_2$ are $(\Xi,\cQ)$-tangles and we apply the skein relation and trivial arc relation inductively in $\cS_\xi(\Xi,\cQ)$ to derive the equations in Figure \ref{fig:figures/center4}.

\FIGc{figures/center4}{Resolving crossings between $a^{(N)}$ and $T_1,T_2$ in $\cS_\xi(\Xi,\cQ)$.}{3.35cm}

By functoriality, the computation in $\cS_\xi(\Xi,\cQ)$ is true in $\cS_\xi\MN$ as well. We see from Figure \ref{fig:figures/center4} that $\Phi_\xi(a)$ is transparent if and only if $\xi^{N}=\xi^{-N}$, ie that $\xi^{2N}=1$. We also see that $\Phi_\xi(a)$ is skew-transparent if and only if $\xi^{N}=-\xi^{-N}$, ie that $\xi^{2N}=-1$.
\end{proof}

\section{Center for \texorpdfstring{$q$}{q} not a root of unity}\lbl{sec.reducedskein}

Throughout this section $R$ is a commutative domain with a distinguished invertible element $q^{1/2}$ and $\SP$ is a marked surface. In particular we note that $R$ is no longer required to be Noetherian in this section. We write $\cS$ for the skein algebra $\cS\SP$ defined over $R$. We will calculate the center of $\cS\SP$ for the case when $q$ is not a root of unity.

\subsection{Center of the skein algebra} Let $\cH$ denote the set of all unmarked components in $\pS$ and $\cHd$ the set of all marked components.  If $\beta\in \cH$ let $z_\beta=\beta$ as an element of $\cS$. If $\beta\in \cHd$ let 
$$ z_\beta = \left [ \prod a \right] \in \cS,$$
where the product is over all boundary $\cP$-arcs in $\beta$.

\begin{theorem}
\lbl{thm.center1}
Suppose $\SP$ is a  marked surface. Assume that $q$ is not a root of unity. Then the center $Z(\cS\SP)$ of $\cS\SP$ is the $R$-subalgebra generated by $\{z_\beta \mid  \beta \in\cH\cup \cHd\}$.
\end{theorem}
\begin{proof}
It is easy to verify Theorem \ref{thm.center1} for the few  cases of  non quasi-triangulable surfaces. We will from now on assume that $\SP$ is quasitriangulable, and fix a quasitriangulation $\D$ of $\SP$.

We write $Z(A)$ to denote the center of an algebra $A$. 
\begin{lemma}\lbl{r.weakcenters}
Let $A \subset B$ be $R$-algebras such that $A$ weakly generates $B$. Then $Z(A) \subset  Z(B)$.
\end{lemma}

\begin{proof}
 Let $x \in Z(A)$. Then $x$ commutes with elements of $A$, and hence with their inverses in $B$ if the inverses exist.  So $x \in Z(B)$.
\end{proof}

 Since $\sX_+(\D)$ weakly generates $\XD$, and $\sX_+(\D) \subset \cS \subset \XD$ by
 Theorem \ref{r.torus}, we have 
\begin{corollary}\lbl{l.centerinclusions}
One has   
$Z(\cS) \subset Z(\XD) $. Consequently $Z(\cS) =  Z(\XD) \cap  \cS$.
\end{corollary}

\begin{lemma}\lbl{l.zbetacentral} For all $\beta \in\cH\cup \cHd$, one has
$z_\beta \in Z(\cS) \subset Z(\sX(\D))$.
\end{lemma}

\begin{proof} It is clear that $z_\beta \in Z(\cS)$ if $\beta\in \cH$.
Let $\beta \in \cHd$.
Any  $\cP$-knot $\al$ can be isotoped away from the boundary, and therefore $z_\beta \al = \al z_\beta$.
Let $a \in \cS$ be a $\cP$-arc. If $a$ does not end at some $p \in \beta \cap \cP$ then $az_\beta=z_\beta a$ is immediate. Assume that $a$ has an end at $p \in \beta \cap \cP$. $\cP$-isotope $a$ so that its interior does not intersect $\pS$. Then in the support of $z_\beta$ there is one strand clockwise to $a$ at $p$ and one counterclockwise. Therefore $az_\beta = z_\beta a$ by the reordering relation in Figure \ref{fig:figures/boundary}.
Since $\cS$ is generated as an $R$-algebra by $\cP$-arcs and $\cP$-knots, we have $z_\beta \in Z(\cS)$.
\end{proof}

 For each $\beta \in \cHd$, we define $\bk_\beta \in \BZ^\D$ so that $z_\beta= X^{\bk_\beta}$. In other words,
\be\lbl{e.bkbeta}
\bk_\beta(a) = \left\{ \begin{array}{ll}
1 & \text{if }a \subset \beta, \\
0 & \text{otherwise.} \end{array} \right.
\ee

We write $P$ for the vertex matrix of $\D$ (see Subsection \ref{sec.vmatrix}).

\begin{lemma}\lbl{l.xcenter}
One has $Z(\sX(\D)) =  R[\cH][X^\bk \mid  \bk \in \ker P]$.

\end{lemma}

\begin{proof}

Recall that $\sX(\D)$ is a $\BZ^\D$-graded algebra given by
$$
\sX(\D) = \bigoplus_{\bk \in \BZ^\D} R[\cH] \cdot X^\bk.
$$
The center of a graded algebra is the direct sum of the centers of the homogeneous parts. Hence
$$
Z(\sX(\D)) = \bigoplus_{\bk \in \BZ^\D: X^\bk \text{ is central}} R[\cH] \cdot X^\bk.
$$
By the commutation relation \eqref{e.normalizedtorus} we have $X^\bk X^\bl = q^{\langle \bk, \bl \rangle_P} X^\bl X^\bk$. Thus $X^\bk$ is central if and only if $q^{\langle \bk, \bl \rangle_P} = 1$ for all $\bl \in \BZ^\D$. Since $q$ is not a root of unity, this is true if and only if $\langle \bk, \bl \rangle_P = 0$ for all $\bl\in \BZ^\D$. Equivalently, $\bk \in \ker P$. This proves the lemma.
\end{proof}

\begin{lemma}\lbl{r.kerPZm}
The kernel $\ker P$ is the free $\BZ$-module with basis $\{\bk_\beta \mid  \beta \in \cHd\}$.
\end{lemma}

\begin{proof}  Let $\text{Null}(P)$ be the nullity of $P$.

Lemmas \ref{l.zbetacentral} and \ref{l.xcenter} imply that $\bk_\beta \in \ker P$ for each $\beta \in \cHd$. Since the $\bk_\beta$'s, as functions from $\D$ to $\BZ$, have pairwise disjoint supports, the set $\{\bk_\beta \mid  \beta \in \cHd\}$ is $\BQ$-linear independent. In particular, 
\be 
\lbl{eq.hh}
\text{Null}(P) \ge |\cHd|.
\ee

{\em Claim 1.} Assume that $\beta\in \cH$ is an unmarked boundary component. Choose a point $p_\beta \in \beta$ and let $\cP'=\cP \cup \{p_\beta\}$. Then let $\D'$ be an extension of $\D$ to a $\cP'$-quasitriangulation as depicted in Figure \ref{fig:figures/holepoint} (this guarantees that $\D \subset \D'$), and $P'$ the associated vertex matrix. Then $\text{Null}(P') \geq \text{Null}(P) + 1$.

{\em Proof of Claim 1.} Consider $\BZ^\D \subset \BZ^{\D'}$ via extension by zero. Choose a $\BZ$-basis $B$ of $\ker P$. Then because the $\D \times \D$ submatrix of $P'$ equals $P$, one has $B\subset \ker P'$. Let $\bk_\beta \in \BZ^{\D'}$ be as given in \eqref{e.bkbeta} with $\D$ replaced by $\D'$. Then $\bk_\beta \in \ker P'$.
Since $\bk_\beta$ does not have support in $\D$, $\bk_\beta$ is $\BZ$-linearly independent of $B$. Therefore $\text{Null}(P')$ must be at least 1 greater than $\text{Null}(P)$. This completes the proof of Claim 1.

{\em Claim 2.} One has $\text{Null}(P) = |\cHd|$.

{\em Proof of Claim 2.} 
By \cite[Lemma 4.4(b)]{Le2017}, the claim is true if $\SP$ is totally marked, i.e. if $\cH=\emptyset$. 

Suppose $|\cH|=k$. By sequentially adding marked points to unmarked components in $\cH$ and extending the triangulation as in Claim 1, we get a totally marked surface $(\Sigma, \cP^{(k)})$ with a new vertex matrix $P^{(k)}$. From Claim 1 we have $\text{Null}(P^{(k)}) \ge \text{Null}(P) + k$. On the other hand since $(\Sigma, \cP^{(k)})$ is totally marked and having $|\cHd|+k$ boundary components, we have $\text{Null}(P^{(k)})= |\cHd|+k$. It follows that $|\cHd|   \ge \text{Null}(P)$. Together with \eqref{eq.hh} this shows $\text{Null}(P) = |\cHd|$, completing the proof of Claim 2.

Claim 2 and the fact that  $\{\bk_\beta \mid  \beta \in \cHd\}$ is a $\BQ$-linear independent subset of $\ker P$ shows that  $\{\bk_\beta \mid  \beta \in \cHd\}$ is a $\BQ$-basis of $\ker P$. Let us show that $\{\bk_\beta \mid  \beta \in \cHd\}$ is a $\BZ$-basis of $\ker P$.

Let $x \in \BZ^\D$ be in $\ker P$. Since $\{\bk_\beta \mid  \beta \in \cHd\}$ is $\BQ$-basis of $\ker P$, we have
$$
x = \sum_{\beta \in \cHd} c_\beta \bk_\beta, \quad c_\beta \in \BQ.
$$
Since $\bk_\beta$'s are functions from $\D$ to $\BZ$, have pairwise disjoint supports, and $x:\D \to \BZ$ has integer values, each $c_\beta$ must be an integer. Hence $x$ is a $\BZ$-linear combination of $\{\bk_\beta \mid  \beta \in \cHd\}$. This shows that $\{\bk_\beta \mid  \beta \in \cHd\}$ is a $\BZ$-basis of $\ker P$.
\end{proof}
Lemmas \ref{l.xcenter} and \ref{r.kerPZm} show that $Z(\XD)= R[\cH][X^{\bk_\beta} \mid \beta \in \cHd]$, which is a subset of $\cS$.  By Corollary~\ref{l.centerinclusions}, we also have $Z(\cS)= Z(\XD) \cap \cS= R[z_\beta | \beta \in \cH \cup \cHd]$, completing the proof.
\end{proof}

%% file: maximalorder.tex
\chapter{\texorpdfstring{$\cS(\fS)$}{S(S)} is a maximal order}\lbl{c.maxorder}

In this chapter we show that the Kauffman bracket skein algebra of a finite-type surface with at least one puncture is a maximal order.

\section{Setting}

Throughout this chapter we fix a finite type surface $\fS=\ofS \setminus \cV$ with $|\cV| \geq 1$ and negative Euler characteristic. We set our base ring $R=\BC$, and let $\xi \in \Cx$ be a root of unity. We write $\cS:=\cS_\xi(\fS)$.

\section{Proof of result}

Our proof consists of defining a good $\BN$-filtration on $\cS$ that yields an associated graded algebra that is a monomial subalgebra of a quantum torus (see Subsection \ref{sec.monoidsubalgebra}). Then following Proposition \ref{r.ALamaximal}, that associated graded algebra is a maximal order, and then so is $\cS$ following Proposition~\ref{r.gradedorderlift}.

\subsection{Skein algebra filtration}\lbl{sec.skeinfiltration}

Suppose $\Pi = \{a_1,\ldots,a_l\}$, where each $a_i$ is an ideal arc or loop on $\fS$.

For each $n \in \BN$, define $F_n := F_n^\Pi(\cS_{\xi}(\fS))$ to be the $\BC$-subspace of $\cS_{\xi}(\fS)$ spanned by $L \in \cS$ such that $\sum_{i=1}^l \mu(a_i,L) \leq n$. Here $\mu(a_i, L)$ is the geometric intersection number. Then $F_n$ defines an $\BN$-filtration on $\cS_{\xi}(\fS)$, see e.g. \cite{FKL2017,Le2015,Le2018}.

\begin{proposition}\lbl{r.skeinfilter}
The sequence $\{F_n\}_{n=0}^\infty$ is a $\BN$-filtration of $\cS_{\xi}(\fS)$ compatible with the algebra structure.
\end{proposition}

\begin{proof}
It is clear that $\cup F_n = \cS_{\xi}(\fS)$, and for $m<n$, $F_m \subset F_n$. So we show that $F_mF_n \subset F_{m+n}$. Suppose $K \in F_m$, $L \in F_n$. Using the skein relation, $KL = \sum c_j J_j$, where $c_j \in \BC$ and $J_j$ is a simple diagram obtained by smoothing all of the crossings in $KL$, and so $\mu(J_j,a_i) \leq \mu(K,a_i) + \mu(L,a_i)$ for all $i=1,\ldots,l$ and all $j$ by \cite[Lemma 4.7]{Mu2012}, and so $KL \in F_{m+n}$.
\end{proof}

\no{

\subsection{Pants coordinates for closed surfaces}

\red{Thang}
}

\subsection{Properties of the monoid of curves}
Set $r:=6g-6+3p = -3\chi(\fS)$, where $\chi(\fS)$ is the Euler characteristic of $\fS$. Let $\La$ be the image of $B_\fS$ in $\BN^r$ under the edge coordinate map $E_\D$ as given in ~\eqref{e.edgecoords}.

\begin{proposition}\lbl{r.skeinprimmonoid}
$\Lambda$ is a primitive submonoid of $\underline{\La}$ such that $\Cone(\La)$ is closed.
\end{proposition}
\begin{proof}

Let $a,b,c$ be distinct edges of an ideal triangle $T$ of $\Sigma$. A vector $\bk \in \BZ^r$ is in $\La$ if and only if it satisfies the following properties for all ideal triangles of $\D$ with edges labeled $a,b,c$:
\begin{enumerate}[(i)]
\item $k_a,k_b,k_c \geq 0$,
\item $k_a+k_b+k_c$ is even,
\item $k_a \leq k_b+k_c$.
\end{enumerate}

Note that $0 = E_\D(\emptyset)$, and $\Lambda$ is closed under addition. Thus $\Lambda$ is a submonoid of $\BZ^r$.

$\Cone(\La)$ is closed since it may be described with a finite number of inequalities.

Let $n \in \BN$, $\bk \in \uLa$ such that $n\bk \in \Lambda$. Let $a,b,c$ be edges of an ideal triangle of $\D$. It is clear that since $n\bk$ satisfies properties (i) and (iii) for an admissible edge coloring then so does $\bk$. Furthermore, since $\uLa$ is the group completion of $\La$, for all $\bl \in \uLa$ we must have that $l_a+l_b+l_c$ is even, and so $\bk$ satisfies (ii) as well. Thus $\bk \in \La$, so $\La$ is primitive in $\uLa$.
\end{proof}

\subsection{Associated graded algebra}\lbl{sec.skeingraded}
Let $p \geq 1$ and fix an ordered ideal triangulation $\D$ of $\fS$. Let $S: \Lambda \to B_\fS$ be the inverse of $E_\D$ restricted to $\Lambda$. By abuse of notation, consider $\D$ as an unordered set and consider the filtration $F_n:=F_n^\D(\cS_{\xi}(\fS))$ as laid out in Subsection \ref{sec.skeinfiltration}. Meaning, $F_n$ is the $\BC$-subspace of $\cS_{\xi}(\fS)$ spanned by $S(\bk)$, $\bk \in \Lambda$ where $|\bk|:=\sum_{i=1}^r \bk(i) \leq n$.

\begin{proposition}\lbl{r.skeinproduct}
For $\bk,\bk' \in \Lambda$, there exists $C(\bk,\bk') \in \BZ$ such that
\be\lbl{e.skeinqcommute}
S(\bk)S(\bk') = q^{C(\bk,\bk')}S(\bk+\bk') \ \ \ \ (\text{mod }F_{<|\bk|+|\bk'|}).
\ee
\end{proposition}
This was proved in \cite{AF2017,FKL2017} for a slightly different filtration but the simple proof there also applies in this case. An explicit formula for $C(\bk,\bk')$ is given in \cite{FKL2017}.

\subsection{Maximal order}

\begin{theorem}\lbl{r.mainresult}
Let $\fS$ be a finite-type surface, and let $\xi \in \BC^\times$ be a root of unity. Then $\cS_{\xi}(\fS)$ is (i) a domain, (ii) finitely generated over its center, and (iii) a maximal order.
\end{theorem}
\begin{proof} Part (i) is Lemma~\ref{r.finitedomain}. Part (ii) was shown in \cite{FKL2017}. By Proposition \ref{r.skeinprimmonoid} and ~\eqref{e.skeinqcommute}, the algebra $\cS$ satisfies all the assumption of Proposition \ref{r.gradedorderlift} and hence is a maximal order.
\end{proof}

\begin{remark}
It is also true that the skein algebra of surfaces with non-negative Euler characteristic is a maximal order. We leave these cases, which have a direct proof, as an exercise.
\end{remark}

\begin{remark}
It is also true that the reduced skein algebra as described in \cite{PS2018} and the Muller skein algebra as described in \cite{Mu2012} (including the version that allows unmarked boundary components as given in \cite{LP2018}) are maximal orders. The extension of our proof to these cases is simple and we leave it as an exercise for the reader.
\end{remark}

\section{Applications}

In this section we present two consequences of Theorem \ref{r.mainresult} and one conjecture. In Subsection \ref{s.normal} we show how Theorem \ref{r.mainresult} implies that $SL_2\BC$-character varieties are normal, which illustrate the power of utilizing quantizations to study classical topological objects. In Subsection \ref{s.refunicity} we show that Theorem \ref{r.mainresult} gives a refined version of the unicity theorem of \cite{FKL2017}. Lastly, in Subsection \ref{s.tqc} we give a conjecture detailing how Theorem \ref{r.mainresult} may open a pathway towards a scalable topological quantum compiling algorithm.

\subsection{Character varieties are normal}\lbl{s.normal}

Recall that $\Sp(\cS_q(\fS))$ is the set of equivalence classes of finite-dimensional irreducible representations of $\SxiS$ over $\BC$. Let $\ZxiS$ be the center of $\SxiS$. We have the central character map
\be 
 \chi: \Sp(\SxiS) \to \Sp(\ZxiS).
 \ee

 Suppose $\xi$ is a root of unity. Since $\SxiS$ is finitely generated as a $\BC$-algebra and finitely-generated as a module over its center, a form of the Artin-Tate lemma (see 13.9.10, \cite{MRS2001}) says that $\ZxiS$ is also finitely-generated as a $\BC$-algebra. Being a subring of the domain $\SxiS$, the algebra $\ZxiS$ is also a domain. It follows that $\Sp(\ZxiS)$ is an affine variety, denoted by $\YxiS$, and is called the {\em classical shadows variety} of $\SxiS$ in \cite{FKL2017}. As the center of a maximal order is integrally closed, we get the following consequence.
 
 \begin{theorem}\lbl{s.shadow}
 For every finite type surface $\fS$ and every root of unity $\xi$ the classical shadows variety $\YxiS$ is a normal affine variety.
 \end{theorem}
 The classical shadows variety is closely related to the $SL_2(\BC)$-character variety $X(\fS):=X(\fS,SL_2\BC)$ of $\fS$ (see Section \ref{s.charvariety}). Assume that $\ord(\xi)\neq 0 \pmod 4$. Then there is a finite regular map 

 \be
 \YxiS \onto X(\fS)
 \lbl{eq.finite}
 \ee
and hence they have the same dimension, see \cite{FKL2017}. When $\xi=-1$ the skein algebra $\cS(\fS)$ is equal to its center and also equal to the coordinate ring of the $SL_2\BC$-character variety $X(\fS)$. Hence we have the following.
 \begin{corollary} For every finite type surface $\fS$, the $SL_2$-character variety $X(\fS)$ is a normal affine variety.
 \end{corollary}

 \begin{remark} The fact that the character variety of a surface is normal was proved in \cite{Si94-1,Si94-2} by a difficult method. Here we recovered this result by using skein algebra theory.
 \end{remark}

\subsection{Refinement of unicity theorem}\lbl{s.refunicity}
Suppose $\xi$ is a root of unity. Let $N =\ord(\xi^4)$. Recall from Subsection \ref{sec.maxorderrep} that the degree $\deg(\SxiS)$ is the square of the dimension of the division algebra $\widetilde {\SxiS}$ over its center $\widetilde{\ZxiS}$, which is a field. The degree $\deg(\SxiS)$ was calculated in \cite{FKL2019}:
 \be 
 \deg(\SxiS) =  D(\fS, \xi):=\begin{cases} N^{3g-3+p}   \quad &\text{if } \ \ord(\xi) \neq 0 \pmod 4 \\
 2^g N^{3g-3+p}   &\text{if } \ \ord(\xi) =  0 \pmod 4.
 \end{cases}
 \ee
Since $\SxiS$ is a maximal order, by applying Theorem \ref{r.dCKP}, we get the following refinement of the unicity theorem of \cite{FKL2017}.
\begin{theorem}\lbl{r.refinedunicity}
Suppose $\fS$ is a finite type surface of negative Euler characteristic and $\xi$ is a root of unity.

(a) The variety of shadows $\YxiS$ parameterizes the $D(\fS, \xi)$-dimensional semisimple representations of $\SxiS$.

(b) The central character map $\chi: \Sp(\SxiS) \to \YxiS$ is surjective and each fiber consists of all those irreducible representations of $\Sp(\SxiS)$ which are irreducible components of the corresponding semisimple representation.  In particular, each irreducible representation of $\Sp(\SxiS)$ has dimension $\leq D(\fS, \xi)$. A point $a\in \YxiS$ is called generic if the dimension of the irreducible representation $\chi^{-1}(a)$ is $D(\fS, \xi)$.

(c) The set of generic points in $\YxiS$ is a non-empty Zariski open dense subset of $\YxiS$.

(d) Let $\cU \subset \YxiS$ be the open and dense subset consisting of all points $u$ such that $|\chi^{-1}(u)| = 1$. Let $x \in \YxiS \setminus \cU$ and write $\chi^{-1}(x) = \{r_1,\ldots,r_s\}$. Then since $\sum_{i=1}^s \dim(r_i) = D(\fS,\xi)^2$, in particular we have that $s \leq D(\fS,\xi)$.
\end{theorem}

The unicity theorem of \cite{FKL2017} only says that $\chi$ is surjective, the dimension of every irreducible representation is $\leq D(\fS, \xi)$, and the set of generic points is Zariski dense in $\YxiS$.

\subsection{Topological quantum compiling}\lbl{s.tqc}

The author conjectures that the fact that skein algebras of surfaces are maximal orders has a potential real-world application in compiling for topological quantum computers. Topological quantum computers were first introduced by Freedman and Kitaev \cite{Fr1998,Ki2003}.

In \cite{KY2015}, Kliuchnikov and Yard give an asymptotically optimal algorithm for exact synthesis of quantum gates for topological quantum computers utilizing Fibonacci anyons, $SU(2)_k$ anyons, and others. In particular, they utilize the quantum $SU(2)$ Witten-Reshetikhin-Turaev Chern-Simons topological quantum field theory at level $k$. Their algorithm utilizes in a crucial manner the presence of maximal orders in the division algebra generated by the set of quantum gates. Translated into the quantum topology lexicon, their algorithm shows how to efficiently factorize elements of the quantum $SU(2)$ WRT Chern-Simons representations of the mapping class group of the thrice-punctured disk as a product of elements of a particular generating set. Here, the punctures correspond to anyons - while for a skein algebra, anyons correspond to Jones-Wenzl idempotents. As was shown in \cite{BW2016}, a corresponding representation may also be obtained for the Kauffman bracket skein algebra of the surface. Thus, we conjecture that under certain circumstances an analogous exact synthesis algorithm exists for skein algebras, though we note that the anyons associated to Jones-Wenzl idempotents for Kauffman bracket skein algebras are most closely associated with Temperley-Lieb-Jones anyons as described in \cite{Wa2010}. However, all of the aforementioned anyons share a close relationship.

Let us be more precise. Following the prescription laid out at the beginning of Section 3 of \cite{KY2015}, we believe the answer to the following question is within close reach. Let $\fS$ be the once-punctured torus and $\cS:=\cS(\fS)$ the Kauffman bracket skein algebra of $\fS$. Write $Z$ for the center of $\cS$ and $\tZ$ the field of fractions of $Z$. Then it is known that the division algebra $\cS \otimes_Z \tZ$ is isomorphic to the division algebra of a quantum torus on 3 generators called the quantum Teichm\"uller space of the surface, and that for $\xi=\pm i$ this quantum torus is a quaternion algebra over $\tZ$. Let $S$ be a set of prime ideals of $Z$, and let $\cS_S$ be the set of all elements of $\cS$ such that their reduced norm factors into elements of $S$. Then utilizing the results in \cite{KY2015} one may be able to determine a finite set $\mathcal{G} \subset \cS$ such that any element of $\cS_S$ can be written as $x_1 \cdots x_n$ for $x_k \in \mathcal{G}$ and find the most efficient such factorizations in terms of length. The only obvious difference between these contexts is that $\tZ$ is not a number field, but this can be rectified by working in a representation.

Potentially important elements of a way to generalize the above to other surfaces and roots of unity include the triangular decomposition of skein algebras \cite{Le2016}, the structure of the center of $\cS$ as investigated in \cite{BW2016-2,FKL2017}, as well as the more general notion of transparent and skew-transparent elements investigated in \cite{Le2015,LP2018}, the relationship between skein algebras at different roots of unity as revealed by the Chebyshev(-Frobenius) homomorphism \cite{BW2016-2,LP2018}, and the unicity theorem first proved in \cite{FKL2017} and refined in Theorem \ref{r.refinedunicity}. The real-world importance of the use of the entire mapping class group of arbitrary punctured surfaces with boundary as opposed to just punctured disks is pinpointed in \cite{BF2016}, which shows that topological quantum computation on a punctured surface with boundary may be equivalently performed on two parallel flat surfaces joined by bridges.

%% file: appendix.tex
\begin{appendices}

\addtocontents{toc}{\protect\renewcommand{\protect\cftchappresnum}{\appendixname\space}}
\addtocontents{toc}{\protect\renewcommand{\protect\cftchapnumwidth}{6em}}

\chapter{Geometric proof that \texorpdfstring{$\cS$}{S} is a domain}\lbl{a.domain}
\no{\red{From 9/15/16}}

In this appendix we give a purely geometric proof that $\cS\SP$ contains no zero divisors for a marked surface $\SP$. Actually, we only prove the statement that essential $\cP$-arcs are not zero divisors, which is sufficient for our purposes and implies that $\cS\SP$ contains no zero divisors by \cite[Corollary 6.16]{Mu2012}. Our proof broadly follows the same strategy as Muller with the following exception. Muller's proof shows that a $\cP$-arc $a$ is not a zero divisor by first choosing an arbitrary ordering on the set single-component $\cP$-tangles to define a notion of an ``initial'' $\cP$-tangle of an element $x \in \cS$. Then, he uses this data to define a map $\gamma_a$ of the set of essential $\cP$-tangles induced by multiplication by $a$, which is shown to be injective.

We prove that a given $\cP$-arc $a$ is not a zero divisor by defining a similar injective map $\gamma_a$ on the set of all essential $\cP$-tangles, but based on a natural notion of initial $\cP$-tangle based on intersection number with $a$ and degree in $q$. While our method ends up involving greater technical difficulty than the method of Muller, we avoid making arbitrary choices and thus our argument reveals more of the geometric intuition behind why $\cS\SP$ is a domain.

Throughout this appendix we fix a marked surface $\SP$ and we write $\cS:=\cS\SP$.

\section{Injective map on essential \texorpdfstring{$\cP$}{P}-tangles}

In this section we construct a map $\inqa: \cT^e\SP \to \cT^e\SP$, where $\cT^e\SP$ is the free $\BZ[q^{\pm 1/2}]$-module of essential $\cP$-tangles considered up to $\cP$-isotopy.

\begin{proposition}\lbl{z.inj}
Let $a$ be an essential $\cP$-arc in $\SP$. Then the map

\[
T \mapsto \inqa(aT),
\]
defined below, is an injection from the set of essential $\cP$-tangles to itself.
\end{proposition}

By \cite[Lemma 4.1]{Mu2012}, essential $\cP$-tangles form an $R$-basis of $\cS$. Thus any element $x \in \cS$ can be uniquely expressed as
\be\lbl{e.inqa1}
x = \sum_{c \in \frac{1}{2}\BZ} q^c x_c,
\ee
where each $x_c$ is a finite $\BZ$-linear combination of essential $\cP$-tangles and all but finitely many $x_c$ are nonzero. In order to show that a given essential $\cP$-arc $a \in \cS$ is not a zero divisor we define a notion of ``highest degree term of $x$ relative to $q$ and $a$''. Define the {\em initial} formal $\BZ$-linear combination of essential $\cP$-tangles of $x$ relative to $q,a$, written $\inqa(x)$, as follows. \no{Let $x_n$ denote the $\BZ$-linear combination of essential $\cP$-tangles in~\eqref{e.inqa1} with the highest power of $q$ in front of it.} Let $n \in \frac{1}{2} \BZ$ be the highest power of $q$ appearing on the right hand side of~\eqref{e.inqa1}, and $x_n$ the corresponding $\BZ$-linear combination of essential $\cP$-tangles. Then we may write in a unique fashion a finite sum
\[
x_n = \sum_{T_j \in \supp(x_n)} c_j T_j, \ \ \ \ c_j \in \BZ.
\]
Let $M = \max_{T_j \in \supp(x_n)}\{\mu(T_j,a)\}$. Then we define
\[
\inqa(x) := \sum_{\substack{T_j \in \supp(x_n) \\ \mu(T_j,x)=M}} c_j T_j.
\]
As in \cite{Mu2012}, it will be shown that $\inqa(x)$ is equivalent to choosing the `positive smoothing' of every crossing in $a \cdot x$, where $a \cdot x$ denotes the superposition of $a$ above $x$. To be more precise, $a \cdot x$ denotes the element $\sum c_j (a \cup T_j)$ where $a$ has been $\cP$-isotoped to lie in $\Sigma \times (0,1)$ and $T_j$ has been $\cP$-isotoped to lie in $\Sigma \times (-1,0)$ such that their projection onto $\Sigma \times \{0\}$ intersects transversely. This will require several additional lemmas.

Let $a$ be an essential $\cP$-arc in $\SP$ and $T$ an essential $\cP$-tangle. We say that a given $\cP$-isotopy representative of $T$ is {\em $a$-taut} if $a$ and $T$ are transverse and have minimal geometric intersection index. Choose an $a$-taut $\cP$-isotopy representative of $T$. Denote by $a \cap T$ the set of crossings in the superposition $a \cdot T$ not including boundary intersections. For any function $\sigma: a \cap T \to \{\pm 1\}$, we denote by $R_\sigma \in \cT\SP$ the (not necessarily essential) $\cP$-tangle obtained by choosing a small disc neighborhood of each crossing and switching each crossing sent to $+1$ with a {\em positive smoothing} and each crossing sent to $-1$ with a {\em negative smoothing} as in this local picture.

\FIGc{figures/smoothing}{Positive smoothing on the left, negative smoothing on the right.}{2cm}

We also define $\|\sigma\| := \sum_{p \in a \cap T} \sigma(p) = |\sigma^{-1}(+1)|-|\sigma^{-1}(-1)|$. Then we may write
\[
aT = q^c \sum_{\sigma:a \cap T \to \{\pm 1\}} q^{\|\sigma\|}R_\sigma,
\]
with $c \in \frac{1}{2}\BZ$ the exponent given by~\eqref{e.simul}.

\begin{lemma}\lbl{loopinequalitylemma}
Choose a function $\sigma: a \cap T \to {\pm 1}$. Then for $R_\sigma$ we have the inequality
\be
\lbl{loopinequality}
2|\{\text{contractible $\cP$-knots}\}|+|\{\text{contractible $\cP$-arcs}\}| \leq |\sigma^{-1}(-1)|.
\ee
\end{lemma}
\begin{proof}
Our proof follows that of \cite[Lemma C.1]{Mu2012}. Let $\Xi$ be a closed tubular neighborhood of $a$ small enough that $T$ intersects $\Xi$ a minimal number of times and large enough to contain the chosen disc neighborhoods of each crossing in $a \cdot T$.

If $R_\sigma$ has no contractible components then~\eqref{loopinequality} is trivially true. Consider a contractible component $z$ in $R_\sigma$. Let $\Om \subseteq \Sigma$ be the closed disc in $\Sigma$ with $\partial \Om = z$. To help us count how many negative smoothings $z$ passes through, we construct a graph $G$. The vertices of $G$ are given by connected components of $\Om \setminus \partial \Xi$. Connect two vertices by an edge if the corresponding components of $\Om \setminus \Xi$ share a common boundary in $\Om \cap \partial \Xi$. Then $G$ is a retract of the disc $\Om$ and therefore is a tree.

Being a tree, $G$ must have at least two degree 1 vertices (leaves). We will show that in the case that $z$ is a contractible arc, at least one of these leaves will correspond to $z$ passing through a negative smoothing, and if $z$ is a contractible loop then all leaves will correspond to $z$ passing through a negative smoothing. Let $\Om_0$ be a connected component of $\Om \setminus \partial \Xi$ corresponding to a leaf. We consider the three possible cases:

\begin{itemize}
\item $\partial \Om_0$ contains a marked point. Then $z$ ends on this marked point and is therefore a contractible arc. It follows that $G$ has at most one leaf of this type.
\item $\partial \Om_0 \subset \Om \cap \Xi$ and $\partial \Om_0$ contains no marked point. $\partial \Om_0$ must contain a connected component of $a \setminus T$ with endpoints at crossings in $x \cap \alpha$. In $R_\sigma$ one of these crossings must be positive, and one must be negative. If both were positive or negative then $\Om_0$ could cross over $x$ and thus have two or more boundary components in $D \cap \partial \Xi$, a contradiction. Thus each leaf of this type implies a negative smoothing in $R_\sigma$.
\item $\partial \Om_0 \subset \Om \setminus \Xi$ and $\partial \Om_0$ contains no marked point. Then the segment of the boundary of $\partial \Om_0$ corresponding to $z$ can be deformed to the interior of $T$, reducing the crossing number of $\Xi$ and $T$. This contradicts our assumption and so there can not be any leaves of this type.
\end{itemize}

Thus a contractible $\cP$-arc $z$ has $G$ with one leaf of the first type and at least one leaf of the second type, implying that $z$ passes through at least one negative smoothing. A contractible $\cP$-knot $z$ has $G$ with all leaves of the second type, implying that $z$ passes through at least two negative smoothings.
\end{proof}

\begin{remark}
The proof of Lemma \ref{loopinequalitylemma} also works to show that
\be
\lbl{posloopinequality}
2|\{\text{contractible $\cP$-knots}\}|+|\{\text{contractible $\cP$-arcs}\}| \leq |\sigma^{-1}(+1)|.
\ee
\end{remark}

After smoothing, $R_\sigma$ may have contractible $\cP$-knots and $\cP$-arcs. Let $K_\sigma$ denote the number of contractible loops in $R_\sigma$. Write $\widehat{R}_\sigma$ for the $\cP$-tangle obtained by deleting contractible components. Let $\delta_\sigma = 0$ if $\widehat{R}_\sigma$ contains a contractible $\cP$-arc, and $\delta_\sigma = 1$ otherwise. Then via the skein relations, we have
\[
R_\sigma = (-q^2-q^{-2})^{K_\sigma}\delta_\sigma\widehat{R}_\sigma.
\]
By the binomial formula, we have
\[
aT = q^c \sum_{\sigma: a \cap T}(-1)^{K_\sigma} \delta_\sigma \widehat{R}_\sigma \sum_{i=0}^{K_\sigma} \binom{K_\sigma}{i} q^{\|\sigma\| + 2K_\sigma - 4i}
\]
Write $k=|a \cap T| = |\sigma^{-1}(\{\pm 1\})|$, so that $\|\sigma\| = k - \sigma^{-1}(-1)$. Then
\[
aT = q^{c + k} \sum_{\sigma: a \cap T}(-1)^{K_\sigma} \delta_\sigma \widehat{R}_\sigma \sum_{i=0}^{K_\sigma} \binom{K_\sigma}{i} q^{2(K_\sigma - |\sigma^{-1}(-1)| - 2i)}.
\]
Inequality~\eqref{loopinequality} implies the weaker inequality $K_\sigma - |\sigma^{-1}(-1)| \leq 0$ with equality only when $R_\sigma$ has an equal number of contractible loops and negative smoothings. Therefore the highest $q$-degree appearing in $aT$ is $c + k$, and the only $\widehat{R}_\sigma$'s appearing with this $q$-degree are those with an equal number of contractible $\cP$-knots and negative smoothings. Let $\sigma_\pm$ be the function sending every crossing to $\pm 1$. Then $\widehat{R}_{\sigma_+}$ has degree $q^{c + k}$ since~\eqref{loopinequality} implies $\widehat{R}_{\sigma_+}$ has no contractible loops. The $q$-degree of $\widehat{R}_{\sigma_-}$ is necessarily smaller than $c + k$ since $\widehat{R}_{\sigma_-}$ has no contractible loops by~\eqref{posloopinequality}.

Now we must show that $\mu(a,\widehat{R}_{\sigma_+}) > \mu(a,\widehat{R}_\sigma)$ for all $\sigma \neq \sigma_\pm$. The following lemma will be useful.

\begin{lemma}
Write $T \setminus a = \cup_i T_i$, and $a \setminus T = \cup_j a_j$ where the $T_i$ and $a_j$ are the connected components of $T \setminus a$ and $a \setminus T$, respectively. Then $T$ is $a$-taut if and only if no $T_i$ cobounds a disc with $a$.
\end{lemma}
\begin{proof}
We say that {\em $T_i$ cobounds a disc with $a$} if $\overline{T_i \cup a_j}$ (the closure of $T_i \cup a_j$) bounds a disc for some $a_j$.

Assume that $T$ is not $a$-taut. Then $T$ is $\cP$-isotopic to some $T^\prime$ which is $a$-taut. $T$ and $T^\prime$ are then related by a series of Reidemeister moves. The first Reidemeister move is not allowed as it requires self-intersection, which is not permitted. The third Reidemeister move is also forbidden for the same reason. At least two of the strands must belong to either $a$ or $T$, and therefore $a$ or $T$ must be self-intersecting.

\FIGc{figures/R1andR3}{Forbidden Reidemeister moves.}{2cm}

Therefore the isotopy from $T$ to $T^\prime$ is given by a sequence of second Reidemeister moves, each one of which removes a disc cobounded by $a$ and $T$.

\FIGc{figures/R2}{Permitted Reidemeister move.}{2cm}

Thus we have shown that if $T$ is not $a$-taut then it has some $T_i$ that cobounds a disc with $a$, equivalently, that if no $T_i$ cobounds a disc with $a$ then $T$ is $a$-taut.

Similarly, assume that $T$ is $a$-taut and that $T$ cobounds a disc with $a$. Then applying the second Reidemeister move in a neighborhood of the cobounded disc reduces $\mu(T,a)$ by 2, which is a contradiction.
\end{proof}

We now define an injective map $\gamma_a$ of essential $\cP$-tangles $T$ that explicitly constructs the $\cP$-tangle obtained by replacing all crossings of $a \cdot T$ with positive smoothings. It will turn out that $\gamma_a(T) = \inqa(aT)$. We will construct this map directly in terms of local pictures of $\Sigma$ to make the inverse easy to construct. This is the same map given in Appendix C of \cite{Mu2012}.

Consider a tubular neighborhood $\Xi$ of $a$. For a given essential $\cP$-tangle $T$, choose an $a$-taut $\cP$-isotopy representative of $T$. We construct $\gamma_a(T)$ inside of $\Xi$. Cut each crossing in $a \cdot T$ and reconnect the strands by shifting to the right along $a$. Spare ends on either side are attached to the endpoints of $a$. $a$ may have distinct or identical endpoints; each case is illustrated in Figure \ref{fig:figures/gammax}

\FIGc{figures/gammax}{Explicit construction of $\gamma_a(T)$.}{6cm}

We construct an inverse to $\gamma_a$. Let $T$ be an $a$-taut essential $\cP$-tangle. Define an essential $\cP$-tangle $\nu_a(T)$ as follows.

\begin{itemize}
\item If $a$ has one end at a marked point $p$, and there are no strands of $T$ counterclockwise to $a$ at $p$, then $\nu_a(T)$ is the empty $\cP$-tangle $\emptyset$.
\item If $a$ has both ends at a marked point $p$, and there are fewer than two strands of $T$ counter-clockwise to $a$, then $\nu_a(T)$ is the empty $\cP$-tangle $\emptyset$.
\item Otherwise, construct $\nu_a(T)$ as follows. Cut $T$ along $a$, and at each end of $a$, disconnect the first strand of $T$ counter-clockwise to $a$. Reconnect these ends by shifting to the left along $a$.
\end{itemize}

Then by construction $\nu_a(\gamma_a(T))=T$, and so $\gamma_a$ is injective.

\begin{lemma}\lbl{positivetaut}
Let $a$ be an essential $\cP$-arc and $T$ be an $a$-taut essential $\cP$-tangle. Then $T_+ = \gamma_a(T)$ is also $a$-taut.
\end{lemma}

\begin{proof}
Let $T$ be $a$-taut and write $T \setminus a = \cup_i T_i$ be the connected components of $T \setminus a$. Let $\Xi$ be a {\em banana neighborhood} of $a$ with respect to $T$. A banana neighborhood is defined as follows. Let $a_-,a_+$ be two $\cP$-arcs $\cP$-isotopic to $a$, one on either side of $a$, intersecting only at the endpoint(s) of $a$. We choose $a_-$ and $a_+$ so that $T$ is $a_-$-taut and $a_+$-taut. Thus to minimize intersection with $T$ we must choose $a_-,a_+$ so that the intersection at the endpoint(s) has $a$ sandwiched between $a_-$ and $a_+$ with no strands of $T$ between them, again by the Reidemeister move argument given above. Then $\Xi$ is defined to be the disc bounded between $a_-$ and $a_+$. One advantage of banana neighborhoods is that we can ignore the behavior of $T$ near the marked points of $a$.

\FIGc{figures/banananbhd1}{Banana neighborhood of $a$, when $a$ has two or one distinct endpoints. $\Xi$ is presented in checkerboard pattern. A solid line denotes the boundary of $\Sigma$, and a dotted line denotes the boundary of the neighborhood depicted.}{3cm}
Denote by $\Xi_-$ the subset of $\Xi$ contained between $a_-$ and $a$, similarly denote $\Xi_+$ for the subset of $\Xi$ contained between $a_+$ and $a$. Let $T_i$ be a component of $T \setminus a$, and represent it by an embedding $T_i: (0,1) \to \Sigma$. For a neighborhood $D$, we say that $T_i$ {\em ends in} $D$ if there exists some $\ve > 0$ such that $T_i((0,\ve)) \subseteq D$ or $T_i((1-\ve,1)) \subseteq D$. We also say that $T_i$ {\em ends at} a crossing $c \in a \cap T$ or marked point $p \in \cP$ if $c \in \overline{T_i}$.

Let $\overset\circ{T}$ be $T$ minus the endpoints. Then $T$ induces a map of sets:
\[
w_{\pm}: \{\overset\circ{T} \cap a \} \stackrel{\cong}{\to} \{\overset\circ{T} \cap a_{\pm}\},
\]
given by sending each $p \in \overset\circ{T} \cap a$ to the unique crossing in $\overset\circ{T} \cap a_{\pm}$ connected to $p$ by a component of $T \setminus \{a,a_-,a_+\}$. Let $T_+ = \gamma_a(T)$ be the $\cP$-link obtained by replacing every crossing in $a \cdot T$ with a positive smoothing. %

We claim that $T_+$ is $a$-taut. Write $T_+ \setminus a = \cup_i T_{+,i}$ where the $T_{+,i}$ are the connected components of $T_+ \setminus a$. Suppose a component $T_{+,i}$ cobounds a disc $\Om$ with $a$. $\Om$ may contain other components $T_{+,j}$ which also cobound a disc, pick an $T_{+,j}$ which cobounds a disc that does not contain any other components of $T_+ \setminus a$. Call this component $\beta$ and the disc it bounds $\Om_0$.

Assume that $\beta$ ends on a marked point $p$ (which is necessarily a marked point that $a$ ends on). We can assume without loss of generality that the other end of $\beta$ is in $\Xi_+$. We call $q$ the point at which $\beta$ intersects $a_+$ other than the marked point. Let $\beta'$ be the segment of $\beta \setminus a_+$ with endpoints $p$ and $q$. Let $T_i$ be the segment of $T \setminus \{a \cup a_{\pm}\}$ with endpoints $q$ and $w_+^{-1}(q)$. Then because $T_+$ was constructed entirely within $\Xi$ and $\beta' \cap \Xi = \emptyset$ by the properties of banana neighborhoods, we have that $\beta' \subset T$ and so $\overline{\beta'} \cup \overline{T_i}$ is a segment of $T \setminus a$ which cobounds a disc with $a$. This is a contradiction, therefore $\beta$ cannot end on a point marked.

We claim that the endpoints $p_1,p_2$ of $\beta$ are adjacent crossings in $a \cap T_+$. Assume otherwise. Then there is a component $T_{+,i}$ that ends in $\Om_0$ and ends at a crossing between $p_1$ and $p_2$. By assumption, $T_{+,i}$ does not cobound a disc inside of $\Om_0$, and therefore also ends at a crossing or marked point outside of $\Om_0$. Thus $T_{+,i}$ must intersect $\beta$ at some point, which is forbidden. This is a contradiction, therefore $p_1$ and $p_2$ are adjacent crossings.

We claim that both ends of $\beta$ are either in $\Xi_-$ or in $\Xi_+$. Assume $\beta$ ends in both $\Xi_-$ and $\Xi_+$. Then $a \setminus \overline{\beta}$ has 3 segments $a_1,a_2,a_3$, with the middle segment $a_2$ being the one between $c_1$ and $c_2$. Since $a$ is transverse to $T$, one of the two remaining segments must be contained in $\Om_0$ and end on the boundary of $\Om_0$, which is impossible since $\partial \Om_0 = \overline{\beta} \cup \overline{a_2}$ and $\beta$ can be chosen to avoid marked points.

Without loss of generality, we can say that both ends of $\beta$ are in $\Xi_+$. Then $a_+ \setminus \beta$ has three components $\beta_1,\beta_2,\beta_3$, with the middle segment $\beta_2$ satisfying $\Xi_+ \cap \beta_2 = \emptyset$. Then since $\beta_2 \cap \Xi=\emptyset$, $\beta_2 \subset T$. Let $\{p_1,p_2\}\subset a_+$ be the endpoints of $\beta_2$, and let $T_1,T_2$ be the connected components of $T \setminus \{a\cup a_{\pm}\}$ whose endpoints are $p_1,w_+(p_1)$ and $p_2,w_+(p_2)$. Then $\overline{T_1 \cup T_2 \cup \beta_2} \subset T$ cobounds a disc with $a$, and therefore $T$ is not $a$-taut, a contradiction.
\end{proof}

\begin{lemma}
$\mu(a,\widehat{R}_{\sigma_+}) > \mu(a,\widehat{R}_\sigma)$ for all $\sigma \neq \sigma_\pm$. Therefore $\inqa(aT) = \widehat{R}_{\sigma_+} = \gamma_a(T)$.
\end{lemma}

\begin{proof}
First we claim that $\mu(a,\gamma_a(T)) = k-1$, where $k:=\mu(a,T)$. We prove this using local pictures, which will be sufficient as Lemma \ref{positivetaut} implies that $\gamma_a(T)$ is $a$-taut since $T$ is $a$-taut. Choose a banana neighborhood $B$ of $a$ with respect to $\gamma_a(T)$. There are two cases to consider, case (1) being where $a$ has different endpoints and case (2) being where both endpoints of $a$ are the same.

We number the crossings in $a \cdot T$ as $1,2,\ldots,k$ going from top to bottom for case (1) and going counterclockwise from the marked point for case (2). We consider the superposition $a \cdot \gamma_a(T)$ to compute the crossing number $\mu(a,\gamma_a(T))$. At crossings $1$ and $k$, we know that $a$ does not intersect $\gamma_a(T)$ since banana neighborhoods avoid any components of $\gamma_a(T) \setminus a$ ending on the marked point(s) of $a$.

For each of the $k-1$ pairs of adjacent crossings, the local picture is as shown in Figure \ref{fig:figures/crossing} for both cases.

\FIGc{figures/crossing}{$a$ intersects $\gamma_a(T)$ between each pair of adjacent crossings.}{3cm}

Thus $\mu(a,\gamma_a(T)) = k-1$. Next we show that $\gamma_a(T)$ has the highest crossing number with $a$ of all $\widehat{R}_\sigma$ other than $\sigma_-$. Let $\sigma \neq \sigma_\pm$. We claim that $\mu(a,\widehat{R}_{\sigma}) \leq k-2$. Due to $\cP$-isotopy invariance, we can build $\widehat{R}_{\sigma}$ from the picture $a \cdot T$ by repeatedly smoothing a crossing and then applying a $\cP$-isotopy within $B$ to guarantee minimal crossing number until we have obtained a picture as prescribed by $\sigma$. Number the crossings in $a \cdot T$ as $1,2,\ldots,k$ as above. Choose $i,j$ such that $\sigma(i) = +1$ and $\sigma(j) = -1$. After smoothing at crossings $i$ and $j$, we perform a $\cP$-isotopy to pull away a strand. Which direction we pull the strand is in determined by whether $i>j$ or $i<j$. Case (1) is depicted in Figure \ref{fig:figures/pullaway1} and case (2) is depicted in Figure \ref{fig:figures/pullaway2}.

\FIGc{figures/pullaway1}{$i>j$ on the left and $i<j$ on the right.}{3cm}
\FIGc{figures/pullaway2}{$i>j$ on the left and $i<j$ on the right.}{3cm}

Now the local picture of $a \cdot T$ after applying the above process has $a$ intersecting $k-2$ strands of $T$, and therefore $\mu(a,\widehat{R}_{\sigma}) \leq k-2$.
\end{proof}

\begin{remark}
This completes the proof of Proposition \ref{z.inj}.
\end{remark}

\section{\texorpdfstring{$\cS$}{S} is a domain}

For this section we set the base ring $R$ to be a commutative domain (not necessarily Noetherian) with distinguished invertible element $q^{1/2}$.

\begin{lemma}\lbl{r.arczero}
Let $a$ be an essential $\cP$-arc. Then $a$ is not a zero divisor in $\cS$.
\end{lemma}

\begin{proof}
Consider $x \in \cS$ such that $ax=0$. Extend $\inqa$ to an $R$-linear map of $\cS$. We may write $x$ uniquely as a finite sum
\[
x = \sum_{T_j \in \supp(x)} \lambda_j T_j
\]
with $0 \neq \lambda_j \in R$ and $T_j$ essential $\cP$-tangles. Then
\[
0 = ax = \sum_{T_j \in \supp(x)} \lambda_j aT_j = \sum_{T_j \in \supp(x)} \lambda_j (q^{n_j}\gamma_a(T_j) + \text{lower order terms in }q)
\]

Let $M = \max_{T_j \in \supp(x)}\{\deg_q(\lambda_j)+n_j\}$, where $\deg_q(\lambda_j)$ is the exponent of the highest degree $q$ factor in $\lambda_j$, and $N = \max_{T_j \in \supp(x)}\{\mu(a,T_j)\}$. Then because $\inqa$ drops all terms not of maximal $q$-degree and $a$-crossing number, we may ignore such terms in the summation and have that
\[
0 = \inqa(az) = \sum_{\substack{T_j \in \supp(x) \\ \deg_q(\lambda_j+n_j)=M \\ \mu(a,T_j)=N}} \inqa(\lambda_j) \gamma_a(T_j)
\]
Then because $\gamma_a$ is an injection and essential $\cP$-tangles form an $R$-basis, the elements $[\gamma_a(T_j)]$ are linearly independent over $R$. Since $\inqa(\lambda_j)$ must be nonzero, we must have that $\supp(x)$ is empty, and therefore $x=0$. Thus $a$ is not a left zero divisor. In addition, since $a$ is reflection invariant, $a$ is also not a right zero divisor.
\end{proof}

\begin{corollary}
Let $\SP$ be a marked surface. Then $\cS\SP$ is a domain.
\end{corollary}
\begin{proof}
Lemma \ref{r.arczero} implies that $\cS$ has no zero divisors by \cite[Corollary 6.11]{Mu2012}.
\end{proof}

\chapter{Chebyshev-Frobenius homomorphism for the marked disk}\lbl{app.simple}

\no{\section{Marked disk}}
\no{\red{From 11/2016 draft}}

In this Appendix, we prove directly that the Chebyshev-Frobenius homomorphism $\Phi_\xi$ as given in Chapter \ref{c.cf} is well-defined on the disk with four marked points, denoted $(\Om,\cP)$. This is done by investigating how the skein relations apply to intersecting $\cP$-arcs raised to the $N$th power, resulting in $2^{N^2}$ terms, most of which then either equal zero due to the trivial arc relation or cancel when the quantum parameter $\xi^4$ is a root of unity of order $N$.

This exercise reveals intuition on how the Chebyshev-Frobenius homomorphism behaves for $\cP$-arcs at the level of the skein algebra without involving the Muller algebra technique we used to prove the existence of the homomorphism. Throughout this section we write $\cS_q :=\cS_q(\Om,\cP)$. Our base ring is either $\BZ[q^{\pm 1/2}]$ when dealing with an indeterminate $q$, or $\BC$ when we set $q$ to be a root of unity $\xi$. 

For $\xi$ a primitive root of unity, $N:=\ord(\xi^4)$, and $\ve:=\xi^{N^2}$, we prove the following equality in $\cS_\ve$.

\FIGc{figures/case4eqn1}{The skein relation that must be satisfied.}{2.5cm}

We label the single component $\cP$-arcs in $(\Om,\cP)$ as in Figure \ref{fig:figures/markeddiskarcs}.

\FIGc{figures/markeddiskarcs}{Labels of the $\cP$-arcs in $(\Om,\cP)$.}{2.5cm}

\begin{lemma}\lbl{l.case41}
Let the base ring be $R=\BZ[q^{1/2}]$ and let $N$ be a positive integer. Then in $\cS_q$ we have that

\[
x^Ny^N=q^{-N^2}a^Nc^N \sum_{j=0}^N \binom{N}{j}_{q^4} q^{2j^2}(bda^{-1}c^{-1})^j.
\]
\end{lemma}

\begin{proof}
The set of crossings in the product $x^Ny^N$ in $\cS_q$ naturally form an $N \times N$ grid. We label each crossing with the index $(i,j)$, $i,j=1,\ldots,N$, denoting row and column. Thus, using skein relations, $x^Ny^N$ resolves as a sum of $2^{N^2}$ terms, one for each way of filling the $N \times N$ grid with $+$'s and $-$'s to represent positive or negative crossings. This is known as the state sum picture of the Kauffman bracket skein relations \cite{Ka1987}. However, many of these terms will turn out to be zero by the trivial arc relation. Consider the set of all resolutions of crossings which have a $+$ in the top left corner, as shown the example of $x^4y^4$ in Figure \ref{fig:figures/case4crossings}.

\FIGc{figures/case4crossings}{$x^4y^4$ with the $(1,1)$ crossing resolved to be positive.}{3.5cm}

\FIGc{figures/case4vanish}{$x^4y^4$ with the $(1,1)$ crossing resolve to be positive and the $(2,1)$ crossing resolved to be negative.}{3.5cm}

Then, if crossing $(2,1)$ is made to be negative, the strand marked with $\star$ in Figure \ref{fig:figures/case4vanish} will be set to 0 and therefore any term with a $+$ crossing in the $(1,1)$ position and a $-$ crossing in the $(2,1)$ position will vanish. Thus, if the $(1,1)$ crossing is $+$, the $(2,1)$ crossing must be $+$ as well in order for the associated term to not vanish. Continuing inductively, this means that the entire first column of the grid must consist of $+$ crossings in order for the term to not vanish.
\[
\begin{array}{|c|c|c|c|} \hline
+ & \phantom{+} & \phantom{+} & \phantom{+} \\ \hline
\phantom{+} &  &  &  \\ \hline
\phantom{+} &  &  & \\ \hline
\phantom{+} &  &  &  \\ \hline \end{array} \ \ \text{ implies } \ \ 
\begin{array}{|c|c|c|c|} \hline
+ & \phantom{+} & \phantom{+} & \phantom{+} \\ \hline
+ &  &  &  \\ \hline
+ &  &  & \\ \hline
+ &  &  &  \\ \hline \end{array}.
\]
Similarly, looking to the right of the $(1,1)$ crossing, if the $(1,2)$ crossing is negative then a trivial arc will be obtained and the entire term will vanish. So by the same logic as the first column, the entire first row of crossings must be positive as well.
\[
\begin{array}{|c|c|c|c|} \hline
+ & \phantom{+} & \phantom{+} & \phantom{+} \\ \hline
+ &  &  &  \\ \hline
+ &  &  & \\ \hline
+ &  &  &  \\ \hline \end{array} \ \ \text{ implies} \ \ 
\begin{array}{|c|c|c|c|} \hline
+ & + & + & + \\ \hline
+ &  &  &  \\ \hline
+ &  &  & \\ \hline
+ &  &  &  \\ \hline \end{array}.
\]
Thus, the only nonvanishing term that has a $+$ crossing in the $(1,1)$ position is given by $q^{2N-1}bdx^{N-1}y^{N-1}$ (the other terms obtained by resolving the remaining crossings are ``contained'' in $x^{N-1}y^{N-1}$). Following similar logic as for the $+$ crossing, if the $(1,1)$ crossing is $-$, then the crossings to the right of the $(1,1)$ crossing are also forced to be $-$, and then the crossings in the $N$th column are forced to be $-$ as well in order for the term to not vanish.
\[
\begin{array}{|c|c|c|c|} \hline
- & \phantom{-} & \phantom{-} & \phantom{-} \\ \hline
\phantom{-} &  &  &  \\ \hline
\phantom{-} &  &  & \\ \hline
\phantom{-} &  &  &  \\ \hline \end{array} \ \ \text{ implies } \ \
\begin{array}{|c|c|c|c|} \hline
- & - & - & - \\ \hline
\phantom{-} &  &  &  \\ \hline
\phantom{-} &  &  & \\ \hline
\phantom{-} &  &  &  \\ \hline \end{array} \ \ \text{ implies } \ \ 
\begin{array}{|c|c|c|c|} \hline
- & - & - & - \\ \hline
 &  &  & - \\ \hline
 &  &  & - \\ \hline
 &  &  & - \\ \hline \end{array}.
\]
Thus we have that
\[
x^Ny^N = q^{2N-1}bdx^{N-1}y^{N-1} + q^{1-2N}acx^{N-1}y^{N-1}.
\]
Proceeding inductively,
\[
x^Ny^N = \prod_{j=0}^{N-1}(q^{2N-1-2j}bd+q^{1-2N+2j}ac).
\]
Next we note that both $b$ and $d$ commute with the product $ac$.\no{Furthermore, $[bd]=bd$ and $[ac]=ac$, } Working in the division algebra $\tcS_q$, we can write
\begin{align*}
\prod_{j=0}^{N-1}(q^{2N-1-2j}bd+q^{1-2N+2j}ac) & = \prod_{j=0}^{N-1}q^{1-2N+2j}ac(1+q^{2(2N-1-2j)}bda^{-1}c^{-1}) \\
& = \left( \prod_{j=0}^{N-1}q^{1-2N+2j} \right) a^Nc^N \times \\
& \left( \prod_{j=0}^{N-1}(1+q^{2(2N-1-2j)}bda^{-1}c^{-1}) \right) \\
& = q^{-N^2}a^Nc^N \prod_{j=0}^{N-1}(1+q^{2(2N-1-2j)}bda^{-1}c^{-1}).
\end{align*}
For a variable $u$, we write $(1+u)_q^n$ to represent the product $(1+u)(1+qu) \cdots (1+q^{n-1}u) = \prod_{j=0}^{n-1}(1+q^ju)$. It is a formula of Gauss that tells us
\[
(1+u)_q^n = \sum_{j=0}^n \qbinom nj_q q^{\frac{j(j-1)}{2}}u^j.
\]
Consider $u=bda^{-1}c^{-1}$. Then we have that
\[
(1+q^2u)_{q^4}^N = \prod_{j=0}^{N-1}(1+q^{2+4j}u) = \prod_{j=0}^{N-1}(1+q^{2(2N-1-2j)}u).
\]
Then
\begin{align*}
x^Ny^N = q^{-N^2}a^Nc^N(1+q^2u)_{q^4}^N & = q^{-N^2}a^Nc^N \sum_{j=0}^N \qbinom Nj_{q^4} q^{2j(j-1)}q^{2j}u^j \\
& = q^{-N^2}a^Nc^N \sum_{j=0}^N \qbinom Nj_{q^4} q^{2j^2}u^j.
\end{align*}
\end{proof}

\begin{corollary}
If we set $q=\xi$ then $\qbinom Nj_{\xi^4}= 0$ for all $1 \leq j \leq N-1$ since $\xi^4$ is a primitive root of unity of order $N$ by Corollary \ref{c.32}. Thus, in $\cS_\xi$, we have that
\[
x^Ny^N = \ve^{-1}a^Nc^N + \ve^{-1}a^Nc^N\ve^2(bda^{-1}c^{-1})^N = \ve^{-1}a^Nc^N+\ve b^Nd^N,
\]
which is the expression shown in Figure \ref{fig:figures/case4eqn1}. Therefore the Chebyshev-Frobenius homomorphism $\Phi_\xi: \cS_\ve \to \cS_\xi$ satisfies the skein relation and so is well-defined.
\end{corollary}

\end{appendices}

%% file: vita.tex
\chapter*{Vita}
\addcontentsline{toc}{chapter}{Vita}  %
Jonathan Paprocki was born in March 1989 in Lapeer, MI, and is known by his friends as ``Poprox''. He graduated with Highest Honors from the Georgia Institute of Technology in 2011 with a B.S. in Applied Mathematics and a B.S. in Physics, then with a Ph.D. in Mathematics in 2019.

His academic interests include mathematical physics, quantum topology, quantum computing, distributed systems, biomimicry, regenerative agriculture, categorical informatics, cybernetics, complex systems, and philosophy of physics and mathematics. The two scientists whose work most inspires him are mathematician Louis Kauffman and physicist John Wheeler. In his free time he enjoys artistic pursuits, especially in the context of the regional Burning Man community with the open-ended Survey Flags project.